%% file: 000-MainFile.tex
\documentclass[11pt]{article}

\usepackage{amsthm}
\usepackage{amsmath}
\usepackage{mathrsfs}
\usepackage{mathtools}
\usepackage{amsmath}
\usepackage{amsfonts}
\usepackage{graphics}
\usepackage{listings}
\usepackage{booktabs}
\usepackage{graphicx}
\usepackage{amsfonts}
\usepackage{amssymb}
\usepackage{setspace}
\usepackage{listings}
\usepackage{longtable}
\usepackage{color}
\usepackage{booktabs}
\usepackage{xspace}
\usepackage{setspace}
\usepackage{scrextend}
\usepackage{subcaption}
\usepackage[numbers]{natbib}
\usepackage{float}
\usepackage{caption}
\usepackage{graphicx}
\usepackage{tikz}
\usetikzlibrary{arrows, arrows.meta}
\usepackage{natbib}
\usepackage{hyperref}
\usepackage{todonotes}
\usepackage[utf8x]{inputenc}
\usepackage[T1]{fontenc}
\usepackage{textcomp}
\usepackage{textgreek}
\usepackage{enumitem}
\input{001-Preambule}

\title{\bf An algorithmic guide for finite-dimensional\\optimal control problems}
\author{Jean-Baptiste Caillau\thanks{Universit\'e C\^ote d'Azur, CNRS, Inria, LJAD (\texttt{jean-baptiste.caillau@univ-cotedazur.fr})}, Roberto Ferretti\thanks{Dipartimento di Matematica e Fisica, Universit\`a Roma Tre, Roma (\texttt{ferretti@mat.uniroma3.it})},
Emmanuel Tr\'elat\thanks{Sorbonne Universit\'e, CNRS, Universit\'e de Paris, Inria, Laboratoire Jacques-Louis Lions (LJLL), F-75005 Paris, France (\texttt{emmanuel.trelat@sorbonne-universite.fr}).},
Hasnaa Zidani\thanks{Insa Rouen Normandie, LMI (\texttt{hasnaa.zidani@insa-rouen.fr})}}

\date{}
\begin{document}

\maketitle
\begin{abstract}
\noindent We survey the main numerical techniques for finite-dimensional nonlinear optimal control. 
The chapter is written as a guide to practitioners who wish to get rapidly acquainted with the main numerical methods used to efficiently solve an optimal control problem. 
We consider two classical examples, simple but significant enough to be enriched and generalized to other settings:  Zermelo and Goddard problems. We provide sample of the codes used to solve them and make these codes available online.
We discuss direct and indirect methods, Hamilton--Jacobi approach, ending with optimistic planning.
The examples illustrate the pros and cons of each method,
and we show how these approaches can be combined into powerful tools for the numerical solution of optimal control problems for ordinary differential equations.
\end{abstract}


%
%%%%%%%%%%%%%%%%%%%%%%%%%%%%%%%%%%%%%%%%%%%%%%%%%%%%%%%%%%%%%%%%%%%%%
%%
\setcounter{tocdepth}{1}
\tableofcontents
%%
%=====================================================================
%\input{010-introduction.tex}
\input{010-Statement.tex}
\input{020-Direct_Methods}
\input{030-Shooting_Methods}
\input{040-HJ_approach}
\input{050-OP.tex}

%\appendix

%=====================================================================

%\bibliographystyle{plain}
%\bibliography{climbing}
%\pagebreak

%\bibliographystyle{unsrt}

\small
\bibliographystyle{abbrv}
\bibliography{bibmain,bibcomp,Biblio}
%\bibliography{bibmain,bibcomp,bib_pmp}
%\input{Main.bbl}
%\bibliography{Main}

%\pagebreak
% \appendix
% \input{Appendice_A}
% \input{Appendice_C}

\end{document}

%% file: 001-Preambule.tex
\usepackage{amsmath, amsfonts, amssymb, amsthm,mathtools,algorithm, algorithmic}
\usepackage[USenglish]{babel}
\usepackage{xspace,colortbl}
\usepackage{float}
\usepackage{bm}

\newcommand{\cmg}{{m_g}}
\newcommand{\cmf}{{m_f}}
   
\newcommand{\inte}[1]{{\overset{\circ}{#1}}}

\newcommand{\var}{\vartheta}
\newcommand{\vp}{\varphi}

\newcommand{\Equivalent}{\Longleftrightarrow}

\def\eps{\varepsilon}

\newcommand{\Z}{\mathbb{Z}}
\newcommand{\BB}{\mathbb B}
\newcommand{\N}{\mathbb{N}}
\newcommand{\R}{\mathbb{R}}
\newcommand{\XX}{\mathbb{X}}

\newcommand{\mt}{\theta}

\newcommand{\yxuk}{y_k^{x,\bu}}
\newcommand{\yxuN}{y_N^{x,\bu}}
\newcommand{\U}{\mathbb{U}}

\newcommand{\bx}{{\bf x}}
\newcommand{\by}{{\bf y}}
\newcommand{\bz}{{\bf z}}
\newcommand{\br}{{\bf r}}
\newcommand{\bbm}{{\bf m}}
\def\bq{{\boldsymbol q}}
\def\bu{{\boldsymbol u}}
\def\bv{{\boldsymbol v}}

\newcommand{\be}{\begin{eqnarray}}
\newcommand{\ee}{\end{eqnarray}}
\newcommand{\beno}{\begin{eqnarray*}\nonumber}
\newcommand{\eeno}{\nonumber\end{eqnarray*}}
\newcommand{\barr}[1]{\begin{array}{#1}}
\newcommand{\earr}{\end{array}}

\newcommand{\converge}{\rightarrow}

\newcommand{\Hyp}[1]{{\rm{(${\bf H_{#1}}$)}}}

\newcommand{\Hypu}[1]{\rm{($\bf \widetilde{H}_{1}$)}}
\newcommand{\cS}{\mathcal{S}}

\newcommand{\cT}{\mathcal{T}}
\newcommand{\cR}{\mathcal{R}}

\newcommand{\cC}{\mathcal{C}}

\newcommand{\cH}{\mathcal{H}}

\newcommand{\cU}{\mathcal{U}}
\newcommand{\cD}{\mathcal{D}}

\newcommand{\cK}{\mathcal{K}}
\newcommand{\cV}{\mathcal{V}}
\newcommand{\cW}{\mathcal{W}}

\newcommand{\VV}{\cV}
\newcommand{\WW}{\cW}

\definecolor{ggreen}{cmyk}{1,0,.8,0.5}

\newcommand{\comment}[1]{}

\newcommand{\bp}{{\boldsymbol p}}

\newcommand{\Co}{\text{co}} %- for the convex envelop 

\theoremstyle{plain}
\newtheorem{thm}{Theorem}[section]  
\newtheorem{definition}{Definition}[section]

\newtheorem{lem}{Lemma}[section]
\newtheorem{prop}{Proposition}[section]

\theoremstyle{definition}\newtheorem{rem}{Remark}[section]

\newcommand{\argmin}{\mbox{argmin}}

%% file: 010-Statement.tex
%%%%%%%%%%%%%%%%%%%%%%%%%%%%%%%%%%%%%%%%%%%%%%%%%%%%%%%%%%%%%%%%%%
\section{Introduction and statement of the problem}
%%%%%%%%%%%%%%%%%%%%%%%%%%%%%%%%%%%%%%%%%%%%%%%%%%%%%%%%%%%%%%%%%%
\subsection{Brief overview}

Optimal control theory has been widely developed since many decades. The theoretical and numerical achievements, motivated by a body of diverse applications in various domains, provide valuable  insights into the
nature of optimal controls and the corresponding optimal trajectories.
In this chapter our objective is to provide practitioners with a guide to the most powerful but however easy-to-use methods and algorithms to solve efficiently a given nonlinear optimal control problem in finite dimension.

The most intuitive and popular numerical methods for solving an optimal control problem, called \emph{direct methods}, consist of {\em first discretizing then optimizing}. Such approaches have been investigated in a number of contributions (see, e.g., \cite{Betts}).  From the theoretical point of view, the efficiency  of these methods has been established for some classes of control problems (see \cite{BonnansLaurentVarin,GRKF,Hager2000,RossFahroo,SanzSerna}). From the numerical point of view, the direct approach benefits from the tremendous advances  in numerical optimization methods achieved in the last decades. As shown in Section \ref{sec_direct}, the direct approach is extremely easy to implement for general control problems with constraints on both the control variable and the state variable. However, in  general, direct methods may provide only locally optimal solutions and may require a fair initialization of the iterative process in optimization algorithms. Also, they may lack numerical accuracy, and may become computationally demanding in high dimension.

A major breakthrough in optimal control theory were achieved in the 1950’s by Pontryagin’s research group,
who successfully generalized and extended to a general nonlinear optimal control setting the classical Euler-Lagrange and Weierstrass conditions of
the Calculus of Variations. The early optimality conditions, called Pontryagin Maximum principle (in short, PMP), were subsequently strengthened and extended using methods of convex
and non-smooth analysis, or methods of differential geometry (see \cite{Pontryagin_book}). The PMP inspired effective computational schemes, like the shooting method presented in Section \ref{s4}.  As illustrated on some classical examples, the shooting method provides very accurate numerical solutions. The major drawback of this method is that it requires an {\em a priori} knowledge of the structure of the solution, as well as a good approximation of the adjoint state.  In addition, shooting methods are difficult to implement in state-constrained problems, in particular for a large number of state constraints.

Another major advance in optimal control was achieved in the 1950's by Richard Bellman, who provided a  description of how the minimum cost depends on initial conditions, as the solution to the so-called Hamilton--Jacobi--Bellman (HJB) partial differential equation. Whenever applicable, the Hamilton--Jacobi approach yields a global solution to the optimal control
problem, and provides the optimal control in a feedback form, suitable for many engineering applications. Despite these advantages, this 
approach suffers from the difficulty of computing the solution to the Hamilton--Jacobi equation in higher dimensions.
The numerical simulations presented in Section \ref{s5} show that the approximation of the HJB equation, even on coarse grids, provides solutions which depict well the qualitative structure of the optimal trajectories and may thus be used to guess the structure of optimal solutions. However, if an accurate computation of the optimal trajectories is required, the use of fine grids causes a strong increase in computational complexity.

More recently, other global methods have been developed for control problems, such as the optimistic planning (OP) algorithms introduced in Section \ref{s42}. These methods are based on a discretization of the control space and do not require any discretization of the state space. As a consequence, OP methods are particularly efficient in control problems with a low-dimensional control space; moreover, this approach seems also well suited for problems where the dynamics and the cost are given by learning models.
A preliminary analysis of the complexity of OP methods is now well established, but further developments are expected to make them more accurate, to possibly integrate the knowledge of the structure of trajectories, and also for a more efficient implementation. 

Throughout the chapter we consider two well known but representative optimal control examples, the Zermelo and the Goddard problem: both are quite simple but can be extended towards more intricate models. We use these two examples to illustrate the numerical methods and show how they can be rapidly and efficiently implemented, with up-to-date existing solvers.
We argue that the various approaches are complementary rather than in competition, and can be suitably combined to exploit their peculiarities.
For instance, HJB approach and direct methods can be used to obtain a rough estimate of adjoint state and cost, which could provide a good initial guess for the more accurate shooting method. Moreover, HJB solvers may allow to rule out local minima as it will be illustrated on the Zermelo problem with obstacle.

We hope that the codes that we provide, which are also available on the web, can serve as templates to readers interested in adapting them to their specific setting.

\paragraph{Notations.} 
%============================================
Throughout the chapter, $\R$ denotes the set of real numbers,  
$\langle \cdot,\cdot\rangle$ and $\|\cdot\|$ denote respectively
the Euclidean inner product and the norm on $\R^N$(for any $N\geq 1$),
$\BB_N=\{x\in\R^N:\ \|x\|\leq 1\}$ is the closed unit ball 
(also denoted $\BB$ if there is no ambiguity) and $\BB(x;r)=x+r\BB$. 
For any set $S\subseteq\R^N$, $\inte{S}, \overline{S}$, $\partial S$, $\Co S$
denote its interior, closure, boundary, and convex envelope, respectively. \ For any $a,b\in \R$, we define
	$a\bigvee b:= \max(a,b).$
	Similarly, for $a_1,\cdots,a_m\in \R$, we define 
	$ \bigvee_{i=1}^{m}a_i:= \max(a_1,\cdots,a_m).$
The  notation $W^{1,1}([a,b])$ stands for the usual Sobolev space $\{f\in L^1([a,b]), f'\in L^1([a,b])\}$.
Finally, the abbreviation "w.r.t." stands for "with respect to", and "a.e." means "almost everywhere".

%-------------------------------------------------
\subsection{Formulation of the optimal control problem}
%------------------------------------------------
Let $d,r\in\N^*$, let $T>0$ be
a fixed final time horizon and let $U$ be a compact subset of $\R^r$ (with $r\geq 1$).
We consider the finite-dimensional control system in $\R^d$ (for $0\leq t<T$)
%\begin{subequations}
\begin{eqnarray}\label{eq:etat}
 & & \dot \bx(s) = f(s,\bx(s), \bu(s)), \ \mbox{a.e.} \ s\in (t,T),
\end{eqnarray}
%\end{subequations}
%
where the control input $\bu:[0,T]\longrightarrow \R^r$ ($r\geq 1$) is a measurable function such that $\bu(s)\in U$ 
for almost every~$s\in [0,T]$. Throughout the paper, we assume that

\medskip

%\begin{align*}

%\begin{indent}
\mbox{\Hyp{0}}
$\quad$ \text{$U$ is a compact subset of $\R^r$.}
%\end{indent}

\medskip

\noindent
We denote by $\cU$ the set of all admissible controls
$$\cU:=\{\bu:[0,T]\longrightarrow \R^r \mbox{ measurable, and }
\bu(s)\in U \mbox{ a.e.}\}.$$
The dynamics $f:[0,T]\times\R^d\times U\converge \R^d$ 
  satisfies: 
\begin{align*}
\hspace*{0.5cm} 
\mbox{{\Hyp{1a}}} \hspace*{0.5cm}
\begin{cases}
(i) \quad  f \text{ is continuous on }[0,T]\times\R^d\times U;\\
(ii)\quad \text{$x\rightarrow f(s,x,u)$ is locally Lipschitz continuous in the following sense:}\\
\phantom{(ii)\quad}
  \forall R>0,\ \exists k_R\geq 0,\ \forall (x,y)\in (\BB_{d}(0;R))^2,\ \forall (s,u)\in[0,T]\times U \\
  \phantom{(ii)\quad}
  \hspace{1cm}
  \|f(s,x,u)-f(s,y,u)\| \leq k_R \|x-y \|; \\
(iii)\  \exists c_f>0\text{ such that }\ \
  \|f(s,x,u)\|\  \leq\ c_f(1+\|x\|) \ \  \forall (s,x,u)\in[0,T]\times \R^d\times U.
\end{cases}\hspace*{3cm}
\end{align*}
\noindent The assumptions 
$(ii)$-$(iii)$ of \Hyp{1a} guarantee, 
for every $\bu\in \cU$ and $(t,x)\in[0,T]\times\R^d$, 
the existence of an absolutely continuous curve $\bx:[t,T]\to\R^N$
which satisfies \eqref{eq:etat} and the initial condition $\bx(t)=x$.
By the Gronwall lemma,
\begin{align}\label{eq:traj_gronwall}
1+ \|\bx(s)\|\leq(1+\|x\|)e^{c_f(s-t)}\qquad\forall s\in[t,T].
\end{align}
Given any $t\in [0,T]$ and any $x\in\R^d$,
the set of all admissible pairs 
control-and-trajectories starting at $x$ at time $t$ is denoted by
$$
  \XX_{[t,T]}(x):= 
  \{ (\bx,\bu) \in W^{1,1}(t,T)\times\cU\mid 
   \ (\bx,\bu) \mbox{ satisfies \eqref{eq:etat} with } \bx(t)=x \}.
$$
Throughout the chapter, we consider the (Bolza) optimal control problem  
\begin{equation}\left\{\begin{array}{l}
\mbox{minimize } \ 
   \displaystyle\varphi(\bx(T)) + \int_t^T \ell(s,\bx(s),\bu(s))\,ds, \\
 \hspace*{1cm} (\bx,\bu)\in \XX_{[t,T]}(x), 
\\
   \hspace*{1cm} g(\bx(s))\leq 0 \quad \mbox{for } s\in [t,T],\\
 \hspace*{1cm} g_f(\bx(T))\leq 0,
\end{array}\right. \label{eq.Pb_CO}
\end{equation}
with the convention that $\inf\emptyset=+\infty$. 
The distributed cost $\ell:[0,T]\times \R^d\times U\to \R$, the final cost $\varphi:\R^d\to\R$,
and the constraint functions $g$ and $g_f$ are given functions satisfying:
\begin{align*}
\hspace*{0.5cm} 
\mbox{{\Hyp{1b}}} \hspace*{0.5cm}
\begin{cases}
  (i) \quad  \ell \text{ is continuous on }[0,T]\times\R^d\times U;\\
%ii)\ \ \text{ for every $s\in[0,T]$ and $u\in U$,}\ f(s,\cdot,u)\ \mbox{is of class}\ C^1.\\
(ii)\quad \text{$x\rightarrow \ell(s,x,u)$ is locally Lipschitz continuous in the following sense:}\\
\phantom{(ii)\quad}
  %\hspace{2cm} 
  \forall R>0,\ \exists k_R\geq 0,\ \forall (x,y)\in (\BB_{d}(0;R))^2,\ \forall (s,u)\in[0,T]\times U \\
  %\hspace{3cm}
\phantom{(ii)\quad}
  \hspace{1cm}
  |\ell(s,x,u)-\ell(s,y,u)| \leq k_R \|x-y \|; \\
(iii)\  \exists c_\ell>0\text{ such that }\ \   |\ell(s,x,u)|  \leq\ c_\ell(1+\|x\|)\ \  \forall (s,x,u)\in[0,T]\times \R^d\times U.
\end{cases}\hspace*{3cm}
\end{align*}

\Hyp{2} $\varphi:\R^d\to \R$ is 
locally Lipschitz continuous
and there exists a constant $c_\varphi\geq 0$ such that 
$$|\varphi(y)|\leq c_\varphi(1+\|y\|).$$

\Hyp{3} The constraint functions 
$g:\R^d\to \R^\cmg$ and $g_f:\R^d\to \R^\cmf$ are 
locally Lipschitz continuous (with $\cmg,\cmf\geq 1$)
and there exists a constant $c_g\geq 0$ such that, for every $y\in \R^d$, 
$$\|g(y)\|+\|g_f(y)\|\leq c_g(1+\|y\|).$$
\noindent The so-called augmented control system
\begin{subequations} \label{eq:augmented_state}
\begin{eqnarray}
\dot\bx(s) & = & f(x,\bx(s),\bu(s))\quad \mbox{a.e. } s\in (t,T),\\
\dot\bz(s) & = & -\ell(s,\bx(s),\bu(s))\quad \mbox{a.e. } s\in (t,T),
\end{eqnarray}
\end{subequations}
is usually considered in optimal control theory to recast the Bolza problem in the Mayer form
\begin{equation}\left\{\begin{array}{l}
\mbox{minimize } \ 
   \displaystyle\varphi(\bx(T)) - \bz(T), \\
 \hspace*{1cm} (\bx,\bu)\in \XX_{[t,T]}(x_0), 
\\
   \hspace*{1cm} g(\bx(s))\leq 0 \quad \mbox{for } s\in [t,T],\\
 \hspace*{1cm} g_f(\bx(T))\leq 0,
\end{array}\right. \label{eq.Pb_CO_Reformulation}
\end{equation}

In what follows, we will {\bf sometimes} assume that the augmented dynamics satisfies the following assumption (convex epigraph):

\Hyp{4} For any $s\in [0,T]$ and every $x\in \R^d$, 
  %$f(s,x,U)$ 
$$
  \bigg\{ \left(\begin{array}{c}
   f(s,x,u)\\ -\ell(s,x,u)+\eta \end{array}\right), \ \ u\in U,\ \ -c_\ell(1+\|x\|)+\ell(s,x,u)\leq \eta\leq 0 \bigg\}
  \quad \mbox{is a convex set.}
$$ 
%\fbox{VERIFIER/ADAPTER}

\begin{rem} Note that if $\ell\equiv 0$ (Mayer problem), \Hyp{4} reduces to
  \begin{equation*} \label{eq:H4b}
  \mbox{$f(s,x,U)$ convex for all $(s,x)\in[0,T]\times \R^d$.}
 \end{equation*}
\end{rem}

\begin{rem}\label{rem.existence}
For $0\leq a<b \leq T$ and $(x,z)\in \R^d\times \R$, consider the set of all trajectories satisfying \eqref{eq:augmented_state} on the time interval $[a,b]$, for a control input $u\in \cU$, and starting at a position $(x,z)$ at time $a$:
\begin{eqnarray*}
  \cS_{[a,b]}(x,z):=\{ (\bx,\bz)\in W^{1,1}([a,b])\mid \exists \bu\in \cU
  \ \mbox{ such that } (\bx,\bz,\bu) \mbox{ satisfies } \eqref{eq:augmented_state}, \hspace*{1cm}& & \\
\mbox{ with } \bx(a)=x,\ \bz(a)=z\}. & & 
\end{eqnarray*}
\noindent Under assumptions \Hyp{0}, \Hyp{1a}, \Hyp{1b},  by \eqref{eq:traj_gronwall},
the set $\cS_{[a,b]}(x,z)$ is bounded in $W^{1,1}([a,b])$.  Moreover, if \Hyp{4} is satisfied then  $\cS_{[a,b]}(x,z)$ is a compact set in $W^{1,1}([a,b])$ endowed with the $C^0([a,b])$-topology (see
\cite[Theorem 1.4.1]{CELAUB}). 
Therefore, 
if there exists  a trajectory $\bx\in \cS_{[0,T]}(x_0,0)$ 
that satisfies the  constraints
$g(\bx(s))\leq 0$ for all $s\in [t,T]$ and $g(\bx(T))\leq 0$,
then the  control problem \eqref{eq.Pb_CO_Reformulation} has an optimal solution.
%%%
%%%
%%%  
\end{rem}

\subsection{Examples} \label{s2}
In this section, we present two well-known examples that we are going to consider throughout. In our opinion, they illustrate nicely the most classical difficulties encountered in theoretical and numerical optimal control, and lend themselves to a number of more complicated variants. They are expected to serve as ``templates'' to the reader who aims at getting acquainted with the main issues in numerical optimal control.

\paragraph{Example 1: Zermelo problem} 
A boat with coordinates $\bx(t)=(\bx_1(t),\bx_2(t))$ navigates through a canal $\R\times[a,b]$, starting at 
$\bx(0)=x=(x_1,x_2)$, and wants to reach an island $\cC$ with minimal cost. The cost function may be an energy, the final time, etc. The control system is 
\begin{subequations}\label{eq:dyn_zermelo}
\begin{eqnarray}
  & & \dot \bx_1(t) =  \bv(t) \cos(\bu(t)) + h(\bx(t)),\\
  & & \dot \bx_2 (t) =  \bv(t) \sin(\bu(t)),
\end{eqnarray}
\end{subequations}
where $\bu(t) \in [0,2\pi]$ is the first control (angle), $\bv(t)\in [0,V_{max}]$ is a second control (speed of the boat), 
and $h(\bx(t))$ is the current drift (along the $x_1$-axis).
Because of the drift term (which can be greater than $V_{\max}$), the system may not be controllable.  Consider a  target $\cC:=\mathbb{B}(c,r_0)$ that is a   ball with radius $r_0\geq 0$ and centered at a given point $c$ located in the canal.  The target  represented by a function $g_f$ defined by 
$g_f(x):=\|x-c\|-r_0$  as 
$$x\in \cC \Longleftrightarrow g_f(x)\leq 0.$$
Consider also a set of constraints $\cK:=\{x\in \R^2,\ g(x)\leq 0\}$ where $g$ is a given function that is non-positive in a region where the boat can move and $g$ is positive in the location of the obstacles that the boat should avoid.   In this example, the cost function could be the time, or the energy, required to steer the boat  from a given position $x$ to the target $\cC$.  

The Zermelo problem has several variants depending on the choice of the dynamic $h$, on the expression of the constraints, as well as on the values of the different parameters entering the model ($a,b,$ and $V_{max}$). In the next sections, we will consider different settings to better illustrate the pros and cons of each numerical method. 

\paragraph{Example 2: Goddard Problem}
We consider the optimal control problem associated with a vertical ascension flight of 
a rocket,  known as Goddard problem. 
 The dynamics of the rocket is  defined with three state variables: $r$, the altitude, $v$ the relative speed  and $m$, the total mass. In general, a  dimensionless version for the motion equations is considered in the literature:
 \begin{subequations}
\be
  \dot \br(t) & = & \bv(t) \label{eq101}\\
  \dot \bv(t) & = & \frac{1}{m} (T_{max}\, \bu(t) - D(\br(t),\bv(t))) - \frac{1}{\br(t)^2} \\
  \dot\bbm(t) & = & - T_{max}\, b\, \bu(t) \label{eq103}
\ee 
\end{subequations}
where $D(r,v)$ is the drag force and $T_{max}\cdot u$ is the thrust force.
The  dimensionless initial state of the rocket is given by 
$\br(0)=1$, $\bv(0)=0$, $\bbm(0)=1$.
Here $\br(0)=1$ corresponds to the Earth's ground level and $m=1$ to the initial total mass of the rocket.
The motion of the rocket is controlled 
by the thrust factor $\bu(s)\in[0, 1] $ so that  the thrust force is on the interval $ T_{max}\cdot \bu(t)\in [0, T_{max}]$.
The drag  force is  a nonlinear function of $\br$ and $\bv$. Its expression depends on the choice of a model for the 
atmosphere and  on the structure of the rocket:
$D(r,v)= C_D \rho(r) |v| v$,
where $C_D$ is the drag coefficient of the rocket and $\rho$ is the atmospheric density.  In this section we 
consider  the  case of constant drag coefficient   (it depends in general on the Mach number)  and exponential model of the 
atmospheric density:
$$
  D(r,v)= C_D\ v^2\ e^{-\beta\,(r-1)} \quad 
  \forall r\geq 1, v\geq 0.
$$
This definition of the rocket's model has been widely studied in the literature. We take
$C_D=310.0$, $\beta=500.0$, $T_{max}=3.5$, $b=2.0$,
as in \cite{Tsiotras_Kelley_92} and other references. The corresponding model is then an approximation of real flight conditions for some rockets. 
The optimal control problem consists of  maximizing the final altitude $\br(t_f)$ at free final time $t_f>0$ under 
a final  constraint on the fuel consumption, and a pointwise constraint on the velocity. 
Finally, the optimal control problem is formulated as follows
\begin{equation}\label{Pb.Goddard}
\left\{\begin{array}{l}
\mbox{minimize } \ 
   \displaystyle\bigg(-\br(t_f)\bigg), \\
 \hspace*{0.5cm} (\br,\bv,\bbm) \mbox{ satisfies \eqref{eq101}--\eqref{eq103} with } \br(0)=1, \ \bv(0)=0, \ 
 \bbm(0)=1, 
\\
   \hspace*{0.5cm} 
    \bv(s)\leq v_{\max} \ \ \mbox{a.e. on } (0,t_f)\\
   \hspace*{0.5cm}  m^\ast-\bbm(t_f)\le 0 .
\end{array}\right. 
\end{equation}
% $$
% \begin{array}{l}
%  \hspace*{-0.8cm} \disp \inf_{\bu(.), t_f>0}\  \\[2mm]
%   \dot \br (s)= \bv(s), \quad   \br(0)=1,   \\[2mm]
%   \dot \bv(s) =\displaystyle \frac{1}{m} (T_{max}\, \bu(s) - D(\br(s),\bv(s))) - \frac{1}{\br(s)^2}, \quad  \ \bv(0)=0, \\[2mm]
%   \dot \bbm(s) = - T_{max}\, b\, \bu(s), \quad  \bbm(0)=1;\\[2mm] 
%   m^\ast-\bbm(t_f)\le 0 \quad \mbox{and }  \bv(s)\leq v_{\max} \ \ \mbox{a.e. on } (0,t_f)
% \end{array}
% $$
In the numerical simulations that will be presented in the next sections, the limit of fuel consumption is  $m^\ast=0.6$ and the maximum velocity is $v_{\max}=0.1$.

%% file: 020-Direct_Methods.tex
% todo: deal with greek letters in verbatim; cite bulirsh et al in section 4
\section{Direct methods: nonlinear programming}\label{sec_direct}
%{\color{red}
%Outline:
%\begin{itemize}
%\item general formulation (including as general as possible constraints)
%\item Goddard and Zermelo examples (+ constraints)---refer to 2.3
%\item basic collocation schemes (midpoint / Crank-Nicolson)
%\item convergence results (Hager {\em al} for collocation methods)
%\item modern optimisation and differentiable programming
%\item a do it yourself direct solver (+ notebooks: \texttt{JuMP, ct})
%\item misc: include example of LMI / SDP solve of convex control (e.g. %CDC paper) (alternative discretisation); direct shooting (Diehl {\em et %al}); with DAE (Gerdts {\em et al}); robustness (?)
%\end{itemize}
%}

\subsection{Principle}
We call \emph{direct methods} all numerical methods consisting of numerically solving the optimal control problem as follows: without applying a priori any first-order necessary condition for optimality, we choose a discretization for the state and for the control, we choose a numerical scheme to discretize the control system (differential equation) and the cost functional (integral quadrature), so that the discretized optimal control problem is expressed as a family of nonlinear optimization problems in finite dimension, indexed by a discretization parameter $N$, of the form
\begin{equation}\label{pboptim}
\min_{Z\in \mathcal{C}} F(Z)
\end{equation}
where $Z=(x_1,\ldots,x_N,u_1,\ldots,u_N)$ and
\begin{equation}\label{defC}
\mathcal{C} = \{ Z\ \vert\  g_i(Z)=0,\ i\in{1,\ldots,r},\quad g_j(Z)\leq 0,\ j\in{r+1,\ldots,m} \} .
\end{equation}
This is a classical optimization problem under constraints in finite dimension, with a dimension growing as the discretization is refined. 
Of course, there exist an infinite number of variants to discretize the problem and ending up with a problem of the form \eqref{pboptim}. We discuss hereafter several classes of discretizations. Once this transcription has been done, one can then implement a number of various optimization routines to solve \eqref{pboptim}.

Let us first explain hereafter one of the simplest possible discretizations.
Consider the optimal control problem \eqref{eq.Pb_CO} with $t=0$ as initial time.
Consider a subdivision $0=t_0<t_1<\cdots<t_N=T$ of the interval $[0,T]$. Controls are discretized on $U$-valued piecewise constant functions on this subdivision. To discretize the control system, we choose the explicit Euler method: setting $h_i=t_{i+1}-t_i$, we have $x_{i+1}=x_i+h_if(t_i,x_i,u_i)$ for $i=0,\ldots,N-1$. To discretize the integral cost, we choose the left rectangle method (which is equivalent to the explicit Euler method for the augmented system). We obtain the nonlinear programming problem ($x_0$ being known)
\begin{equation*}
\begin{split}
& \min C(x_1,\ldots,x_N,u_0,\ldots,u_{N-1}), \\
& x_{i+1}=x_i+h_if(t_i,x_i,u_i),\quad u_i\in U,\\
& g(x_{i+1})\leq 0, \qquad i=0,\ldots, N-1,\\
& g_f(x_N)\leq 0.
\end{split}
\end{equation*}

\noindent We note that this approach is flexible and robust insofar it is not much sensitive to the model (contrarily to the shooting method, described further): it is very easy to add various constraints to the optimal control problem. This is why direct methods are often privileged when the model is not completely fixed. The resulting numerical simulations often give an interesting feedback that may lead to change or adapt the optimal control model under consideration.

\subsection{Practical numerical implementation}
The numerical implementation of such a nonlinear programming problem is standard and can be done in a number of ways, for instance using a penalty method or a sequential quadratic programming (SQP) method or dual methods (like Uzawa's). 
A survey on the use of direct methods in optimal control, with a special interest to applications in aerospace, can be found in \cite{Betts}.

From the point of view of practical implementation, in the last years much progress has been done in the direction of combining automatic differentiation softwares (such as the modelling language \texttt{AMPL}, see \cite{AMPL}, or \texttt{CasADi}, see \cite{Andersson2018}) with expert optimization routines (such as the open-source package \texttt{Ipopt}, see \cite{IPOPT}, providing an interior point optimization algorithm for large-scale differential algebraic systems combined with a filter line-search method). 
With such tools, it has become very simple to implement with only few lines of code difficult (nonacademic) optimal control problems, with success and within a reasonable time of computation. Websites such as \texttt{NEOS} (\href{http://www.neos-server.org/neos/solvers/index.html}{www.neos-server.org/neos/solvers}) propose to launch online such kinds of computation: codes can be written in a modelling language such as \cite{AMPL} (or others) and can be combined with many optimization routines (specialized either for linear problems, nonlinear, mixed, discrete, etc). The advantage of using \texttt{NEOS} is that one has nothing to install on his own machine, and moreover one can test a large number of possible optimization routines.

Note that there exist a large number (open-source or not) of automatic differentiation softwares and of optimization routines. It is not our aim to provide a list of them, since they are easy to find on the web.
Note also that \texttt{AMPL}, which is very easy and friendly to use, is however not free of charge (although the licence is not expensive) and that \texttt{CasADi} offers a very good and efficient free alternative, see \href{https://web.casadi.org}{web.casadi.org}.

\subsection{Variants}
As alluded above, there exist many possible approaches to discretize an optimal control problem, see \cite{Betts} where the important sparsity issues are also discussed. Among those various approaches, we quote the following.

Collocation methods consist of choosing specific points or nodes on every subinterval of a given subdivision of the time interval. Such methods approximate the trajectories and the controls by polynomials on each subinterval. The collocation conditions state that the derivatives of the approximated state match with the dynamics at each node. 

Spectral and pseudospectral methods are another class in which the above nodes are chosen as the zeros of specific polynomials such as Gauss-Legendre or Gauss-Lobatto polynomials. Such polynomials are used as a basis to approximate trajectories and controls in appropriate approximation spaces. 
Since they share nice orthogonality properties, the collocation conditions turn into constraints that are easily tractable for numerical purposes. We refer the reader to \cite{ElnagarKazemi,GRKF,RossFahroo} and to the references therein for more details.

There exist also some probabilistic approaches, such as the method described in \cite{HLPT} which consists of first relaxing the optimal control problem in measure spaces and then of seeking the optimal control as an occupation measure, which is approximated by a finite number of its moments (see  \cite{Lasserre}). This approach relies on algebraic geometry tools and reduces the optimal control problem to some finite dimensional optimization problem involving linear matrix inequalities (LMI).
On this topic involving Sums-of-Square (SoS) considerations, we refer the reader to another chapter of the Handbook, \cite{Lasserre_Handbook}.

\begin{rem}
Direct methods are characterized by first discretizing and then optimizing, i.e., optimality conditions are applied in a second step to the discretized model; in contrast to this approach, indirect methods (to be discussed in the next section) consist of applying first optimality conditions (the Pontryagin Maximum Principle) and then discretizing the resulting boundary value problem.
While the latter method clearly falls in the classical Lax scheme, ``consistency plus stability imply convergence'', there is a serious gap there in the direct approaches: to ensure convergence using the Lax scheme, one would a priori need a \emph{uniform} (with respect to $N$) consistency property, which fails in general because the optimal control problem is an optimization problem in infinite dimensions. 
Surprisingly simple examples of divergence are provided in \cite{Hager2000}.
In this same paper, it is shown that convergence is obtained for ``smooth enough'' optimal control problems without any final state constraint, discretized with Runge--Kutta methods with positive coefficients. 
We also refer to \cite{BonnansLaurentVarin,SanzSerna} for further comments and considerations on symplectic integrators.
Convergence has also been established for classes of Legendre pseudospectral methods (see \cite{ElnagarKazemi,GRKF,RossFahroo}).
\end{rem}

\subsection{Goddard problem by a direct approach} \label{s32}
To illustrate the method, we present a treatment of the Goddard case presented Section~\ref{s2} using the Crank-Nicolson scheme. This scheme, which is \emph{dual} to the midpoint rule (in the sense defined in \cite{hairer-2006a}), has the advantage of using only gridpoints (contrary to the midpoint scheme). It is very easy to implement and the code below is a basic ``do it yourself'' direct solver that can easily be adapted to other problems. We use the nice \texttt{JuMP} interface in \texttt{julia} to define the discretization of the problem and call the celebrated interior point solver \texttt{Ipopt} previously mentioned. The symbolic-numeric framework put forward by \texttt{Julia} allows to efficiently and transparently use AD (automatic differentiation / differentiable programming) and sparse linear algebra (which is important for structured constraints stemming from one-step like methods to discretize the dynamics).
In the spirit of reproducible research \cite{donoho-1995a}, the code itself is available and executable online on the gallery of the \texttt{ct:\,control toolbox} project.\footnote{\href{https://ct.gitlabpages.inria.fr/gallery}%
{ct.gitlabpages.inria.fr/gallery}}
Discretizing the dynamics boils down to the following lines (note that nonlinear expressions encoding the right-hand side are first defined):
\begin{verbatim}
# Dynamics
@NLexpressions(sys, begin
    # D = Cd v^2 exp(-β(r-1))
    D[i = 1:N+1], Cd * v[i]^2 * exp(-β * (r[i] - 1.0))
    # r'= v
    dr[i = 1:N+1], v[i]
    # v' = (Tmax.u-D)/m - 1/r^2
    dv[i = 1:N+1], (Tmax*u[i]-D[i])/m[i] - 1/r[i]^2
    # m' = -b.Tmax.u
    dm[i = 1:N+1], -b*Tmax*u[i]
end)

# Crank-Nicolson scheme
@NLconstraints(sys, begin
    con_dr[i = 1:N], r[i+1] == r[i] + Δt * (dr[i] + dr[i+1])/2.0
    con_dv[i = 1:N], v[i+1] == v[i] + Δt * (dv[i] + dv[i+1])/2.0
    con_dm[i = 1:N], m[i+1] == m[i] + Δt * (dm[i] + dm[i+1])/2.0
end)
\end{verbatim}
Note that the constraints have been labeled so that the corresponding Lagrange multipliers can be retrieved. They are indeed approximations of the costate of the continuous problem and will be used as such to initialize successfully a subsequent shooting method (see Section~\ref{s4}).

A strong benefit of direct methods is that state constraints are very easy to handle. We may for instance add a state constraint on the velocity (note that some other constraints have been added, directly when defining the unknowns of the problem, to improve convergence of the solver; what would be a complication for indirect methods is actually an asset here):
\begin{verbatim}
# As final time is free, time step Δt is unknown
@variables(sys, begin
    0.0 ≤ Δt
    r[1:N+1] ≥ r0
    0 ≤ v[1:N+1] ≤ vmax
    mf ≤ m[1:N+1] ≤ m0
    0.0 ≤ u[1:N+1] ≤ 1.0
end)
\end{verbatim}
Boundary constraints are obviously added in the same way, and for $N=100$ gridpoints we get the results given on Figures~\ref{fig31} and \ref{fig32}. The resulting optimization problem solved has about $400$ variables (the state is of dimension $3$ and the control is scalar) and $300$ equality constraints (the discretized dynamics), plus box constraints.
One could use a finer grid (up to the price of a larger problem), but the result we obtain turns to be precise enough to trigger convergence of much more accurate solver by shooting in the next section. The bang-singular-constrained-bang structure of the solution has indeed been captured by the direct solver, and this is essentially all we need to resort to indirect methods to complete the computation.

\begin{figure}
    \centering
    \includegraphics[width=0.7\textwidth]{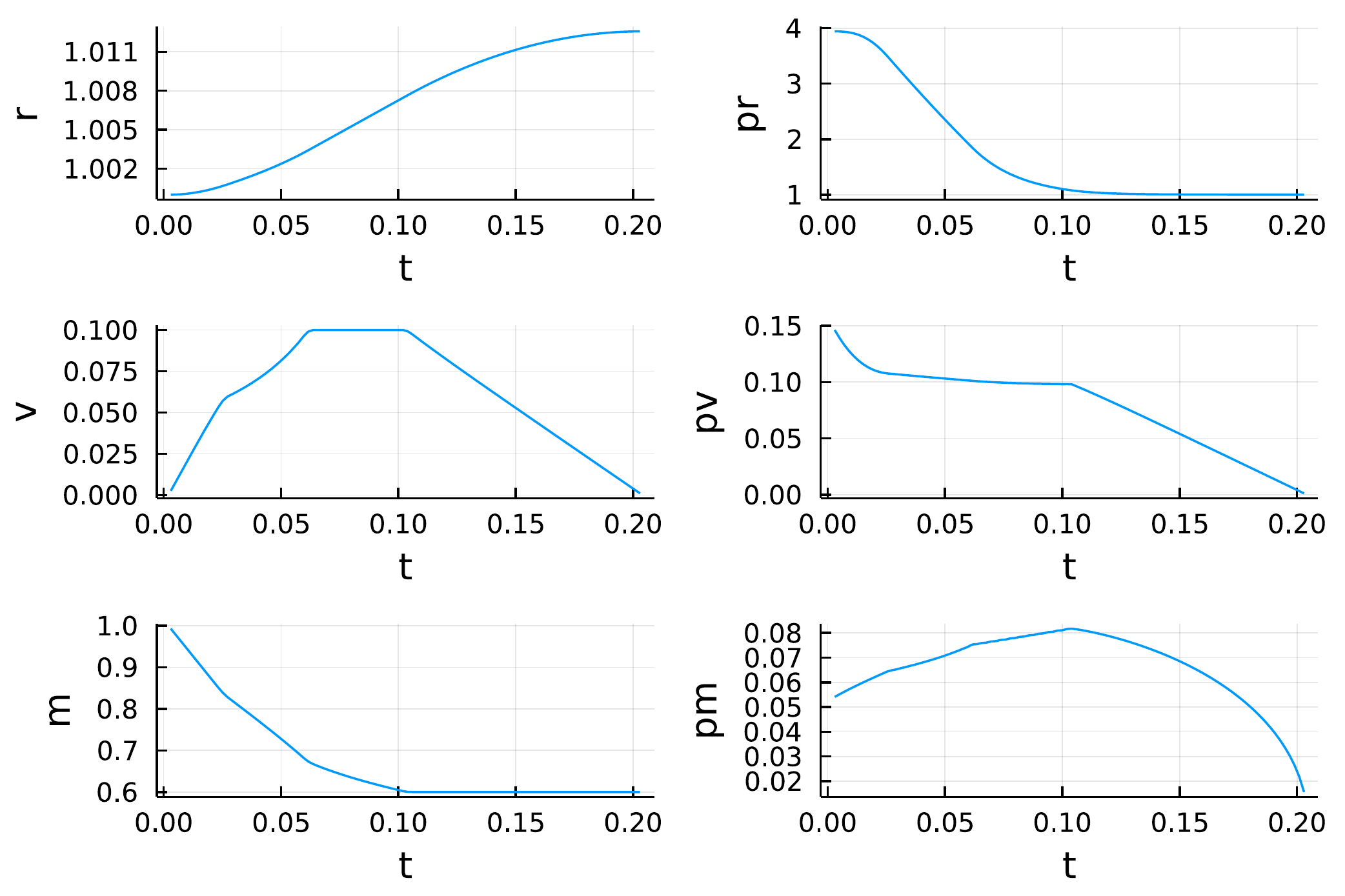}
    \caption{\small Goddard problem: result of the direct code (Crank-Nicolson scheme), states and Lagrange multipliers associated with the equality constraints that discretize the dynamics. These multipliers approximate the adjoint states and will be used as such to initialise an indirect method. A boundary arc is observed as the state constraint on the velocity is saturated.}
    \label{fig31}
\end{figure}

\begin{figure}
    \centering
    \includegraphics[width=0.7\textwidth,height=0.3\textheight]{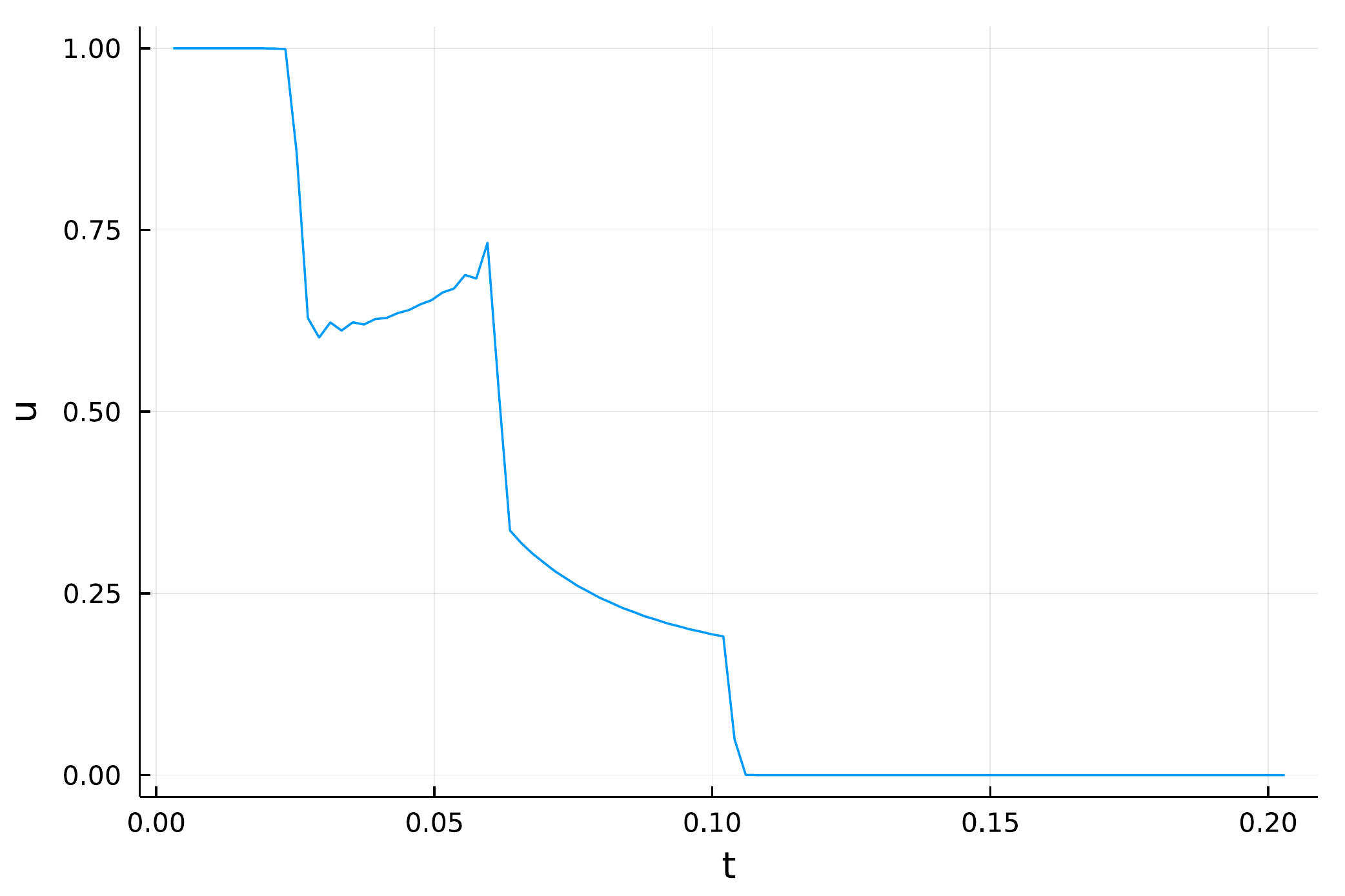}
    \caption{\small Goddard problem: result of the direct code, control values. A structure with four different subarcs is observed. A simple inspection indicates a concatenation of a bang arc ($u \simeq 1$), possibly of a singular arc interior to the control bounds ($u \in (0,1)$), of a boundary arc due to the state constraint on the velocity, and of a final bang arc ($u \simeq 0$). This structure information, of combinatorial nature, is obtained here without any \emph{a priori} knowledge on the solution. It is the key to define an appropriate shooting function and solve very accurately the problem thanks to an indirect method.}
    \label{fig32}
\end{figure}

%% file: 030-Shooting_Methods.tex
% todo: change sign of state constraint (+ notations); check refs
\section{Indirect approaches: the shooting method} \label{s4}
%{\color{red}
%\begin{itemize}
%    \item Principe de Pontryagin
%    \item Inward pointing Condition: Normality of the PMP
%    \item Shooting method : numerics, convergence
%    \item Turnpike property
%\end{itemize}
%}

\subsection{Pontryagin maximum principle}
In all this section, we assume that hypotheses \Hyp{0}-\Hyp{3} are satisfied, and that the functions 
 $\varphi, f, \ell $ and
$g_f$ are of class $C^1$ w.r.t. the space variable $x$.
Consider the optimal control problem \eqref{eq.Pb_CO} with $t=0$ as initial time, and with $g=0$, i.e., for the moment, with no state constraint. 
%\todo{est-ce qu'on garde ici $\mathcal{U}_{T,U}$ ou est-ce qu'on change en $\mathcal U$ comme dans le reste des sections? : oui, OK}
Recall that the \textit{end-point mapping}
%:\ \R^d\times \R^+\times{\mathcal{U}\rightarrow\R^d$
of the system is defined by $E(x_0,T,\bu)=\bx(x_0,T,\bu)$, where $t\mapsto \bx(x_0,t,\bu)$ is the trajectory solution of the control system, corresponding to the control $\bu$, such that $\bx(x_0,0,\bu)=x_0$.
%We denote by $\mathcal{U}$ is the set of controls taking their values in $U$ and such 
%that the corresponding trajectory is well defined on $[0,T]$ (no blow-up). %\todo{On doit faire des hypoth\`eses sur la dynamique. J'ai écrit que les (H0)-(H3) sont vérifiées + fonctions C1.: OK}
%
%The set $\mathcal{U}_{T,\R^r}$, endowed with the standard topology of
% $L^\infty([0,T],\R^r)$, is open, and
The end-point mapping is well defined and $C^1$ for $\bu$ belonging to a neighbourhood in $L^\infty([0,T],\R^r)$ of the reference control (and $C^k$  whenever the dynamics are $C^k$).
Denoting by $C$ the cost functional, the optimal control problem can be written as the infinite dimensional minimization problem of minimizing $C$ over the considered set of controls, under the constraint $g_f(E(x_0,T,\bu))\leq 0$ on the final state.

We first assume that we are in the simple case where the initial point $x_0$ is fixed, the final point $x_1$ is fixed, the final time $T$ is fixed, $U=\R^r$ and there are no state constraints along the trajectory. Then, the optimal control problem is exactly equivalent to the optimization problem
\begin{equation}\label{optpb1}
\min_{E(x_0,T,\bu)=x_1} C(T,\bu).
\end{equation}
If $\bu$ is optimal, then there must exist Lagrange multipliers
$(\psi,\psi^0)\in\left(\R^d\times\R\right)\setminus\{0\}$ such that
\begin{equation}\label{lagrmultrule}
\psi \cdot dE_{x_0,T}(\bu) = -\psi^0 dC_T(\bu).
\end{equation}
\noindent Defining the Lagrangian $L_T(\bu,\psi,\psi^0)=\psi E_{x_0,T}(\bu) + \psi^0 C_T(\bu)$, the first-order condition \eqref{lagrmultrule} is written as
\begin{equation}\label{1storderneccond}
\frac{\partial L_T}{\partial u}(\bu,\psi,\psi^0)= 0.
\end{equation}
The Pontryagin maximum principle (PMP), recalled hereafter, is a far-reaching statement of these conditions (see \cite{Pontryagin_book}, see also \cite{LeeMarkus}).
We recall here the statement of the PMP in the case without pointwise state constraints (that is for $g=0$).
%\todo{Attention: Hamiltonien non défini: OK}
\begin{thm}
If the trajectory $\bx$, associated to the optimal control $\bu$ on $[0,T]$, is optimal, then it must be the projection of an \textit{extremal} $(\bx,\bp,p^0,u)$ (called extremal lift), where $p^0\leq 0$ and $\bp:[0,T]\rightarrow\R^d$ is an absolutely continuous mapping called \textit{adjoint vector}, with $(\bp,p^0)\neq (0,0)$, such that
%\todo{shouldn't $x(t)$ be $\bx(t)$, and $u(t)$ be $\bu(t)$ (in all the section)?: Yes, done}
\begin{equation*} \label{systPMP}
\dot{\bx}(t)=\frac{\partial H}{\partial p}(t,\bx(t),\bp(t),p^0,\bu(t)),\
\dot{\bp}(t)=-\frac{\partial H}{\partial x}(t,\bx(t),\bp(t),p^0,\bu(t)),
\end{equation*}
almost everywhere on $[0,T]$,
where $H(t,x,p,p^0,u)=\langle p,f(t,x,u)\rangle+p^0\ell(t,x,u)$ is the Hamiltonian, and the condition
\begin{equation} \label{contraintePMP}
H(t,\bx(t),\bp(t),p^0,\bu(t))=\max_{v\in U} H(t,\bx(t),\bp(t),p^0,v)
\end{equation}
holds almost everywhere on $[0,T]$.
If moreover the final time $T$ is not fixed, then 
\begin{equation}\label{condannul}
\max_{v\in U} H(T,\bx(T),\bp(T),p^0,v)=0.
\end{equation}
Additionally, if $g_f(x(T))=0$, then the adjoint vector can be built in order to satisfy the transversality condition
\begin{equation}\label{condt2}
\bp(T)-p^0\nabla\varphi(\bx(T))\ \bot\ \ker dg_f(\bx(T)),
\end{equation}
where $dg_f(x)$ stands for the Jacobian of $g_f$ at point $x$.
%If only some components of $g_f(x(T))$ are zero then the orthogonality relation \eqref{condt2} holds for the corresponding components.
\end{thm}

The adjoint vector of the Pontryagin maximum principle can be constructed so that, up to a scaling, $(\bp(T),p^0)=(\psi,\psi^0)$ from \eqref{lagrmultrule}.
In particular, the Lagrange multiplier $\psi$ is unique (up to a multiplicative scalar) if and only if the trajectory $\bx$ admits a unique extremal lift (up to scaling).
The scalar $p^0$ is a Lagrange multiplier associated with the cost.
When $p^0<0$, the extremal is said to be \textit{normal}, and in this case, since the Lagrange multiplier is defined up to scaling, it is usual to normalize it so that $p^0=-1$. When $p^0=0$, the extremal is said to be \textit{abnormal}.
In many situations, where some qualification conditions hold, abnormal extremals do not exist in the problem under consideration, but in general it is difficult to guess whether, given some initial and final conditions, these qualification conditions hold or not.
In lack of control constraints, i.e., when $U=\R^r$, abnormal extremals project exactly onto singular trajectories. Recall that a couple $(\bx,\bu)$ defined on $[0,T]$ is said to be singular when the linearized control system along it is not controllable in time $T$; equivalently, in terms of the end-point mapping, the Fr\'echet differential $dE_{x_0,T}$ is not surjective. 
In the normal case, $\psi=\bp(T)$ coincides (up to a scaling) with the gradient of the value function (solution of the Hamilton--Jacobi equation). This point is further discussed in Section~\ref{sec-PMP-HJB}.

\begin{rem}
Generically, the solution of the optimal control problem is unique, and moreover it has a unique extremal lift. This well known fact is related to the differentiability properties of the value function (see for instance \cite{Aubin_90,cla-vin-87}, and see \cite{CannarsaSinestrari,RT2,RT1,Stefani} for results on the size of the set where the value function is differentiable).
\end{rem}

\begin{rem}
The fact that $p^0\leq 0$ is a convention due to Pontryagin, which leads to the maximum principle. The choice $p^0\geq 0$ would lead to a minimum condition, instead.
The component $p^0$ appears in the transversality condition \eqref{condt2}.
%even when we consider a cost with $\ell=0$.
Note that, if the final point $x(T)$ is let free (i.e., there is no $g_f$) then this condition leads to $\bp(T)=p^0\nabla\varphi(\bx(T))$, and then necessarily $p^0\neq 0$ and we can normalize to $p^0=-1$.
\end{rem}

\begin{rem} \label{rem44}
When there are some state constraints $g(\bx(t))\leq 0$ along the trajectory, the Pontryagin maximum principle is modified as follows. We keep the same definition for the Hamiltonian $H$.
If $\bx$ is optimal then there must exist $p^0\leq 0$, an absolutely continuous adjoint vector $\bp(\cdot)$ and a nonnegative Radon measure $\mu$, the triple $(\bp,p^0,\mu)$ being nontrivial, such that the adjoint equation is
$$
\dot\bp(t) = -\frac{\partial H}{\partial x} (t,\bx(t),\bq(t),p^0,\bu(t))
$$
almost everywhere, with
$$
\bq(t) = \bp(t) + \int_{[0,t)} \frac{\partial g}{\partial x}(\bx(s))\, d\mu(s) ,
$$
(the interval being closed when $t=T$) and the maximization condition becomes
$$
H(t,\bx(t),\bq(t),p^0,\bu(t)) = \max_{v\in U} H(t,\bx(t),\bq(t),p^0,v)
$$
almost everywhere. Finally, in the transversality condition on the final adjoint vector, one replaces $\bp(T)$ with $\bq(T)$.
Note that, taking $\bq$ as a new adjoint, we have $d\bq(t)=d\bp(t)+\frac{\partial g}{\partial x}(\bx(t))\,d\mu(t)$ and, since $\dot\bp=-\frac{\partial H}{\partial x}$ and provided $\mu$ is absolutely continuous w.r.t. Lebesgue measure, one has
$$
\dot\bq = -\bq\frac{\partial f}{\partial x}-p^0 \frac{\partial f^0}{\partial x} + \nu \frac{\partial g}{\partial x}
$$
for some nonnegative $\nu$ ($d\mu=\nu dt$). This is the formulation that one can also find in the existing literature (possibly with an opposite sign for the state constraint, and obvious changes in the previous expressions).
\end{rem}

\subsection{Shooting method}
To compute optimal trajectories thanks to the Pontryagin Maximum Principle, the first step is to make explicit the maximization condition, at least when this is possible (otherwise this can be done numerically).
A usual assumption, often satisfied, is the strict Legendre assumption, requiring that
$\frac{\partial^2H}{\partial u^2}(t,x,p,p^0,u)$
is negative definite along the reference extremal. Under this assumption, an implicit function argument gives, locally, a control expressed as a function of the state and of the adjoint.
Let us assume, in what follows, that we are in the normal case, $p^0=-1$.
Plugging the resulting expression of the control in the Hamiltonian equations, and defining the \textit{reduced (normal) Hamiltonian} by $H_r(t,x,p)=H(t,x,p,-1,u(x,p))$, we obtain that every normal extremal is solution of the reduced Hamiltonian system
\begin{equation}\label{systreduit}
\dot{\bx}(t)=\frac{\partial H_r}{\partial p}(t,\bx(t),\bp(t)),\quad \dot{\bp}(t)=-\frac{\partial H_r}{\partial x}(t,\bx(t),\bp(t)).
\end{equation}

\begin{definition}\label{defexpmapping}
Denoting by $(\bx(t,x_0,p_0),\bp(t,x_0,p_0))$ the solution of \eqref{systreduit} starting at $(x_0,p_0)$ for $t=0$, the exponential mapping is defined by $\mathrm{exp}_{x_0}(t,p_0)=\bx(t,x_0,p_0)$.
\end{definition}
\noindent The exponential mapping parametrizes the normal extremal flow. 
The abnormal extremal flow can be parametrized as well, provided an appropriate Legendre assumption holds in the abnormal case.

\begin{rem} \label{rem45}
Let us give an example where the Hessian of the Hamiltonian is degenerate: the minimal time problem for single-input control affine systems $\dot{\bx}(t)=F_0(\bx(t))+\bu(t)F_1(\bx(t))$ without constraints on controls. In that case, the maximization condition implies that the bracket $\langle \bp(t),F_1(\bx(t))\rangle$ vanishes along the extremals, and the optimal control $\bu$ is singular. To compute it, we perform two successive derivations in time of the latter relation, obtaining $\langle \bp(t),[F_0,F_1](\bx(t))\rangle=0$ and  $\langle \bp(t),[F_0,[F_0,F_1]](\bx(t))\rangle+\bu(t)\langle \bp(t),[F_1,[F_0,F_1]](\bx(t))\rangle=0$, where $[\cdot,\cdot]$ is the Lie bracket of vector fields. We thus retrieve the control as a function of $x$ and $p$, provided that
$$ \langle \bp(t),[F_1,[F_0,F_1]](\bx(t))\rangle > 0, $$
which is the so-called strong generalized Legendre--Clebsch condition (see, e.g., \cite{BonnardChyba}).
Actually, under generic conditions on the vector fields, the above computation can always be performed (see \cite{CJT2006,CJT2008}). See Goddard example in this section for an example of this computation.
\end{rem}

\begin{rem}
When an abnormal flow can be well defined, we then have to deal with two extremal flows (and two exponential mappings).
In general, however, the abnormal flow ``does not fill much space''. For example, in \cite{Agrachev2009,RT1} it is proved that for control-affine systems without drift (satisfying the H\"ormander assumption), with quadratic cost, the image of the abnormal exponential mapping has an empty interior in the state space, and is even of zero Lebesgue measure under some additional assumptions.
\end{rem}

\begin{rem}
The Pontryagin Maximum Principle is a first-order necessary condition for optimality, asserting that if a trajectory is optimal then it should be sought among projections of extremals joining the initial point to the final target. This yields the shooting method that is described hereafter.
But, before coming to the description of that method, it is interesting to observe that, conversely, the projection of a given extremal may not be (locally or globally) optimal. This important observation has led to develop second-order optimality conditions in optimal control.

In terms of the Lagrangian (in the simplified setting), considering the intrinsic second order derivative $Q_T$ of the Lagrangian, given by
$$
Q_T=\frac{\partial^2 L_T}{\partial^2 u}(\bu,\psi,\psi^0)_{\vert \ker \frac{\partial L_T}{\partial u} },
$$
a second-order necessary condition for optimality is that $Q_T$ be nonpositive, and a second-order sufficient condition for \textit{local} optimality is that $Q_T$ be negative definite. Such conditions admit a number of generalizations for optimal control problems involving control and/or state constraints. 
It happens that, given a fixed extremal starting at $(x_0,p_0)$, in the simplified context and under appropriate generic assumptions, the quadratic form $Q_T$ is not degenerate (i.e., its kernel is trivial) if and only if the exponential mapping $\mathrm{exp}_{x_0}(t_c,\cdot)$ is not an immersion at $p_0$ (that is, its differential is not injective). This result, coming from symplectic considerations and generalizing the Riccati theory (see \cite{AgrachevSachkov,BFT}), yields to algorithms for computing the first conjugate time along a given extremal (see \cite{BCT2007} for a survey). By definition, the first conjugate time along an extremal is the first time $t_c$ at which the quadratic form $Q_{t_c}$ has a nontrivial kernel. This means that the trajectory $x(\cdot)$ under consideration is locally optimal (in $L^\infty$ topology) on $[0,t]$ if and only if $t<t_c$.
Computing a first conjugate time amounts to computing the first time at which some determinant along the extremal vanishes. More generally, the fact that the exponential mapping is not an immersion can be translated in terms of Jacobi fields, like in Riemannian geometry. 
\end{rem}

Let us now describe the contents of the shooting method.
The Pontryagin Maximum Principle states that every optimal trajectory is the projection of an extremal. After making explicit the maximization condition, the problem is reduced (for instance, in the normal case) to an extremal system of the form $\dot{\bz}(t) = F(t,\bz(t))$, where $\bz(t)=(\bx(t),\bp(t))$, and initial, final, transversality conditions, are written as $R(\bz(0),\bz(T))=0$. We thus end up with a boundary value problem (BVP) of the form
\begin{equation}\label{pbvallim}
\dot{\bz}(t)=F(t,\bz(t)),\quad R(\bz(0),\bz(T))=0.
\end{equation}
Denote by $\bz(t,z_0)$ the solution of the Cauchy problem $\dot{\bz}(t)=F(t,\bz(t))$, $\bz(0)=z_0$,
and set $G(z_0) = R(z_0,\bz(T,z_0))$. The boundary value problem \eqref{pbvallim} is then equivalent to solving $G(z_0)=0$, i.e., to finding a zero of the function $G$. 
By definition, the (single) shooting method consists of implementing a Newton-like method to find a zero of $G$.

The feasibility of the shooting method relies on the fact that the Jacobian of the mapping $G$ is nonzero. According to the above remark, in the simplified case, this determinant is nonzero, i.e., the (single) shooting method is well-posed at time $T$, locally around $p_0$, if and only if the exponential mapping $\mathrm{exp}_{x_0}(t,\cdot)$ is an immersion at $p_0$, that is, if and only if $T$ is not a conjugate time.
Although this result admits generalizations to a number of contexts (see, e.g., \cite{BonnansHermant2007,BonnansHermant2009}), there still misses a complete conjugate time theory involving state and control constraints, in which the trajectories may have bang, singular, boundary arcs.

The single shooting method can be refined to the multiple shooting method, in which one may add a number of intermediate nodes, thus incorporating new (matching) conditions in the shooting function $G$. This can be useful for instance to face numerical instability issues, or to implement the shooting method in bang-bang situations where one knows in advance the number of switchings.
Efficient shooting methods are available in the \texttt{HamPath} package%
\footnote{\href{https://www.hampath.org}{hampath.org}}, now encapsulated in the \texttt{Python} package \texttt{nutopy} of the \texttt{ct: control toolbox} project.\footnote{\href{https://ct.gitlabpages.inria.fr/gallery}{ct.gitlabpages.inria.fr/gallery}}
These open-source packages also contain implementations of conjugate time computations and of several homotopy routines that are particularly useful in a number of contexts (see \cite{Trelat_JOTA2012} for a survey on the use of continuation methods in optimal control).

\begin{rem}
Numerically, the shooting method is the combination of a numerical integration of a differential equation with a Newton method for finding a zero of a map (the shooting function). It thus inherits of the main features of a Newton method: when it converges, the convergence is extremely fast and the result is very accurate. However, it may be difficult to initialize successfully: finding a good initial guess for $z_0$, in the above notations, may be a real challenge. To face with this difficulty, several possible remedies are known, such as the following, surveyed in \cite{Trelat_JOTA2012} (see also the references therein):
\begin{itemize}
    \item Since direct methods are less sensitive to the initialization, it is often successful to first run a direct approach (even with a quite rough mesh) so that, if can obtain its convergence, then the corresponding optimal solution and Lagrange multiplier can be used as an approximation of the searched extremal. 
    \item Continuation and homotopy methods can be combined with the shooting method (and also, by the way, with direct methods): when a given problem happens to be difficult to solve, or quite ill-posed, because of some too restricted parameters or because of too constraining terms in the dynamics, one can try to relax the optimal control problem by adding some continuation parameters in front of those terms, then run a series of shooting methods with the continuation parameters ranging iteratively from $0$ to $1$ (with adaptive steps if necessary). 
    \item Geometric control gives useful information on the local or global structure of the optimal controls. For instance, one can guess in advance the number of switchings in a bang-bang strategy under some appropriate assumptions. This knowledge can then be combined with the Pontryagin maximum principle in order to drastically reduce the complexity of the shooting problem. 
\end{itemize}
Much more could be said on these classical issues, but since they have already been surveyed in \cite{Trelat_JOTA2012} we do not elaborate more. We next describe another powerful remedy that has emerged recently, although it relies on an old concept.
\end{rem}

\subsection{Turnpike property} \label{s43}
Assume that $f$ and $\ell$ do not depend on time.
In few words, the turnpike property stipulates that, for optimal control problems in large time, under mild assumptions it is expected that the optimal solution, the optimal control and the associated adjoint remain essentially close to static values, except at the beginning and at the end of the time interval. Moreover, these static values correspond to the optimal solution of a static optimization problem.
The idea is very easy to understand. When the time $T$ is large, setting $\varepsilon=1/T$ and $\tau=t/T=\varepsilon t$, the optimal control problem consists of determining a trajectory $y(\tau)=\bx(t)$ and a control $v(\tau)=\bu(t)$ solution of $\varepsilon y'(\tau)=f(y(\tau),v(\tau))$ and minimizing the cost $\int_0^1 \ell(y(\tau),v(\tau))\, d\tau$, under various constraints. At the formal level, we see that, when $\varepsilon\rightarrow 0$, at the limit we find the static optimization problem consisting of minimizing $\ell(y,v)$ under the constraint $f(y,v)=0$, i.e., under the constraint of being an equilibrium of the controlled dynamics. 

To give a more precise insight, let us establish the so-called exponential turnpike phenomenon in the linear-quadratic case. 
Let $x_0, x_1, \hat{x}\in\R^d$ and $\hat{u}\in\R^r$ be fixed.
We consider the optimal control problem in fixed final time $T>0$:
\begin{equation}\label{pbLQ}
\begin{split}
& \min \int_0^T \left( (\bx(t)-\hat{x})^\top Q(\bx(t)-\hat{x}) + (\bu(t)-\hat{u})^\top R (\bu(t)-\hat{u}) \right) dt\\
& \hspace*{5mm} 
\dot \bx(t) = A\bx(t) + B\bu(t) \\
& \hspace*{5mm} 
\bx(0)=x_0,\ \bx(T)=x_1 \\
%& \min \int_0^T \left( x(t)^\top Qx(t) + u(t)^\top U u(t) \right) dt
\end{split}
\end{equation}
where $\bx(t)\in \R^d$ and $\bu(t)\in \R^r$, and where $Q$ et $R$ are real-valued symmetric positive definite matrices.
By strict convexity, there exists a unique optimal solution $(\bx_T,\bu_T(\cdot))$ of \eqref{pbLQ}. We assume that the pair $(A,B)$ satisfies the Kalman condition.
The Hamiltonian of the problem is $H(x,p,p^0,u) = \langle p,Ax\rangle+\langle p,Bu\rangle+p^0((x-\hat{x})^\top Q(x-\hat{x})+(u-\hat{u})^\top R(u-\hat{u}))$.
Let us prove that $p^0\neq 0$.
By contradiction, if $p^0=0$ then the condition $\frac{\partial H}{\partial u}=0$ yields $\langle \bp(t),B\rangle=0$, and by successive derivations and using the fact that $\dot \bp(t)=-A^\top \bp(t)$, we obtain $\langle \bp(t),A^kB\rangle=0$, which raises a contradiction with the Kalman condition since $\bp(t)\neq 0$.
We choose then to normalize the adjoint so that $p^0=-1/2$. The condition $\frac{\partial H}{\partial u}=0$ yields $\bu_T(t)=\hat{u}+R^{-1}B^\top \bp_T(t)$, and the extremal system is
\begin{align*}
\dot \bx_T(t) &= A \bx_T(t) + BR^{-1}B^\top \bp_T(t) + B\hat{u} \\
\dot \bp_T(t) &= Q\bx_T(t) - A^\top \bp_T(t) - Q\hat{x},
\end{align*}
i.e.,
$$
\frac{d}{dt} \begin{pmatrix} \bx_T(t)\\ \bp_T(t)\end{pmatrix}
= M \begin{pmatrix} \bx_T(t)\\ \bp_T(t)\end{pmatrix} + \begin{pmatrix} B\hat{u}\\ -Q\hat{x}\end{pmatrix},
$$
where
$$
M = \begin{pmatrix}
A & BR^{-1}B^\top \\
Q & -A^\top.
\end{pmatrix}.
$$
Besides, the static optimization problem is
\begin{equation*}%\label{staticpbLQ} 
%\min_{{(x,u)\in\R^d\times\R^r}\atop{Ax+Bu=0}} \left( (x-\hat{x})^\top Q(x-\hat{x}) + %(u-\hat{u})^\top U(u-\hat{u}) \right).
\min_{(x,u)\in\R^d\times\R^r,\ Ax+Bu=0} \left( (x-\hat{x})^\top Q(x-\hat{x}) + (u-\hat{u})^\top R(u-\hat{u}) \right).
\end{equation*} 
This strictly convex problem has a unique solution $(\bar x,\bar u)$, associated with a normal Lagrange multiplier $(\bar p,-1)$ (the problem is qualified because $\ker(A^\top)\cap\ker(B^\top)=\{0\}$, as a consequence of the Kalman condition). According to the Lagrange multiplier rule, which is here a necessary and sufficient condition for optimality, there exists $\bar p\in\R^d\setminus\{0\}$ such that $\bar u = \hat{u} + R^{-1}B^*\bar p$ and
\begin{equation}\label{extrLQ}
\begin{split}
A\bar x + BR^{-1}B^*\bar p + B\hat{u} &= 0 \\
Q\bar x - A^*\bar p - Q \hat{x} &= 0,
\end{split}
\end{equation}
i.e.,
$$
M \begin{pmatrix} \bar x\\ \bar p\end{pmatrix} + \begin{pmatrix} B\hat{u}\\ -Q\hat{x}\end{pmatrix} = \begin{pmatrix} 0\\ 0\end{pmatrix}.
$$
We have the following exponential turnpike property.

\begin{prop}
There exist constants $C>0$ and $\nu>0$, not depending on $T$, such that
\begin{equation}\label{expturnLQ}
\Vert \bx_T(t)-\bar x\Vert + \Vert \bu_T(t)-\bar u\Vert + \Vert \bp_T(t)-\bar p\Vert \leq C ( e^{-\nu t} + e^{-\nu (T-t)} )\qquad \forall t\in[0,T] .
%\Vert x_T(t)\Vert + \Vert u_T(t)\Vert + \Vert p_T(t)\Vert \leq C ( e^{-\nu t} + e^{-\nu (T-t)} )\qquad \forall t\in[0,T] .
\end{equation}
\end{prop}

\begin{proof}
Using the above optimality systems, we have
$$
\frac{d}{dt} \begin{pmatrix} \bx_T(t)-\bar x\\ \bp_T(t)-\bar p\end{pmatrix}
= M \begin{pmatrix} \bx_T(t)-\bar x\\ \bp_T(t)-\bar p\end{pmatrix} 
$$
In order to prove \eqref{expturnLQ}, the crucial observation is that the matrix $M$, which is Hamiltonian, is hyperbolic, i.e., all its eigenvalues have a nonzero real part (actually, the number of unstable modes is equal to the number of stable modes).
To prove this hyperbolicity property, we start by noting that, as a consequence of the Kalman condition on $(A,B)$, we have
\begin{equation}\label{hautus}
\ker(A^\top- \xi I_n)\cap\ker(B^\top)=\{0\}\qquad\forall \xi\in\mathbb{C} .
\end{equation}
It follows that the matrix $M$ has no purely imaginary eigenvalue. Indeed, let $(z_1,z_2)\in\mathbb{C}^n\times\mathbb{C}^n$and let $\mu\in\R$ be such that $(M-i\mu)\begin{pmatrix}z_1\\z_2\end{pmatrix}=0$. Then,
\begin{align*}
(A-i\mu)z_1+BR^{-1}B^\top z_2&=0\\
Qz_1-(A^\top+i\mu) z_2&=0,
\end{align*}
hence $z_1=Q^{-1}(A^\top+i\mu)z_2$ and thus $(A-i\mu)Q^{-1}(A^\top+i\mu)z_1+BR^{-1}B^\top z_2=0$. Multiplying to the left by $\bar z_2^\top$, we obtain $\Vert Q^{-1/2}(A^\top+i\mu)z_2\Vert^2+\Vert R^{-1/2}B^\top z_2\Vert^2=0$ and hence $(A^\top+i\mu)z_2=0$ and $B^\top z_2=0$. We infer that $z_2=0$ by using \eqref{hautus}. The claim is proved.

Since $M$ is hyperbolic, there exists a real-valued square matrix $P$ of size $2n$ such that
$$
P^{-1}MP = \begin{pmatrix}
M_1 & 0 \\ 0 & M_2
\end{pmatrix}
$$
where all eigenvalues of the matrix $M_1$ have a negative real part, and all eigenvalues of $M_2$ have a positive real part.
Now, setting $\begin{pmatrix} \bx_T-\bar x\\ \bp_T-\bar p\end{pmatrix} = P \begin{pmatrix} v\\ w\end{pmatrix}$, the extremal system gives
$$
\frac{d}{dt} \begin{pmatrix} v(t)\\ w(t)\end{pmatrix} = P^{-1} M P \begin{pmatrix} v(t)\\ w(t)\end{pmatrix} 
$$
so $\dot v(t) = M_1 v(t)$ and $w(t)=M_2w(t)$. All eigenvalues of $M_1$ have a negative real part, hence there exist constants $C_1>0$ and $\nu_1>0$, not depending on $T$, such that $\Vert v(t)\Vert \leq C_1 e^{-\nu_1 t}$. For $w(t)$, we reverse time and we apply the same argument, hence $\Vert w(t)\Vert \leq C_2 e^{-\nu_2 (T-t)}$. Finally, \eqref{expturnLQ} follows by noting that $\bx_T(t)$ et $\bp_T(t)$ are linear combinations of $v(t)$ and $w(t)$.
\end{proof}

The exponential turnpike property \eqref{expturnLQ} says that, except near $t=0$ and $t=T$, the optimal state, control and adjoint are exponentially close to static values, themselves corresponding to the solution of the associated static optimization problem.
In \cite{TrelatZuazua_JDE2015}, the proof of the above proposition is a bit different and relies on the use of the Riccati theory: actually, the matrix $P$ is built by considering the minimal and maximal solutions of the Riccati algebraic equation, which gives an interpretation of the constants $C$ and $\nu$ in terms of these matrices. Anyway, the argument remains very easy and it withstands a number of generalizations: to infinite dimension, to nonlinear dynamics, and also to situations where the turnpike is not restricted, as above, to a singleton but may even consist of a nontrivial set of trajectories (an example being the periodic turnpike).
We refer to the chapter \cite{FaulwasserGrune} of the Volume 1 of the present Handbook, for a recent survey on the turnpike property in optimal control, containing a number of references and commenting also on the important related notion of dissipativity.
What we want to point out here is that, when the turnpike property is satisfied, we can use it to successfully initialize a variant of the shooting method.

\begin{rem}[Variant of the shooting method] \label{rem401}
%Variant of the shooting method.
In the turnpike context, we know that, in the middle of the trajectory, $\bx_T(T/2)$ and $\bp_T(T/2)$ are exponentially close to $\bar x$ and $\bar p$. In such conditions, if this is feasible, it is convenient to first compute the solution of the static problem, then to implement a variant of the shooting method by initializing it ``in the middle'', as follows.
Using the notations of the previous section, the unknown is now $\bz(T/2)\in\R^{2d}$, which is initialized at $(\bar x,\bar p)$, the steady-state solution of the static optimal control problem. % \eqref{staticpb}. 
Then, integrating backwards the extremal system over $[0,T/2]$, we compute $\bz(0)$; integrating forward the extremal system over $[T/2,T]$, we compute $\bz(T)$. Finally, the unknown $\bz(T/2)$ is tuned so that $G(\bz(0),\bz(T))=0$, thanks to a Newton method.
It has been observed in \cite{TrelatZuazua_JDE2015} that this variant of the single shooting method is very efficient.
\end{rem}

\subsection{Solving the Zermelo problem by the shooting method}\label{sec.3.4}
We consider the navigation problem of Zermelo presented in Section~\ref{s2}, with $a=0$ and $b=1$. We take
$V_{\max}=1$, $y_f=(20,1)$, $h(y_2)=3+0.2y_2(1-y_2)$,
where the target is the point $y_f$ (a slightly simpler problem than the general previously presented), while the initial point is $y_0=(0,0)$.
We treat the minimum time case so that the velocity can be set to $V_{\max}$ and the only control is the angle. An obstacle is inserted along the unconstrained optimal trajectory:

$$
\frac{(y_1-y_{1,f}/2)^2}{a_1^2}+\frac{(y_2-y_{2,f}/2.5)^2}{a_2^2} \leq 1
$$
with $a_1=2$ and $a_2=0.1$.
We use a  logarithmic barrier to penalize internally the state constraint and we consider the augmented cost
$$
T - \alpha \int_0^1 \ln \left( \frac{(y_1(t)-y_{1,f}/2)^2}{a_1^2}+\frac{(y_2(t)-y_{2,f}/2.5)^2}{a_2^2} -1 \right) \, dt \to \min.
$$
For a detailed study of internal penalization, we refer to \cite{chaplais-2014a}.
For the computation below, $\alpha$ is set to $10^{-3}$.
To obtain more accurate results, one could perform a numerical continuation on the penalty parameter $\alpha$ as is customary for interior methods \cite{nocedal-2006a}.
%A finer approach would be to perform a continuation on $\alpha$
We apply the maximum principle which leads to the (normal) maximized Hamiltonian
\begin{equation} \label{eq401}
  H(y,p) = V_{\max}\sqrt{p_1^2+p_2^2} + p_1 h(y_2) + \alpha \ln \left( \frac{(y_1-y_{1,f}/2)^2}{a_1^2}+\frac{(y_2-y_{2,f}/2.5)^2}{a_2^2} -1 \right) -1.
\end{equation}
The shooting problem then consists in finding the initial value $\bp(0) \in \mathbb{R}^2$ so that, integrating the flow of the maximized Hamiltonian, the target $y_f$ is reached. Moreover, the equation $H=0$ is added to accomodate the fact that the final time $T$ is free. Depending on the initialization of the shooting method, two solutions are obtained, see Figure~\ref{fig30}. One of them is clearly a local minimizer as one can check comparing the numerical final times. An interesting approach in such a situation is to rely on an HJB solver (see Section \ref{s5}) to retrieve an initial guess of the adjoint state (and of the final time) that avoid strict local minima. It can indeed be checked numerically that the HJB solution on the previous data allows shooting to converge towards the global minimizer on the internally penalized problem.
% added comment
This illustrates how one can leverage the strengths of both HJB and shooting: while HJB might not be able to produce a high precision numerical control, it will allow to select a proxy for the adjoint of the global minimum and for the associated value, good enough to ensure convergence of the shooting method (the difficult issue for indirect methods) towards a precise numerical solution.

\begin{rem}
A refinement of this computation could involve a continuation (or differential homotopy) on the size of the obstacle. Path following methods would indeed allow to track the two branches associated with the global and strictly local minimum, then to decide which one is optimal for each size of the obstacle by comparing the associated final times.
\end{rem}

\begin{figure}
    \centering
    \includegraphics[width=0.7\textwidth]{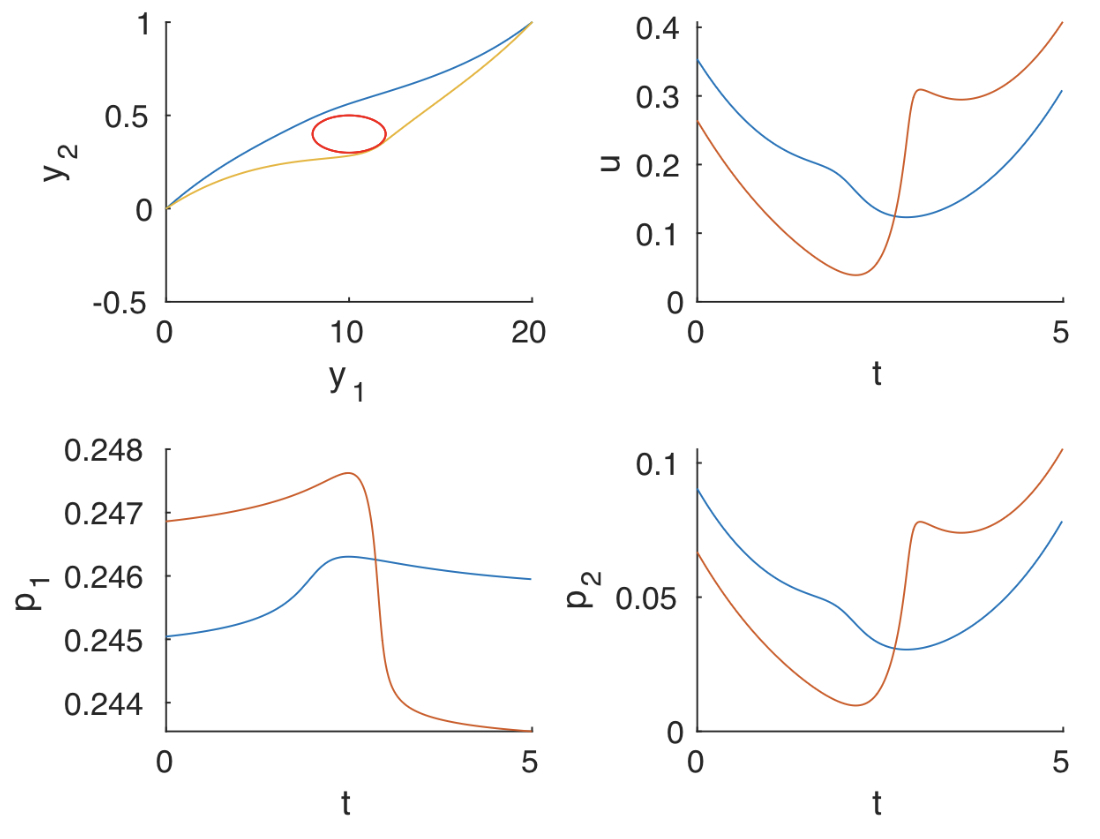}
    \caption{\small Zermelo problem with obstacle: result of the shooting method. The state, costate and control are portrayed for the two solutions (penalisation parameter $\alpha=10^{-3}$). The first solution (in blue) is slightly shorter ($T \simeq 4.98$) than the second one (yellow, $T \simeq 4.99$).}
    \label{fig30}
\end{figure}

Another relevant observation on this problem is related to the turnpike property described in Section~\ref{s43}. Although, the problem is not linear-quadratic, one can easily guess the role played by this property for ``large'' (in terms of $y_1$) final conditions. We keep the same data for $V_\mathrm{max}$ and for the current, remove the obstacle, and now target $y_f=(200,1)$. For such a distant target, one expects the optimal control to use the ``fast lane'' so that, most of the time, $y_2$ remains close to $1/2$ where the current $h$ is maximum. So the guess would be that, for a large part of the trajectory, $y_2 \simeq 1/2$, $u_2=\dot{y}_2 \simeq 0$, $u_1 \simeq 1$, and $T \simeq (y_{1,f}-y_{1,0})/(1+h(1/2))$. The maximized (normal) Hamiltonian is (compare \eqref{eq401})
$$ H(y,p) = V_{\max}\sqrt{p_1^2+p_2^2} + p_1 h(y_2) - 1, $$
so that, on $H=0$, $p \simeq (1/(1+h(1/2,0))$. These approximations can be used to initialize the variant of the shooting method described in Remark~\ref{rem401}. It is straightforwardly implemented in \texttt{Julia} according to\footnote{The code is available and executable online at \href{https://ct.gitlabpages.inria.fr/gallery}{ct.gitlabpages.inria.fr/gallery}}

\begin{verbatim}
# Regular maximized Hamiltonians and associated flow
H(y, p, u) = -1.0 + Vmax*p'*u + p[1]*h(y[2])
ur(p) = p / sqrt(p[1]^2+p[2]^2)
Hr(y, p) = H(y, p, ur(p))
fr = Flow(Hr)

# Shooting function
function shoot(y1, p1, tf)

    yy0, p0 = fr(tf/2.0, y1, p1, t0)
    yyf, pf = fr(tf/2.0, y1, p1, tf)
    s = zeros(eltype(y1), 5)
    s[1:2] =  yy0-y0
    s[3:4] = (yyf-yf) ./ yf
    s[5] = Hr(y1, p1)

    return s

end
\end{verbatim}

\noindent The structure of the optimal solution is as expected, see Figure~\ref{fig30t}.

\begin{figure}
    \centering
    \includegraphics[width=0.7\textwidth]{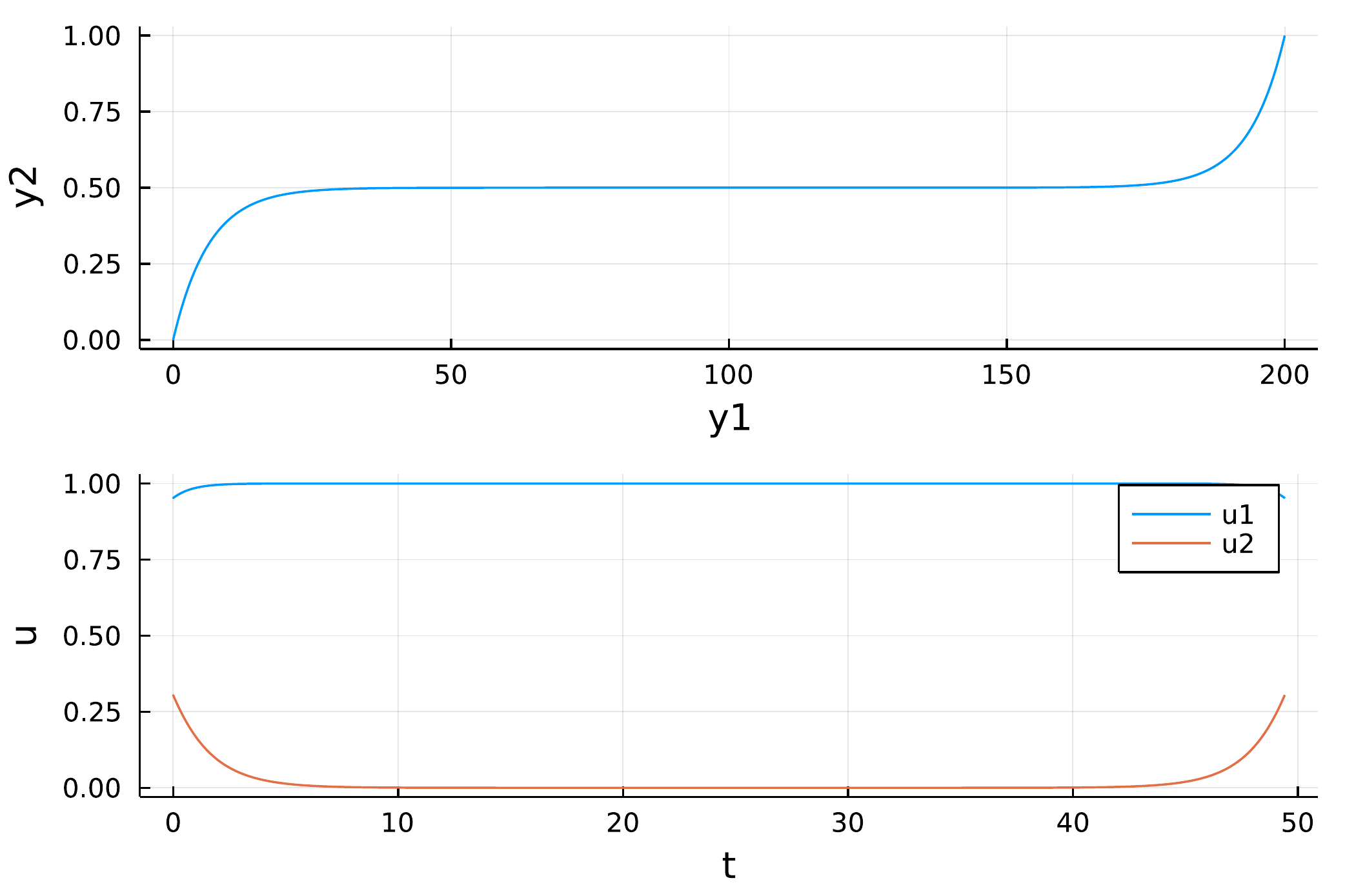}
    \caption{\small Zermelo problem, turnpike property: result of the variant of the shooting method. The state and control are portrayed, illustrating that most of the time is spent close to $y_2=1/2$ where the current is maximum. Accordingly, $u_2 \simeq 0$ on the corresponding control subarc.}
    \label{fig30t}
\end{figure}

\begin{rem}
Further analysis is required to tackle the case of ``strong'' currents. When there exist zones where the drift cannot be compensated by the control, abnormal trajectories come into play and discontinuities of the value function (minimum time) are observed. See \cite{wembe-2022a} for a detailed treatment of such cases. 
\end{rem}

\subsection{Solving the Goddard problem by the shooting method} \label{Section.Goddard_Tir}
We chain the result of Section \ref{s32}, obtained by a direct method, with a shooting method. The previous numerical solution, although not very accurate, has captured the structure of the solution (more precisely, of what one may hope to be a local minimizer, at least). This knowledge allows us to define the appropriate shooting function, namely one that combines four arcs (bang-singular-constrained-bang), each one being the flow of a relevant Hamiltonian. On can then leverage the accurate knowledge (including the Hamiltonian character) gained on each subarc by means of Pontryagin maximum principle to obtain a very accurate solver. In general, knowing the structure alone is not enough to actually solve the problem. One also needs a good initial guess for the zero of the shooting function, which turns to be also provided by the previous direct solver. In order to set up our shooting, we rely on the maximum principle and observe that the control can be either bang, singular of boundary. Indeed, system \eqref{eq101}-\eqref{eq103} is affine in the control and can be written according to
$$ \dot{\bx}(t) = F_0(\bx(t)) + \bu(t)F_1(\bx(t)),\quad \bu(t) \in [0,1], $$
with $\bx=(\br,\bv,\bbm) \in \mathbb{R}^3$ and vector fields that we shall define in \texttt{Julia}, completing the code started Section~\ref{s32}:
\begin{verbatim}
# Dynamics
function F0(x)
    r, v, m = x
    D = Cd * v^2 * exp(-β*(r-1.0))
    F = [ v, -D/m-1.0/r^2, 0.0 ]
    return F
end

function F1(x)
    r, v, m = x
    F = [ 0.0, Tmax/m, -b*Tmax ]
    return F
end
\end{verbatim}
In order to deal with the state constraint $g(\bx(t)) := v_{\max}-\bv(t) \geq 0$,
we follow Remark~\ref{rem44} (note that we use the opposite sign for the constraint) and introduce the Hamiltonian (where the constraint has been directly adjoined, see \cite{hartl-1995a})
$$ H(x,p,u,\mu) = H_0(x, p) + u H_1(x, p) + \mu g(x), $$
where $H_i(x,p):=\langle p,F_i(x) \rangle$ are the Hamiltonian lifts of the aforementioned fields. The maximization condition implies that $u$ is bang ($0$ or $1$) whenever $H_1$ is not zero. Besides, whenever the state constraint is not active, the associated non-negative multiplier $\mu$ vanishes (complementarity condition). As a result, bang arcs are obtained by computing the flow of either $H_0$ (case $H_1<0$, $u=0$) or $H_0+H_1$ (case $H_1>0$, $u=1$). Conversely, when $H_1$ vanishes identically, assuming the state constraint is not active, as indicated in Remark~\ref{rem45} one can differentiate a.e. two times $H_1$ to retrieve the singular control provided the length three Poisson bracket $H_{101}:=\{H_1,\{H_0,H_1\}\}$ is not zero (singular of \emph{order one}; see, e.g., \cite{BonnardChyba}):
$$ u_s(x,p) = -\frac{H_{001}}{H_{101}} \cdot $$
(Same notation used for $H_{001}$.) Plugging this dynamic feedback control into the original Hamiltonian (with $\mu=0$) defines the singular Hamiltonian
$$ H_s(x,p) := H(x,p,u_s(x,p),\mu=0) $$
whose flow coincides with the extremal flow on 
$\Sigma' := \{H_1 = H_{01} = 0\}$
(see, e.g., \cite{AgrachevSachkov}).
Along a boundary arc where the state constraint is activated, if the control is interior ($\bu(t) \in (0,1)$), $H_1$ must also vanish. Moreover, if the constraint is of order one (that is if the control appears when the equality $g(\bx(t))=0$ is differentiated once) which gives in this case (Lie derivative of the constraint along $F_1$)
$ F_1\cdot g \neq 0$,
the interiority of the boundary control implies that the contact with the constraint is transverse at the exit time $t_3$ (where, in our case, a bang arc $u=0$ is joined):
$\dot{g}(t_3+) \neq 0$, and there is no jump on the adjoint \cite{hartl-1995a}. Denoting $t_2$ the entry point, on $(t_1,t_2)$ one retrieves the boundary control by differentiating once $g(\bx(t))=0$ as
$0 = \dot{g} = F_0\cdot g+\bu F_1\cdot g$,
which implies that
$u_b(x) = -\frac{F_0\cdot g}{F_1\cdot g}$
under the two previous assumptions. Similarly, differentiating once $H_1=0$ allows to compute the multiplier $\mu$ as
$0 = \dot{H}_1 = \{H_0+\bu H_1+\mu g,H_1\} = H_{01}-\mu F_1\cdot g$,
so that
$\mu_b(x,p) =  \frac{H_{01}}{F_1\cdot g} $.
The relevant flow is an integral curve of the boundary Hamiltonian
$$ H_b(x,p) := H(x,p,u_b(x),\mu_b(x,p)). $$
All in all, the symbolic-numeric framework allows to define everything in terms of the vector fields and of the constraint:
\begin{verbatim}
# Computation of singular control of order 1
H0(x, p) = p' * F0(x)
H1(x, p) = p' * F1(x)
H01 = Poisson(H0, H1)
H001 = Poisson(H0, H01)
H101 = Poisson(H1, H01)
us(x, p) = -H001(x, p)/H101(x, p)

# Computation of boundary control
g(x) = vmax-x[2] # vmax - v ≥ 0
ub(x) = -Lie(F0, g)(x) / Lie(F1, g)(x)
μb(x, p) = H01(x, p) / Lie(F1, g)(x)

# Hamiltonians (regular, singular, boundary) and associated flows
H(x, p, u, μ) = H0(x, p) + u*H1(x, p) + μ*g(x)
Hr(x, p) = H(x, p, 1.0, 0.0)
Hs(x, p) = H(x, p, us(x, p), 0.0)
Hb(x, p) = H(x, p, ub(x), μb(x, p))
\end{verbatim}
Then, to integrate the Hamiltonians to obtain the flows and define the shooting function in terms of the initial adjoint $p_0$, the entry point of the singular arc $t_1$, the entry point of the boundary arc $t_2$, the exit point $t_3$, and the free final time $T$:
\begin{verbatim}
f0 = Flow(H0)
fr = Flow(Hr)
fs = Flow(Hs)
fb = Flow(Hb)

# Shooting function
function shoot(p0, t1, t2, t3, tf)

    x1, p1 = fr(t0, x0, p0, t1)
    x2, p2 = fs(t1, x1, p1, t2)
    x3, p3 = fb(t2, x2, p2, t3)
    xf, pf = f0(t3, x3, p3, tf)
    s = zeros(eltype(p0), 7)
    s[1:2] = pf[1:2] - [ 1.0, 0.0 ]
    s[3] = xf[3] - mf
    s[4] = H1(x1, p1)
    s[5] = H01(x1, p1)
    s[6] = g(x2)
    s[7] = H0(xf, pf)

    return s
end
\end{verbatim}
It is straightforward to retrieve initial guesses for these unknowns from the direct code solution, most notably by retrieving the Lagrange multipliers from the optimizer through \texttt{JuMP} interface (note the minus sign to accomodate the convention on the adjoint state in contrast with the one on Lagrange multipliers):
\begin{verbatim}
p = -[ [ dual(con_dr[i]), dual(con_dv[i]), dual(con_dm[i]) ] for i in 1:N ]
\end{verbatim}
Moreover, in view of these results,
the constraint on the final mass is assumed to be active (which is in accordance with the final zero bang arc). Automatic differentiation can also be used to compute the derivative of the shooting function (several \texttt{Julia} backends to do this are available and include differentiating properly calls to ODE solvers). See Figures \ref{fig41} and \ref{fig42} for the associated numerical simulations, while the code itself is available and executable online.\footnote{\href{https://ct.gitlabpages.inria.fr/gallery}%
{ct.gitlabpages.inria.fr/gallery}}

\begin{figure}
    \centering
    \includegraphics[width=0.7\textwidth]{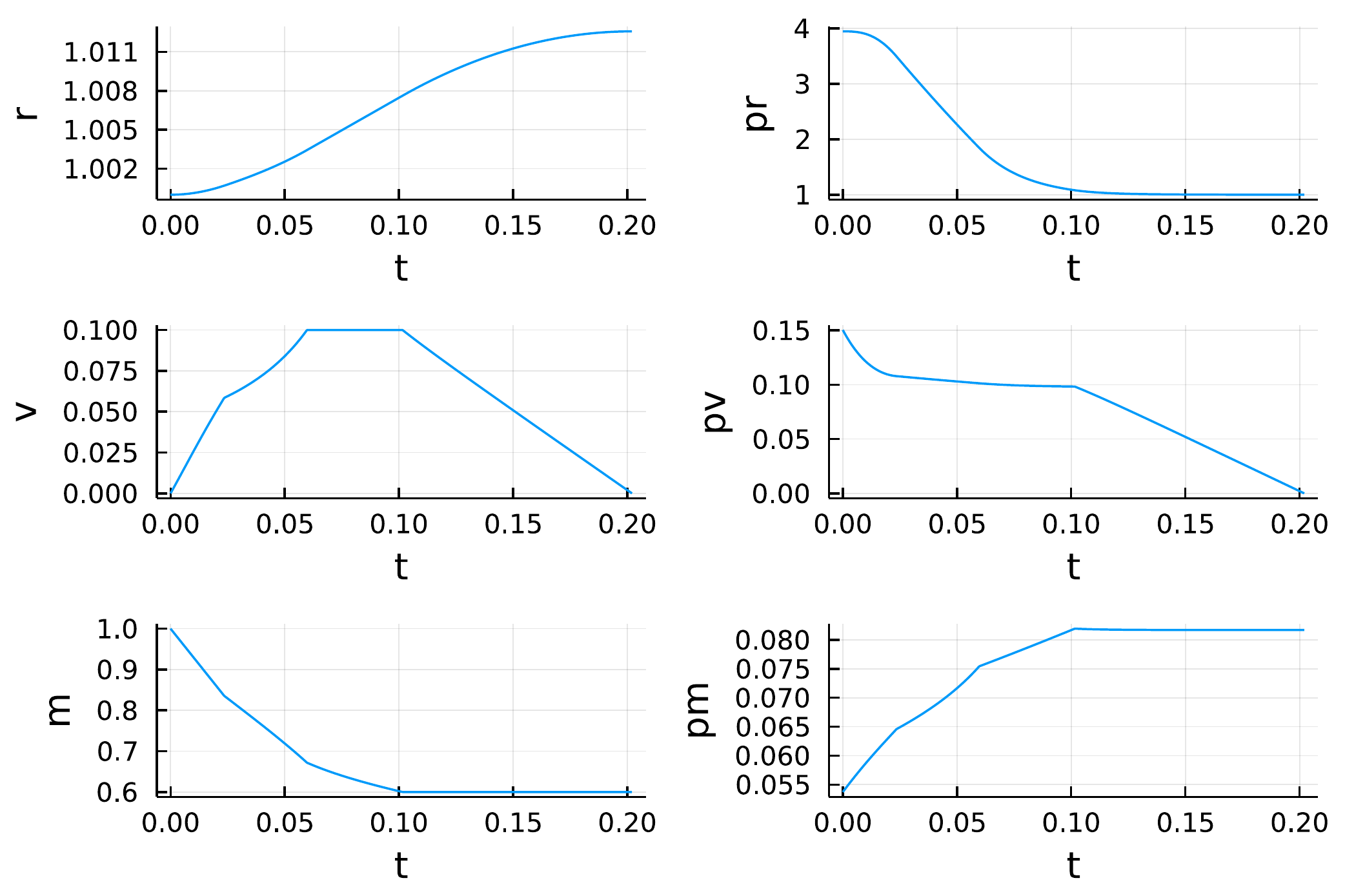}
    \caption{\small Goddard problem: result of the shooting method,
    states and costates. The values are in line with those obtained by the direct method. They are computed by concatenating the integration of the four Hamiltonian flows involved.}
    \label{fig41}
\end{figure}

\begin{figure}
    \centering
    \includegraphics[width=0.7\textwidth]{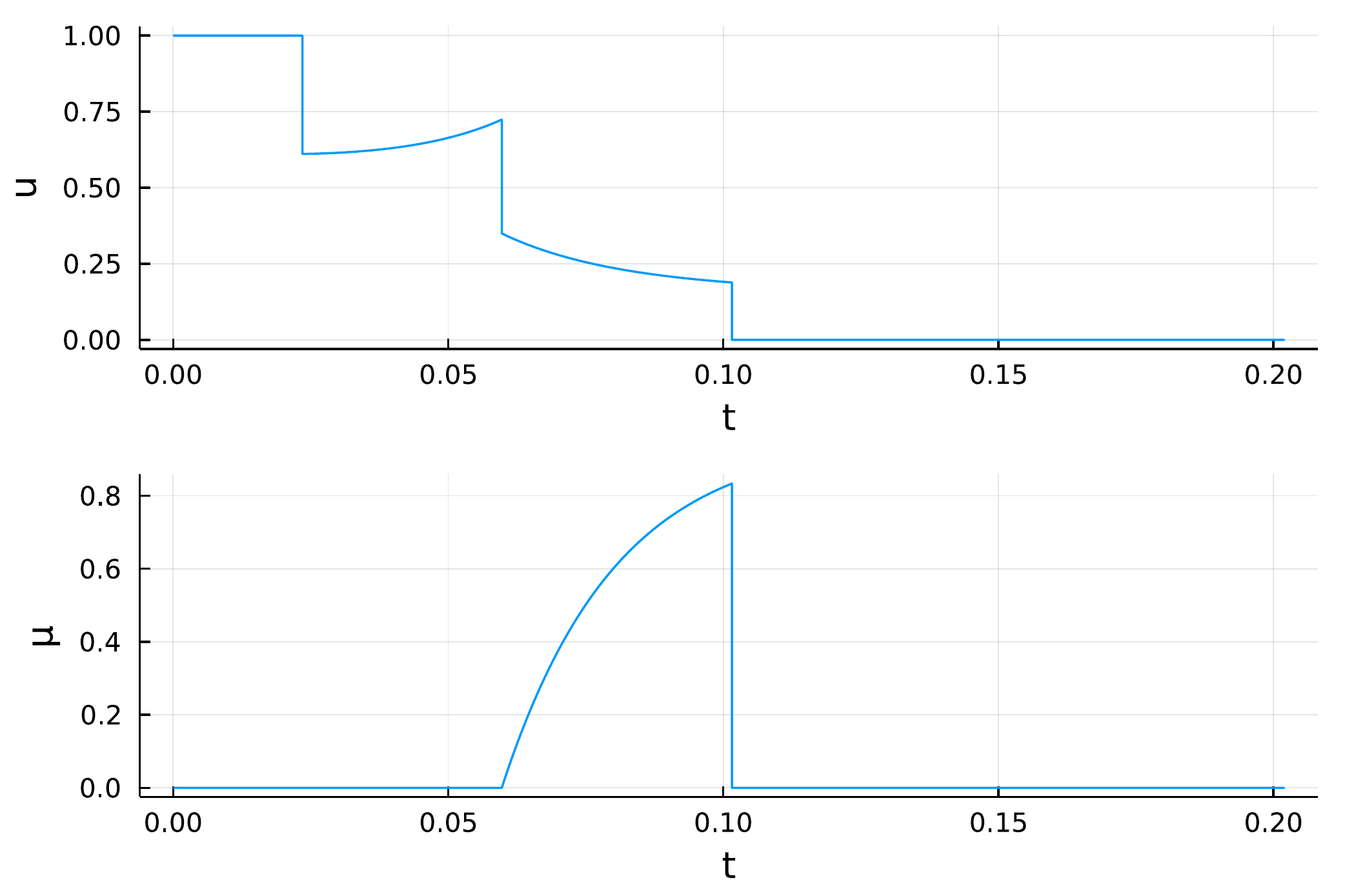}
    \caption{\small Goddard problem: result of the shooting method, control and state constraint multiplier. The bang-singular-constrained-bang structure of the control is very accurately determined, while the state constraint multiplier is positive only along the boundary arc.}
    \label{fig42}
\end{figure}

%% file: 040-HJ_approach.tex
\section{Hamilton--Jacobi--Bellman approach} \label{s5}

The systematic study of optimal control problems dates back to the late 1950s, and one major tool is Dynamical Programming and Hamilton--Jacobi--Bellman (HJB) approach. This approach describes the optimal control problem via the so-called {\it value function} $\VV: [0,T]\times\R^d\longrightarrow \R$, which associates to any initial condition the optimal value of the control problem, and is defined accordingly as 
\begin{eqnarray}
  \VV (t,x) 
  & := & \inf\bigg\{\varphi(\bx(T))  + \int_t^T \ell(s,\bx(s),\bu(s))ds\ \bigg| \nonumber \\
  & & (\bx,\bu)\in \XX_{[t,T]}(x), \ 
  g(\bx(s))\leq 0 \quad \mbox{ for all } \tau\in [s,T], 
  \ \ \mbox{and } g_f(\bx(T))\leq 0\bigg\}.
 \label{eq.vartheta}
\end{eqnarray}
It is known that  the value function $\VV $
 can be characterized as the unique solution, in a suitably weak sense, of a Hamilton--Jacobi type equation \cite{BarCap97}. Starting from the knowledge of the value function, which is typically obtained via numerical approximation, it is possible to reconstruct the optimal solution in feedback form, i.e., with an optimal control given as a function of the current state.

%%%%%=============================
\subsection{Unconstrained Bolza control problems}
%%%%%===========================

In this section we present first some classical results of HJB approach when the optimal control problems is free of  state constraints (i.e.,  $g=g_f\equiv 0$). The control problem is described by the unconstrained value function
 \begin{equation}\label{vdef}
   \VV (t,x) 
 = \inf\bigg\{\varphi(\bx(T))  + \int_t^T \ell(s,\bx(s),\bu(s))ds\ \bigg| 
  (\bx,\bu)\in \XX_{[t,T]}(x)\bigg\}.
 \end{equation}
 %By remark \ref{sec5:compact},  
When \Hyp{4} is satisfied, the control problem admits a solution.
The Gronwall estimate on the trajectories and the Lipschitz regularity of the cost functions ensure  that the value function, although in general non-differentiable, enjoys itself a  Lipschitz continuity property.
\begin{prop}
Under \Hyp{0}-\Hyp{2}, the value function $\VV$ is locally Lipschitz continuous.
\end{prop}

%\medskip

To deal with the lack of smoothness, two important tools have been developed: the theory of viscosity solutions and the non-smooth analysis. The theory of viscosity solutions for nonlinear Hamilton--Jacobi equations, introduced in the early 1980s by Crandall--Lions \cite{CL81,CL83} and Crandall--Evans--Lions \cite{CEL84}, allows to define generalized solutions to broad classes of nonlinear partial differential equations, including the HJB equations of optimal control problems. We refer also to the books \cite{Bar94,BarCap97} for a more complete introduction to this theory. Another important tool is the non-smooth analysis, which addresses to differential analysis for non-smooth functions. We refer the reader to \cite{CELAUB,Clarke2013book,CLSW98,Vinter2000} for an introduction of the theory and its applications.

%%%%%%%%%%%%%%%%%%%%%%%%%%%%%%%%%%%%%%%%%%%%%%%%%%%%%%%%%%%%%%%%%%%%%%%%%%%%%%%%%%%%%%%%%%

%%%%%%%%%%%%%%%%%%%%%%%%%%%%%%%%%%%%%%%%%%%%%%%%%%%%%%%%%%%%%%%%%%%%%%%%%%%%%%%%%%%%%%%%%%%%%
\subsubsection{Dynamic programming and Hamilton--Jacobi--Bellman equation}

The fundamental idea of Dynamic Programming is that the value function $\cV$ satisfies a functional equation, often called the \textit{Dynamic Programming Principle} (DPP). 

\begin{prop}[Dynamic Programming Principle]\label{dpp}
Assume \Hyp{0}--\Hyp{2}, and denote by $\bx^\bu(s)$ the solution of \eqref{eq:etat} for a given control $\bu$, and such that $\bx^\bu(t)=x$.
Then, for all $x\in\R^d$, $t\in[0,T]$ and $h\in[0,T-t]$, the value function $\cV$ satisfies the equality
\begin{equation}\label{eq:dpp}
\cV(t,x)=\inf_{\bu\in\cU}\left\{\int_t^{t+h}\ell(s,\bx^\bu(s),\bu(s))ds + \cV(t+h,\bx^\bu(t+h))\right\}.
\end{equation}
\end{prop}

This principle provides two properties, called sub- and super-optimality, which are defined as follows. 
For any function $V:[0,T]\times\R^d\to\R$,
\begin{enumerate}%[(i)]
\item
we say that $V$ satisfies the {\em super-optimality principle} if for any $t\in[0,T]$, $x\in\R^d$, there exists $(\bx^\bu,\bu)\in \XX_{[t,T]}(x)$ such that
\[
V(t,x)\,\geq \,V(t+h,\bx^\bu(t+h))+\int_t^{t+h}\ell(s,\bx^\bu(s),\bu(s))\,ds,\ \forall\,h\in[0,T-t];
\]
\item
we say that $V$ satisfies the {\em sub-optimality principle} if for any $t\in[0,T]$, $x\in\R^d$, and $(\bx^\bu,\bu)\in \XX_{[t,T]}(x)$,
\[
V(t,x)\,\leq\, V(t+h,\bx^\bu(t+h))+\int_t^{t+h}\ell(s,\bx^\bu(s),\bu(s))\,ds,\ \forall\,h\in[0,T-t].
\]
\end{enumerate}

\noindent In principle, once chosen a ``small'' time increment $h$, the DPP allows to compute the value function at the point $(t,x)$ by splitting the trajectories at time $t+h$ and starting with the position of the trajectory $y_{t,x}$ at time $t+h$. As it will be seen later on, it is possible to construct numerical schemes based on this idea to compute an approximation of the value function.

Under the assumption of differentiabillity for the function $\cV$, we can derive from the DPP its infinitesimal version, the \textit{Hamilton--Jacobi--Bellman} equation
\begin{subequations}\label{sec5:HJB}
\begin{eqnarray}
& & -\cV_t(t,x)+H(t,x,D\cV(t,x))=0 \quad  \text{for } (t,x)\in(0,T)\times\R^d,\label{sec5:HJBa}\\
& & \cV(T,x)=\vp(x) \qquad  \qquad \qquad 
\qquad \qquad \text{for } x\in\R^d,\label{sec5:HJBb}
\end{eqnarray}
\end{subequations}
where the Hamiltonian is given by
\begin{equation}\label{eq:H}
H(t,x,q)=\sup_{u \in U}\big\{-f(t,x,u)\cdot q-\ell(t,x,u)\big\}.
\end{equation}
In general, as mentioned before, neither $\VV$ is differentiable, nor the nonlinear equation \eqref{sec5:HJB} is expected to admit a classical solution. These problems are circumvented by the theory of viscosity solutions and the non-smooth analysis, see \cite{CLSW98,BarCap97,Vinter2000}. 

%%===========
\subsubsection{Theory of viscosity solutions for Hamilton--Jacobi--Bellman equations}
%%==============
First, we recall the definition of viscosity solution for HJB equations (see \cite{CL81,CL83,Bar94,BarCap97}).
\begin{definition}[Viscosity solution]\label{sec5:DEFvs}
Let $V:[0,T]\times\R^d\longrightarrow\R$.
\begin{itemize}
\item [(i)]
We say that $V$ is a viscosity supersolution
if $V$ is lower semicontinuous (lsc) and for any $\phi\in C^1((0,T)\times\R^d)$ and
$(t_0,x_0)\in (0,T)\times \R^d$ local minimum point of $V-\phi$,
we have
\begin{equation*}
-\phi_t(t_0,x_0)+H(t_0,x_0,D\phi(t_0,x_0))\geq 0.
\end{equation*}
\item [(ii)]
We say that $V$ is a viscosity subsolution
if $V$ is upper semicontinuous (usc) and for any $\phi\in C^1((0,T)\times\R^d)$ and
$(t_0,x_0)\in (0,T)\times \R^d$ local maximum point of $V-\phi$,
we have
\begin{equation*}
-\phi_t(t_0,x_0)+H(t_0,x_0,D\phi(t_0,x_0))\leq 0.
\end{equation*}
\item [(iii)]
We say that $V$ is a viscosity solution
if it is both a viscosity supersolution and a viscosity subsolution  and the final condition is satisfied:
%%%%\todo{Is it important here to use usc and lsc functions instead of a buc %%%%function in both cases?}
\begin{equation*}
V(T,x)=\vp(x) \text{ in } \R^d.
\end{equation*}
\end{itemize}
\end{definition}
\begin{rem}
There are also some equivalent definitions which are more local using the super- and sub-differentials, this meaning that the differentials of the test functions can be replaced by some weak differentials of the viscosity solution. See \cite{Bar94,BarCap97} for the definition using the Dini-differentials and \cite{CLSW98} for the definition using the proximal differentials.
\end{rem}

\begin{thm}\label{sec5:existence}
Suppose that \Hyp{0}-\Hyp{2} hold. Then the value function $\VV$ is the unique viscosity solution of \eqref{sec5:HJB} in the sense of Definition \ref{sec5:DEFvs}.
\end{thm}

The first statement of this theorem is that $\VV$ is a viscosity  solution of
\eqref{sec5:HJB}. The proof of this claim relies on the regularity of the value function (continuity) and on the dynamic programming principle \cite{BarCap97}.
The theorem also claims that the value function is the unique solution of \eqref{sec5:HJB}. This is the consequence of the following  equivalences that can be established by non-smooth analysis (see \cite{Clarke2013book,Vinter2000} for the proof): 
\begin{eqnarray*}
\left\{\begin{array}{l}
V:[0,T]\times\R^d  \mbox{ is usc}, \\
V \mbox{ satisfies the sub-optimality principle} \end{array} \right.
&\Longleftrightarrow& V \mbox{ is a viscosity sub-solution of \eqref{sec5:HJBa}}, \\
\left\{\begin{array}{l}
V:[0,T]\times\R^d  \mbox{ is lsc}, \\
V \mbox{ satisfies the super-optimality principle} \end{array} \right.
&\Longleftrightarrow& V \mbox{ is a viscosity super-solution of \eqref{sec5:HJBa}}.
\end{eqnarray*}
From the point of view of PDEs  and viscosity theory, one can also obtain uniqueness by using a general theorem on the comparison principle  which can be stated as follows.
\begin{thm}
Let $V_1$ be a subsolution of \eqref{sec5:HJB} and $V_2$ be a supersolution of \eqref{sec5:HJB} with $V_1(T,x)\leq \vp(x)\leq V_2(T,x)$ for $x\in\R^d$. Then for any $t\in[0,T]$, $x\in\R^d$,
\[
V_1(t,x)\leq V_2(t,x).
\]
\end{thm}

The classical proof of the comparison principle is based on the variable doubling technique. We  refer to \cite{CL81,CL83,Bar94,BarCap97} for more details.     

\if{\color{red}{It consists of considering 
\[
\sup_{t,s,x,y}\left\{v_1(t,x)-v_2(s,y)-\frac{|t-s|^2+|x-y|^2}{\eps^2}\right\},
\]
where $\eps>0$. It is a regularization technique, called sup/inf-convolution for sub/super-solutions. Then the regularization of $v_1$ and $u_2$ can be considered as the viscosity test functions for $v_2$ and $v_1$ respectively, and the comparison result is deduced by the information obtained through the viscosity tests.}}\fi
%%%%%%%%%%%%%%%%%%%%%%%%%%%%%%%%%%%%%%%%%%%%%%%%%

%%%%%%%%%%%%%%%%%%%%%%%%%%%%
\subsection{Other unconstrained control problems and their HJB formulation} \label{Other_OCP}
%%%%%%%%%%%%%%%%%%%%%%%%%%%

In addition to the Bolza problem, which has been taken here as a model, various other formulations have been considered for optimal control problems, in particular without a final time, or in which the final time is itself a parameter to be optimally chosen. We briefly review some of these formulations, while a more extensive discussion can be found, for example, in \cite{BarCap97}.

\paragraph{Minmax control problems.}
In this class of control problems, the cost is not defined in integral form. More precisely, the control problem
 reads as 
 $$\cV^\#(t,x):=\min\Big\{ \varphi(\bx(T)) \bigvee \max_{s\in[t,T]}\Psi(s,\bx(s))\Big| (\bx,\bu)\in \XX_{[t,T]}(x)\Big\}.$$
 Here, the cost function is the maximum between the final cost and a maximum running cost along the trajectory.  Relation between minmax control problems and state constrained control problems have been noticed and exploited to derive necessary conditions of optimal trajectories, see \cite[Chapter 9]{Vinter2000}.  In Hamilton--Jacobi approach, minmax control problems have been also analyzed in \cite{Barron}. It has been shown that the value function 
 is Lipschitz continuous and can be characterized by the HJB variational inequality
 \begin{subequations}\label{eq.HJB.minmax}
 \begin{eqnarray}
& & \min\big(\partial_t\cV^\#(t,x) + H^\#(t,x,D\cV^\#(t,x)),\cV^\#(t,x)-\Psi(t,x)\big) = 0\quad \mbox{on } 
		[0,T)\times \R^d, \\
& & \cV^\#(T,x)=\varphi(x)\vee \Psi(T,x) \quad \mbox{for } x\in \R^d.
\end{eqnarray}
\end{subequations}
In this inequality, the running cost function plays the role of an ``obstacle''. The value function 
satisfies $\cV^\#(t,x)\geq \Psi(t,x)$. It is also a super-solution of the equation:
\begin{equation}\label{eq.soussol}
\partial_tv(t,x) + H^\#(t,x,D_xv(t,x)) =  0 \quad \mbox{on } [0,T)\times \R^d.
\end{equation}
However, $\cV^\#$ is a sub-solution of \eqref{eq.soussol} only in open sets where 
$\cV^\# < \Psi$.
 
\paragraph{Infinite horizon.}

{\em Infinite horizon problems} are intended to model optimal control strategies in the long-time behaviour.
Assume that both the dynamics $f$ and the running cost $\ell$ do not depend on time, so that
\begin{equation}\label{dyn_not}
\dot\bx(s) = f(\bx(s),\bu(s)),
\end{equation}
with the initial condition $\bx(0)=x$.
The {\em discounted infinite horizon} cost functional is defined as
\[
J(x,\bu) = \int_0^\infty\ell(\bx(s),\bu(s))e^{-\lambda s}ds,
\]
where, in addition to the basic assumptions, we require that $\lambda>0$, and, for simplicity, that $\ell$ is uniformly bounded. According to the definition of the cost functional, it is possible to define a value function, which will depend in this case on $x$ alone:
\begin{equation*}
  \cV^\infty(x) :=  \inf\bigg\{J(x,\bu)\ \bigg| (\bx,\bu)\in \XX_{[0,\infty)}(x)\bigg\}.
\end{equation*}
The value function can still be characterized as the viscosity solution of a (stationary) Hamilton--Jacobi--Bellman equation, which takes the form
\begin{equation}\label{ih}
\lambda\cV^\infty(x) + H(x,D\cV^\infty(x)) = 0,
\end{equation}
in which $x\in\R^d$, and the Hamiltonian function $H$ is defined by \eqref{eq:H}.
As for the regularity of the value function, the basic assumptions imply uniform continuity of the value function. H\"older regularity of $\cV^\infty$ holds under the assumptions of boundedness for $\ell$, and global Lipschitz continuity for both $\ell$ and $f$; in addition, $\cV^\infty$ is itself globally Lipschitz continuous if $\lambda$ is larger than the Lipschitz constant of $f$ (see \cite{BarCap97}).

\paragraph{Free final time.}
In {\em free final time problems}, also termed as {\em optimal stopping time problems}, the endtime $\theta$ of the control interval is itself a free parameter to be chosen in an optimal way. In the simplest case, the dynamics is set in the form \eqref{dyn_not} and the cost functional in the form
\[
J(x,(\bu,\theta)) = \int_0^\theta\ell(\bx(s),\bu(s))e^{-\lambda s}ds + e^{-\lambda \theta}\varphi(\bx(\theta)).
\]
Accordingly, the value function is defined as
\begin{equation*}
  \cV^f(x) := \inf\bigg\{J(x,(\bu,\theta))\ \bigg| (\bx,\bu)\in \XX_{[0,\infty)}(x), \theta\geq 0\bigg\}.
\end{equation*}
In this case, the HJB equation is stationary, but comes in the form of an obstacle problem:
\[
\max\big(\lambda\cV^f(x) + H(x,D\cV^f(x)),\cV^f(x)-\varphi(x)\big) = 0.
\]
The state space is then split in two (possibly overlapping) sets: in the first one the first argument of the max vanishes, so that
\[
\lambda\cV^f(x) + H(x,D\cV^f(x)) = 0,
\]
and the optimal control requires to keep the system evolving; in the second set the second argument of the max vanishes, and therefore
\[
\cV^f(x)=\varphi(x).
\]
As soon as the state of the system enters this set, the optimal strategy requires to stop the system, paying the stopping cost $\varphi(\bx(\theta))$. Under the basic assumptions, the value function is uniformly continuous (see \cite{BarCap97}).

\paragraph{Minimum time.}
For simplicity, we assume again in this paragraph that the dynamics does not depend explicitly on time, that is $f(s,x,u)=f(x,u)$. For minimum time function with time-dependent dynamics, we refer to \cite{BBZ10} and the references therein.

In {\em minimum time problems}, the goal is to drive the state of the system, in the shortest time, to the final  closed set  $\mathcal C:=\{x\in \R^d, g_f(x)\leq 0\}$ (called the target).  
% with nonempty interior and a sufficiently smooth boundary. 
In general, the possibility of driving the state to the target in finite time may not be ensured for each initial state; this leads to define the so-called {\em backward reachable set} $\mathcal R$ as the set of initial states $x$ which can be driven to the target in finite time. 
The minimum time control  problem is formulated as
\begin{equation}\label{pb.tmin} \cT(x):=\inf \Big\{t \ \Big| g_f(\bx(t)) \ \ (\bx,\bu)\in\XX_{[0,\infty)}(x)\Big\},
\end{equation}
with the convention that $\cT(x)=+\infty$ when there is no trajectory that starts from $x$ and reaches the target in finite time.  
It is not difficult to show that the value function satisfies a dynamic programming principle (DPP) in the form
$$\cT(x)=\inf\Big\{\cT(\bx(h))+h, \ \ (\bx,\bu)\in\XX_{[0,\infty)}(x)\Big\}.$$
When the backward reachable set is open and the minimum time function is continuous, then the DPP leads to a characterization of $\cT$ by the  HJB equation
\[
 \sup_{u \in U}\big\{-f(x,u)\cdot D\cT(x)\big\}=1,
\]
complemented with the boundary condition
\[
%\begin{cases}
\cT(x)=0 \ \ \mbox{for } x\in \mathcal C \qquad \mbox{and} \quad
\displaystyle\lim_{x\to z}\cT(x)=+\infty  \mbox{ for }  z\in\partial\mathcal R.
%\end{cases}
\]
The drawback of this formulation is that the reachable set  should be known in advance. An alternative formulation makes use of the so-called {\em Kru\v{z}kov transformation}
\[
\cV(x) =
\begin{cases}
1 & \text{ if } \cT(x)=+\infty \\
1-e^{-\cT(x)} & \text{ else, }
\end{cases}
\]
which solves (yet when the time function $\cT$ is continuous) an auxiliary infinite horizon problem of the form \eqref{ih}, with $\lambda=1$, $\ell\equiv 1$, and the boundary condition
\[
\cV(x)=0 \quad \mbox{for } x\in\mathcal \cC.
\]
In this case, the reachable set is obtained as a byproduct of the computation of the value function $\cV$, as
\[
{\mathcal R} = \big\{x\in\R^d\ |\ \cV(x)<1\big\}.
\]
Continuity of the function $\cT$ is closely related to controllability properties satisfied by the system  in a neighbourhood of $\mathcal C$ and more precisely, to the so-called 
{\it Small-Time Local Controllability}, see \cite[Chapter IV]{BarCap97}.  When the target is smooth enough (for instance, assume here that $g_f:\R^d\to \R$ is $C^2$) with a compact boundary,  a necessary condition for the continuity of the minimum time function is given by the condition
\begin{equation}\label{eq.Petrovrelax}
\min_{u\in U}f(x,u)\cdot \nabla g_f(x) \leq 0.
\end{equation}
 This condition is restrictive  and excludes a large class of systems with drift.
An equivalent condition to the Lipschitz continuity of the minimal time function is the
 {\em  Petrov condition} 
\begin{equation}\label{eq.Petrov}
\min_{u\in U}f(x,u)\cdot \nabla g_f(x) <0,
\end{equation}
which  is even more restrictive than \eqref{eq.Petrovrelax}.
 From a geometrical point of view,  Petrov condition states that at every point of a neighborhood of the target there exists an
admissible control such that the corresponding trajectory points towards the target. 

In general, when $g_f$ satisfies assumption \Hyp{3}, the set $\cC$ is closed and the controllability conditions \eqref{eq.Petrovrelax}-\eqref{eq.Petrov} might  not be  satisfied. A practical approach to compute the minimum time function and the corresponding optimal trajectories is based on the {\em level set method},  introduced by Osher and Sethian \cite{osh-set-88}.  Consider the final cost function as 
$$ \Phi(x):= \max\Big(g_{1,f}(x), \cdots, g_{m_f,f}(x)\Big)\quad \mbox{for } x\in \R^d,$$
where $(g_{1,f}, \cdots, g_{m_f,f})$ are the components of the function $g_f$. 
The level set approach consists of considering
the value function $V$  associated to the {\em Mayer problem}
with final cost $\Phi$, defined by
\be
\label{pb:tmin_LS} 
V(t,x):= \min\{\Phi(\bx(t)),\ (\bx,\bu)\in \XX_{[0,t]}(x)\}. 
\end{eqnarray}
The value function $V$ can be characterized by an HJB equation as in Theorem~\ref{sec5:existence}.  This function is Lipschitz continuous,  while the minimum time function may be discontinuous.
Besides, one can notice immediately that if $V(t,x)\leq 0$, then 
there exists an admissible pair $(\bx,\bu)\in \XX_{[0,t]}(x)$ such that 
$\Phi(\bx^\bu_x(t))$, which means that $(\bx,\bu)$ satisfies the final constraint.  More precisely, the set of positions from where it is possible to reach the target at time $t$ is given by
$$\cR(t):=\{x\in \R^d\ \mid\ \exists (\bx,\bu)\in \XX_{[0,t]}(x), g_f(\bx(t))\leq 0\} =
\{x\in \R^d\mid V(t,x)\leq 0\}.$$
Therefore, the value function $V$ gives valuable information on the minimum time function and the corresponding trajectories, without any controllability assumption.
\begin{thm}
\label{theo:LS}
Suppose that \Hyp{0}-\Hyp{3} hold.
 Let $V$ be the value function associated to problem \eqref{pb:tmin_LS}. 
 Then, for every $x\in \R^d$, 
 \begin{eqnarray}
 & & \cT(x):=\min \{t\mid V(t,x)\leq 0\}.\label{Tmin_LS}
 \end{eqnarray}
Furthermore,  for $x\in \R^d$,  any optimal trajectory for the minimal time problem \eqref{pb.tmin} is also an optimal trajectory of the 	control problem \eqref{pb:tmin_LS} 
where $t$ is fixed to the minimum time given by \eqref{Tmin_LS}.
\end{thm} 
The level set approach provides an  effective way to compute the minimum time to reach a target without assuming any specific regularity.  Moreover, it has been shown that the level set method can be generalized to minimum time problems with state constraints (for instance, the case when the trajectory should avoid some obstacles, see Example~1). In this case, the minimum time function is defined as 
$$ \cT(x):=\inf \Big\{t\ \Big| \ \ \exists (\bx,\bu)\in\XX_{[0,\infty)}(x)\ \ \mbox{with } g_f(\bx(t))\leq 0,  \ \ 
\mbox{and } g(\bx(s))\leq 0 \ \mbox{on } [0,t]\Big\}.$$
To use the level set approach in this context, the definition of the value function
should be adapted and defined as
$$
V(t,x):= \min\left\{\Phi(\bx(t))\bigvee \max_{s\in[0,t]}\Psi(\bx(s))\ \mid\ (\bx,\bu)\in \XX_{[0,t]}(x)\right\},
$$
with $\Psi(x):=\max(g_1(x),\ldots,g_{m_g}(x))$
and $\Phi(x):=\max(g_{f,1}(x),\ldots,g_{f,m_f}(x))$ for every $x\in \R^d$.
Here, the value function $V$ is again Lipschitz continuous,  while the minimum time function may be discontinuous. Besides, if $V(t,x)\leq 0$, then 
there exists an admissible pair $(\bx,\bu)\in \XX_{[0,t]}(x)$ such that 
$\Phi(\bx(t))\leq 0$ and
$\max_{s\in[0,t]}\Psi(\bx(s))\leq 0$, which means that $(\bx,\bu)$ satisfies the final and pointwise state constraints.  
With this new definition of the value function $V$,  the statement of Theorem~\ref{theo:LS} remains valid
in the case with obstacles. In particular,  the minimum time value and the corresponding optimal trajectories can be obtained  form the value function $V$ without assuming any  controllability hypothesis.   Finally, notice that the value function $V$ corresponds to a minmax problem, and its characterization is given by the HJB inequality \eqref{eq.HJB.minmax}.

%%%%%%%%%%%%%%%%%%%%%%%%%%%%%%%%%%%%%%
\subsection{Constrained Bolza problems}
%%%%%%%%%%%%%%%%%%%%%%%%%%%%%%%%%%%%%
In this section, we consider a control problem with state constraints.  We  denote by $\cK$ the set of constraints $\cK:=\{x\in \R^d, g(x)\leq 0\}$, and define the set of
admissible trajectories by
\begin{eqnarray*}
\XX^{g}_{[t,T]}(x):=\big\{(\bx,\bu)\in \XX_{[t,T]}(x)\ \mid\ 
g(\bx(s))\leq 0 \text{ for }s\in[t,T]\big\}.
\end{eqnarray*}
We adopt the convention $\cV(t,x)=+\infty$,
when the set of admissible trajectories is empty, i.e., 
$\XX^{g}_{[t,T]}(x)=\emptyset$. 
Similarly to the unconstrained case, the value function $\cV$ satisfies a  dynamic programming principle that can be stated as follows.
\begin{itemize}
  \item [i)] For all $x\in\cK$,
  \begin{equation*}
  \cV(T,x)=\vp(x).
  \end{equation*}
  \item [ii)] Dynamic programming principle: for all $x\in\cK$,
  $\tau\in[0,T]$ and $h\in[0,T-\tau]$, we have:
  \begin{equation}
  \cV(t,x)=\inf_{(\bx,\bu)\in \XX^{g}_{[t,T]}(x)}\cV(t+h,\bx(t+h)) +\int_t^{t+h}\ell(s,\bx(s),\bu(s))\,ds.
  \end{equation}
\end{itemize}

\subsubsection{Inward pointing condition}
To analyze the properties of the value function $\VV$, it is important first to understand the structure of the set the admissible trajectories. This structure depends on an interplay between the dynamics of the 
state equation and the set of constraints $\cK$. 
Assume  in this section that $g:\R^d\to \R$ is $C^{1,1}$ function, and its zero-level set is suitably smooth. Consider  the following controllability assumption:

\textbf{(HK1)} Inward pointing qualification (IPQ) condition: For every $R>0$, 
there exists  $\beta>0$ and $\rho>0$ such that for every $t\in [0,T]$,
\begin{equation}
\min_{u\in U} f(t,y,u)\cdot\nabla g(y)<-\beta, \qquad \forall y\in \partial \cK\cap \mathbb{B}(0,R).
% \ \mbox{ with } |g(y)|\leq \rho.
\end{equation}

The IPQ condition states that the set of constraints $\cK$ has a smooth structure and that on each point of the boundary $\cK$ it is possible to find an admissible control that allows the trajectory to stay in the set $\cK$.  So, the IPQ condition  implies that the set $\cK$ is weakly invariant. Moreover, the IPQ condition  guarantees even a nicer property, called {\bf (NFT) Neighbouring feasible trajectories } principle, whose proof can be found in \cite[Theorem 2.1]{BetBreVin2010}. 
\begin{lem}\label{remin}
Assume {\bf (HK1)}.  Let $(t_0,x_0)\in [0,T]\times \cK$ and let $(\by,\bv)\in \XX_{[t_0,T}(x_0)$. There exists a constant $C$ and a feasible pair $(\bx,\bu)\in \XX_{[t_0,T]}(x_0)$ such that
\[
\bx(t)\in\cK,\ \forall\,t\in[t_0,T],\ \|\by-\bx\|_{W^{1,1}([t_0,T];\R^d)}\leq Cg^+( \by(\cdot)),
\]
where
$\displaystyle 
g^+( \by(\cdot))=\max_{t\in[t_0,T]}\{\max(g( \by(t)),0)\}.
$
\end{lem}
The NFT property states the existence of an admissible control-trajectory pair satisfying the state constraints, close to an admissible pair that violates the state constraints. This property is the key point to  ensure continuity of the value function and to provide a characterization of the
value function in terms of viscosity solutions of the relevant  HJB equation on $\cK$ (see \cite{Soner-1986}).
\begin{thm}\label{continuity}
Assume  \textbf{(HK1)}. Then,
the value function $\cV(\cdot,\cdot)$ is uniformly continuous and bounded on
$[0,T]\times\cK$. Moreover, it is the unique constrained viscosity solution of the 
following HJB equation:
\begin{subequations} 
\label{eq.constrainedviscosity}
\begin{eqnarray}
& & -\partial_t\VV(t,x)+H(t,x,D_x\VV(t,x))\geq 0\quad \mbox{on } [0,T[\times \cK,\\
&& -\partial_t\VV(t,x)+H(t,x,D_x\VV(t,x))\leq 0 \quad \mbox{on }
[0,T[\times \mathop{\cK}\limits^{\circ},
\\
& & \VV(T,x)=\varphi(x),
\end{eqnarray}
\end{subequations}
with $\cV(t,x)=+\infty$ for every $x\in \R^d\setminus\cK$.
\end{thm}

It should be noticed that the HJB equation in the above theorem provides only partial information on the boundary of $\cK$. Moreover, the function $\VV$ takes infinite values outside $\cK$ (i.e.,  $\VV(t,x)=+\infty$ for every $x\not\in\cK$). These two facts make the approximation of $\VV$ on $\cK$ very challenging and require some penalization techniques.

Finally,  let us mention that the IPQ condition { \bf $(HK1)$} can  be weakened a bit  and  generalized to the case with several constraints as in  \cite{Fran-Mazz2013}.  In \cite{CRK}, a ``higher order'' inward pointing condition involving Lie brackets of the dynamics' vector fields is also analyzed.  
\if{Nevertheless, the  condition  remains restrictive for many control problems. 
Besides, the stability results guaranteed by the IPQ condition  ensure the qualification of any optimal trajectory.  Therefore,  when this assumption is satisfied,  any optimal trajectory corresponds to a normal extremal (see previous section).}\fi

\subsubsection{Case of state constraints without controllability assumptions} \label{Aux_CO}

As we mentioned in the previous section, the Lipschitz regularity of the value function requires  an interplay between the dynamics $f$ and the set of constraints $\cK$. When the controllabity condition
{\bf (HK1)} is not satisfied, the value function may be discontinuous and its characterization by a HJB equation becomes very delicate. In this section, we introduce an alternative formulation of state-constrained control problems, in case the controllability assumption is not satisfied.

%\subsubsection{Auxiliary Control problem}
%\label{Aux_CO}
We set $G(y):=\max(g_1(y),\cdots,g_{m_g}(y))$ and $G_f(y):=\max(g_{f,1}(y),\cdots,g_{f,m_f}(y))$ for every $y\in \R^d$.
We introduce an auxiliary control problem and its associated value function $\WW$ defined by
\begin{eqnarray}\label{eq:wg}
  & &\WW(t,x,z):=
   \inf_{\by=(\bx,\zeta)\in {\cal S}_{[t,T]}(x,z)} \quad \bigg( \varphi(\bx(T)) - \zeta(T) \bigg)\ 
    \bigvee \ \max_{\mt\in(t,T)} G(\bx(\mt)) \bigvee \ G_f(\bx(T))  
\ee
for $x\in\R^d$, $z\in\R$, $t\in[0,T]$,  $a\vee b:=\max(a,b)$,  and where the set of trajectories ${\cal S}_{[t,T]}(x,z)$
is defined in Remark~\ref{rem.existence}.
In this auxiliary control problem, the term $\max_{\mt\in[t,T]}  g\big(\bx(\mt)\big)$ is an exact
penalization of the state constraints. Here, we shall use the problem \eqref{eq:wg} to characterize
the epigraph of the value function $\var$ without requiring any additional controllability assumption.
\begin{thm}\label{th:1}
Assume that \Hyp{0}-\Hyp{4} are satisfied.
Then, for any $t\in[0,T]$ and $(x,z)\in \R^d \times \R$,\\
$(i)$ 
$$
  \VV(t,x)-z\leq 0
  \quad \Equivalent \quad  \WW(t,x,z)\leq 0.
$$
$(ii)$ 
In addition, the function $\VV$ is characterized by $\WW$ through the relation
\be\label{eq:vsharp-wg}
  \VV(t,x) = \min\bigg\{z \in \R,\ \WW(t,x,z)\leq 0\bigg\}.
\ee
\end{thm}
{\begin{proof}
$(i)$ Let us assume that $\VV(t,x)\leq z$. So there exists a sequence $(\bx_n,\bu_n)_{n\in\mathbb{N}}$ of admissible pairs in $\XX_{[t,T]}(x)$,
such that 
$$ \lim_{n\rightarrow+\infty} 
  \int_t^T \ell(s,\bx_n(s),\bu_n(s))\,ds + \varphi(\bx_n(T)) - z = \VV(t,x)-z\leq 0.$$
By admissibility, we have for each $n\geq 0$,  $\max_{\mt\in[t,T]} G(\bx_n(\mt)) \leq 0$ and 
$G_f(\bx_n(T))\leq 0$. 
Hence,
\beno
   \WW(t,x,z) \!\!\! & \!\!\! \leq \!\!\!& \!\! \liminf_{n\to \infty}\!\! \Bigg[
  \bigg( \int_t^T \ell(s,\bx_n(s),\bu_n(s))\,ds + \varphi(y_n(T)) - z\bigg) \bigvee \max_{\mt\in[t,T]} G(\bx_n(\mt)) \bigvee G_f(\bx_n(T)) \Bigg]\\
     \!\!
     & \!\! \leq \!\!& \!\!  0.
\eeno
Conversely, let us assume that $\WW(t,x,z) \leq 0$.
We know that ${\cal S}_{[t,T]}(\xi)$ is a compact set in $C^0([t,T])$,
therefore the infimum in $\WW(t,x,z)$ is achieved by some trajectory $\bx\in {\cal S}{[t,T]}((x,z))$
(with an associated control $\bu\in \cU$).
Moreover,
\beno
   0\geq \WW(t,x,z) & = &
  \bigg( \int_t^T \ell(s,\bx(s),\bu(s))\,ds + \varphi(\bx(T)) - z\bigg) \bigvee \max_{\mt\in[t,T]} G(y(\mt))  \bigvee G_f(\bx(T)).
\eeno
On the one hand, $\max_{\mt\in[t,T]} G(\bx(\mt)) \bigvee G_f(\bx(T)) \leq 0$ and $\bx$ satisfies the state constraints,
and on the other hand,
$$  \VV(t,x)-z \ \leq  \ \int_t^T \ell(s,\bx(s),\bu(s))\,ds + \varphi(\bx(T)) - z \ \leq \  0
$$
which is the desired result. 
Finally, statement $(ii)$ is an immediate consequence of $(i)$.
\end{proof}

\begin{rem}
Should the convexity assumption \Hyp{4} not be satisfied, the statements of the above theorem may not hold. Indeed, in general we have
$$ \inf\{z\in \R, \WW(t,x,z) \leq0\} 
\leq \VV(t,x) \leq \inf\{z\in \R, \WW(t,x,z) <0\}.$$
% Consider the  problem where   $\ell(t,x,u)=0$, $f(t,x,u):=(1,u)$ with  $u\in U:=\{\pm 1\}$, 
%  $\cK:=\{(x_1,x_2)\in\R^2,\ |x_2|\leq |x_1-\frac{1}{2}|^2\}$, and $T=1$.
% One can check that for $\bar x=(0,0)^{\sf T}$ there exist no admissible trajectories starting at 
% $\bar x$ and staying in $\cK$ on $[0,T]$.  
% Hence $\VV(0,(0,0))=+\infty$ for all $z\in\R$.
% In contrast, one can define a sequence of trajectories $(\bx_n)_{n\geq 1}$ for $\mt\in(0,T)$ by setting
% \begin{eqnarray}
%  & &  \bx_n(\mt) = \left\{
%  \begin{array}{ll}
%  (\frac{k}{n},0)^{\sf T} + (\mt - \frac{k}{n}) (1,1)^{\sf T} & \mbox{if}\ \mt\in [\frac{k}{n}, \frac {k+1/2}{n}[, \\
% %  (\frac{k+1/2}{n},\frac{1/2}{n})^{\sf T} + (\mt - \frac{k+1/2}{n}) (1,-1)^{\sf T} & \mbox{if}\ 
%     \mt\in [\frac{k+1/2}{n}, \frac {k+1}{n}[.
%  \end{array}  
%  \right.
% \end{eqnarray}
% The sequence $(\bx_n)$ converges uniformly on $[0,1]$ toward the limit $\bx(t)=(t,0)^{\sf T}$, 
% and $\WW(0,x,z)$ has a finite negative value whenever $\varphi((T,0)^{\sf T})<z$.
\end{rem}

The auxiliary control problem suggests  a reformulation of the state-constrained optimal control problem in an augmented state space.  In this new formulation, the constraints are integrated into the functional to be minimized.  The value function $\WW$ is Lipschitz continuous and it can be characterized by an HJB equation without any additional controllability assumption. 
\begin{thm}
Assume that \Hyp{0}-\Hyp{3} are satisfied. Then, the auxiliary value function is Lipschitz continuous and it is the unique viscosity solution of the HJB equation
\begin{eqnarray*}
& & \min\Big(-\partial_t\WW(t,x,z)+H(t,x,D_x\WW(t,x,z), D_z\WW(t,x,z)),  \WW(t,x,z)-G(x)\Big)=0,\\
& & \hspace*{10.7cm} \mbox{on } [0,T[\times\R^d\times \R,\\
& & \WW(T,x,z)=\bigg(\varphi(x)-z\bigg)\bigvee G(x) \bigvee G_f(x),\quad \hspace*{4.1cm}\mbox{on } \R^d\times\R,
\end{eqnarray*}
where the Hamiltonian $H$ is defined by 
$$H(t,x,p,q)= \max_{u\in U}\big(-f(t,x,u)\cdot p-\ell(t,x,u)q\big)$$
for every $(t,x,p,q)\in[0,T]\times\R^d\times\R^d\times\R$.
\end{thm}
We point out that any optimal trajectory for the original problem is also a solution of the auxiliary problem when $z=\bar z:=\cV(t,x)$.  Conversely,  any solution of the auxiliary problem with $z=\bar z$ is an optimal solution of the original state-constrained problem.  As a consequence, the  auxiliary problem provides  the value of the original state-constrained control problem and also allows to reconstruct the optimal trajectories (see \cite{Alt-Bok-Zid-2013,ref5}).
%%%%%%%%%%%%%%%%%%%%%%%%%%%%%%%
\subsection{Relationship between HJB and PMP}\label{sec-PMP-HJB}
%%%%%%%%%%==============================%%%%%%%%
To explain  the relationship between HJB and PMP, we first recall a classical result (see for example \cite{FleRis76}), valid under the very restrictive assumption that $\VV \in C^2$.

\begin{thm}
Consider the optimal control problem \eqref{eq.Pb_CO} with $g=0$ and $g_f=0$ (no state constraint), under the assumptions \Hyp{0}--\Hyp{1a}--\Hyp{1b}. Assume moreover that $f$ and $\ell$ are continuously differentiable with respect to the space variable, that $\cV\in C^2([0,T]\times \R^d)$, and that there exists an optimal pair $(\bx^*,\bu^*)$. Then, the vector $\bp(t)$ defined by
\begin{equation}\label{eq:adj_state_C2}
\bp(t) := -D\cV(t,\bx^*(t))
\end{equation}
satisfies the Pontryagin maximum principle.
\end{thm}
\if{\begin{proof}
In this proof, for shortness of notation, we will denote partial differentiations by subscript.
First, we can differentiate \eqref{eq:adj_state_C2} with respect to time obtaining
\begin{eqnarray}\label{eq:p_dot}
\dot p(t)  = &-\frac{d}{d t} (D\cV)(t,\bx^*(t)) %\nonumber \\
 =  -\cV_{xt}(t,\bx^*(t)) + \cV_{xx}(t,\bx^*(t)) f(t,\bx^*(t),\bu^*(t)).
\end{eqnarray}
From the Bellman equation, taking into account the existence of an optimal control $\bu^*$ we have, along the corresponding optimal trajectory $\bx^*$, 
\begin{equation}\label{eq:hjb_traj}
\cV_t(t,\bx^*(t)) + f(t,\bx^*(t),\bu^*(t)) \cdot D\cV(t,\bx^*(t)) + \ell(t,\bx^*(t),\bu^*(t)) = 0.
\end{equation}
The left-hand side of \eqref{eq:hjb_traj} is differentiable with respect to $x$, and attains its minimum at $\bx^*(t)$. It can therefore be differentiated giving
\begin{eqnarray*}
&& \cV_{tx}(t,\bx^*(t)) + \cV_{xx}(t,\bx^*(t)) f(t,\bx^*(t),\bu^*(t)) \\
&& \hspace{2cm} + f_x(t,\bx^*(t),\bu^*(t)) \cV_x(t,\bx^*(t)) + \ell_x(t,\bx^*(t),\bu^*(t)) = 0,
\end{eqnarray*}
where $f_x$ denotes the Jacobian matrix of $f$. Now, using \eqref{eq:p_dot} and changing the order of differentiation, we obtain
\begin{equation}
\dot p(t) + f_x(t,\bx^*(t),\bu^*(t)) \cV_x(t,\bx^*(t)) + \ell_x(t,\bx^*(t),\bu^*(t)) = 0,
\end{equation}
which coincides with the adjoint state equation. Finally, from the Bellman equation \eqref{sec5:HJB} one obtains \eqref{contraintePMP}.
\end{proof}}\fi 

While this relatively simple result works under unrealistically strong assumptions, more recent theory \cite{vin-88} justifies the PMP--HJB relationship (in a suitably weakened form) under general assumptions.
In the case of Mayer problems with a locally Lipschitz continuous cost, the sensitivity relations have also been studied  in \cite{cla-vin-87, vin-88}.  In these results, the value function is only required to be Lipschitz continuous in a neighborhood of the optimal trajectory.  The final cost is not differentiable and therefore, the costate function is not necessarily unique. In this context the  sensitivity relations assert that there exists 
 $\bp$ verifying the PMP, the terminal conditions
\begin{subequations}\label{PMP-HJB}
\begin{equation}\label{eq.p0pT}
-\bp(0)\in \partial_x\VV(0,\bx^*(0)), \qquad -\bp(T)\in \partial_x\VV(T,\bx^*(T)), 
\end{equation}
and both a \emph{partial sensitivity relation}
 \begin{equation}
 -\bp(t)\in \partial_x\VV(t,\bx^*(t)), \quad \mbox{ a.e. on $(0,T)$}, \label{eq:Link1-a}	
 \end{equation}
and a \emph{global sensitivity relation}
\begin{equation}
(\cH(t,\bx^*(t),\bp(t)),-\bp(t)) \in \partial \VV(t,\bx^*(t)), \quad \mbox{ for all  $t\in[0,T]$}.\label{eq:Link2}	
\end{equation}
\end{subequations}
These relations extend \eqref{eq:adj_state_C2} 
by using the  generalized gradient of the value function (which is well defined for locally Lipschitz continuous functions).  The set of
conditions \eqref{PMP-HJB} is in essence a strengthened necessary condition, asserting that it is possible
to choose a co-state trajectory to satisfy the sensitivity relations. 

The relation \eqref{eq.p0pT} can be simply derived by noting
that the optimal solution $(\bx^*,\bu^*)$ is also solution of the free initial state problem
$$
\min \Big\{\varphi(\bx(T))+\int_0^T\ell(t,\bx(t),\bu(t))\,dt -\cV(0,\bx(0)) \ \Big| \ \dot{\bx}(s)=f(s,\bx(s),\bu(s)) \mbox{ on } (0,T)
\Big\}.$$
Applying necessary optimality conditions to this problem yields a costate arc $\bp^0$. Relation \eqref{eq.p0pT}  turns out to be nothing else than the
transversality conditions at the endpoints.  
With the same reasoning, we can notice that $(\bx^*,\bu^*)$ is also solution of the free initial state problem on $[t,T]$ for every $0\leq t\leq T$
$$
\min \Big\{\varphi(\bx(T))+\int_0^T\ell(t,\bx(t),\bu(t))\,dt -\cV(t,\bx(t)) \ \Big| \  \dot{\bx}(s)=f(s,\bx(s),\bu(s)) \mbox{ on } (0,T)
\Big\}.$$
Here again, the optimality condition applied to the 
free intial point asserts the existence of an adjoint arc $\bp^t$ (which depends on the initial time $t$). The  left-endpoint transversality condition yields the relation 
$\bp^t (t)\in \partial\cV(t,\bx^*(t))$. When the final cost function is $C^1$-regular, the costate $\bp^0$
restricted to $[t, T ]$ is the unique solution to the costate equation on this interval satisfying
the right transversality condition. It follows that $\bp^0(t)$ coincides with $\bp^t (t)$; the
proof of \eqref{eq:Link1-a} is then completed. This analysis breaks down when  the final cost $\varphi$ is non-smooth. Indeed, in that case,
 co-state trajectory may not be unique.  
 An example is
 given in \cite[Section 4]{cla-vin-87} showing that, in some cases, there are a number
of possible choices of co-state trajectories  associated with the same optimal control problem, but not all of them satisfy the sensitivity relations.

The original proof of the sensitivity relations \eqref{PMP-HJB} is given in \cite{cla-vin-87,vin-88,Vinter2000}.
In the case when the control problem is in presence of state constraints, the sensitivity relations can be expressed in term of relations between the adjoint vector $(\bp,p^0)$  and the value function $\WW$ of the auxiliary control problem, defined in Section~\ref{Aux_CO}, see \cite{BokDesZid,HerZid}.  

%%%==========================
\subsection{Numerical methods for HJB}
%%%============================

In order to present the general theory for the approximation of viscosity solutions of HJB equations, we refer to an abstract forward problem
\begin{equation}\label{PG}
\begin{cases}
v_t+H(t,x,Dv)=0 & (t,x)\in (0,T] \times \R^d, \\
v(0,x) = v_0(x) & x\in\R^d,
\end{cases}
\end{equation}
(for some continuous Hamiltonian $H$) and set ourselves in the usual finite difference scheme framework. Time is discretized with a (fixed) time step $\Delta t$, so that $t_k=k\Delta t$; space is discretized with a fixed space step $\Delta x=(\Delta x_1,\cdots, \Delta x_d)$. 
A generic node will be denoted by $x_j=j\Delta x$, for $j\in\mathbb{Z}^d$. We also define  $\Delta=(\Delta x, \Delta t)$. 
More general options can be considered, in particular variable time steps and unstructured space grids, but we will restrict here to the basic ideas. 
In the following, we denote by $V^n_{i}$ the desired  approximation of  $v(t_n,x_i)$, and by  $V^n$ the set of nodal values for the numerical solution $v(t_n,\cdot)$ at time $t_n$. A scheme may be written in compact form as
\begin{equation}\label{def:S}
V^{n+1}=S(\Delta; V^n),
\end{equation}
where $S$ may be defined in terms of its components $S_j$, for $j\in\mathbb{Z}^d$.  
% We also denote by $W(t)$ and $\Phi(t)$ the sets of nodal values of generic functions $w(t,x)$ and $\phi(t,x)$, considered as possibly infinite vectors.

\subsubsection{Monotone schemes}

The first and basic convergence theory aimed at approximating HJ equations of the form \eqref{PG} uses the concept of monotone scheme. Among the various results, we quote here the Barles--Souganidis theory \cite{BarSou91}, which applies to the widest class of schemes and models, including the possibility of treating second-order, degenerate and singular equations. Roughly speaking, this theory states that any monotone, stable and consistent scheme converges to the exact viscosity solution, provided there exists a comparison principle for the limiting equation.
Consider a scheme in the general form \eqref{def:S}. We recall the concepts of consistency, monotonicity and stability. 

\paragraph{Consistency.}
Let $\Delta_m=(\Delta x_m,\Delta t_m)$ be a generic sequence of discretization parameters, $(t_{j_m},x_{j_m})$ be a generic sequence of nodes in the space--time grid such that, for $m\to\infty$,
\begin{equation}\label{hpsuc}
(\Delta x_m,\Delta t_m) \to 0 \quad\mbox{and}\quad (t_{n_m},x_{j_m})\to (t,x).
\end{equation}
The scheme $S$ is said to be {\em consistent} if for any  $\phi \in C^\infty((0,T]\times\R^d)$, we have
% \begin{subequations} \label{Consistence}
% \begin{eqnarray}
% \displaystyle \liminf_{m\to \infty} \frac{\phi(t_{n_m},x_{j_m})-S_{j_m}(\Delta_m;\phi(t_{n_m-1},\cdot)) }{\Delta t_m} & \geq & \phi_t(t,x)+ \underline H(t, x, \phi(t,x), D\phi(t,x)), \label{Consistenzadebole1} \\
% \displaystyle {\limsup_{m\to \infty}} \frac{\phi(t_{n_m},x_{j_m})-S_{j_m}(\Delta_m;\phi(t_{n_m-1},\cdot))}{\Delta t_m} &\leq&  \phi_t(t,x)+ \overline H(t, x, \phi(t,x), D\phi(t,x)), \label{Consistenzadebole2}
%  \end{eqnarray}
%  \end{subequations}
% where $\underline H$ and $\overline H$ denote respectively the lower and upper semicontinuous envelopes of $H$.

% In \eqref{Consistence}, the index of the sequence is $m$, while $j_m$ and $n_m$ denote the corresponding node indices with respect to the $m$th space-time grid. Note that, when the Hamiltonian is continuous, then conditions \eqref{Consistence}
% are equivalent to 
%
\begin{eqnarray} \label{Consistence}
\displaystyle \lim_{m\to \infty} \frac{\phi(t_{n_m},x_{j_m})-S_{j_m}(\Delta_m;\phi(t_{n_m-1},\cdot)) }{\Delta t_m} = \phi_t(t,x)+ H(t, x, D\phi(t,x)),
\end{eqnarray}
%%%% which is the usual notion of consistency for finite 
%%%% difference schemes.

\paragraph{Monotonicity.}
The scheme $S$ is said to be {\em monotone} if, for any couple of vectors $V$ and $W$ such that $V_j\geq W_j$:
\begin{equation}\label{Mon}
S_{j}(\Delta;V) \geq S_{j}(\Delta;W),
\end{equation}
for any $\Delta x$, $\Delta t$ satisfying suitable compatibility conditions, that  are typically in the form of the so-called {\em Courant--Friedrich--Levy (CFL) conditions}.

It is also possible to give a generalized form of definition of monotonicity, to treat some case of high-order scheme. We refer the reader to the discussion carried out in \cite{Bokaetal15}.
\medskip

Given a numerical solution $V^n$, we define its piecewise constant (in time) interpolation $V^{\Delta}$ as
\begin{equation}\label{eq:vdelta}
V^{\Delta}(t,x)=
\begin{cases}
I[V^n](x) &
{\mbox{if} }\;t\in \left[t_n,t_{n+1}\right),\\
v_0(x) & {\mbox{if} }\;t\in[0,\Delta t),
\end{cases}
\end{equation}
where $I[V^n](x)$ denotes an interpolation of the node values in $V^n$, computed at $x$. We remark that the interpolation operator has to satisfy itself a monotonicity property to obtain a monotone scheme (this holds, for example, for a piecewise linear reconstruction).

We can now state (in a slightly rephrased form) the convergence result given in \cite{BarSou91}:

\begin{thm} \label{Th:convergenceBS}
Assume that \eqref{PG} satisfies a comparison principle, and let $v(t,x)$ be the unique viscosity solution of \eqref{PG}. Assume that \eqref{Consistence} and \eqref{Mon} hold. Assume in addition that the family $v^{\Delta t}$ is uniformly bounded in $L^\infty$. Then, $V^{\Delta}(t,x)\rightarrow v(t,x)$ locally uniformly on $\R^d \times [0,T]$ as $\Delta \rightarrow 0$.
\end{thm}

This result directly applies to the most classical cases of monotone schemes, as in the examples below.

\paragraph{Finite difference schemes.}
Given a numerical Hamiltonian $\cH: [0,T]\times \R^d\times\R^d\times\R^d\converge\R $, we define an explicit scheme (see \cite{CL84}) as follows:
%\begin{subequations}\label{eq:schemeUP}
%\be
%  V^{n+1}_{i} 
%    \ = V^{n}_{i} - \Delta t\,\cH(t_n,x_i,\,D^-V^{n}_{i},\,D^+ V^{n}_{i}),
%    \qquad i\in \Z.
%  \label{eq:scheme.a} 
%  %& & \hspace{3cm}  n\in\{N-1,N-2,\dots,1,0\},\ %(y_i,z_j)\in\cG,  \nonumber \\ 
%%  & & \hspace{-0.5cm} w_{i,j}^{N} = \psi_{i,j},\ %(y_i,z_j)\in\cG. \label{eq:scheme.b}
%\ee
%\be
%  & & \hspace{-0.5cm} V_{i}^{0} = v_0(x_i),\ i\in \Z^d. %\label{eq:scheme.b}
%\ee
%\end{subequations}}
%Then, for $n\geq 1$ we compute recursively
\begin{subequations}\label{eq:schemeUP}
\begin{eqnarray}
&&  V^{n+1}_{i} 
    = V^{n}_{i} - \Delta t\,\cH(t_n,x_i,\,D^- V^{n}_{i},\,D^+ V^{n}_{i}),\label{eq:scheme.a} \\
&& V_{i}^{0} = v_0(x_i). \label{eq:scheme.b}
\end{eqnarray}
\end{subequations}
Here, $i\in \Z^d$, and the terms $D^- V^{n}_{i}$ and $D^+ V^{n}_{i}$ represent respectively left and right finite difference approximations of the gradient at $x_i$, defined as
$D^{\pm} V^n_i=(D^\pm_k V^n_i)_{1\leq k\leq d}$
with
\beno
   D^\pm_{k} V^n_i:= \pm \frac{V^n_{i\pm e_k}-V^n_{i}}{\Delta x_k},
\eeno
and where $\{e_k\}_{k=1,\dots,d}$ is the canonical basis of $\R^d$ 
($(e_k)_k=1$ and $(e_k)_j=0$ if $j\neq k$).

%----------------------------

For schemes of this form, and assuming that the numerical Hamiltonian $\cH$ 
is Lipschitz continuous with respect to all its arguments, consistency with $H$ comes down to the condition
$$
\cH(t,x,p,p)= H(t,x,p)
$$
and monotonicity is checked in the form
$$
\frac{\partial\cH}{ \partial p^-_k} (t,x,p^-,p^+) \geq 0, \quad
\frac{\partial\cH}{ \partial p^+_k} (t,x,p^-,p^+) \leq 0.
$$ 
These latter conditions typically require a CFL-type compatibility condition between $\Delta t$ and $\Delta x_i$.
%{\small
%\be \label{eq:CFL}
%   \Delta t \sum_{k=1}^d \frac{1}{\Delta y_k}\bigg\{
%  \bigg|\frac{\partial\cH}{ \partial p^-_k} (y,p^-,p^+)\bigg| +
%  \bigg|\frac{\partial\cH}{ \partial p^+_k} (y,p^-,p^+)\bigg|
%  \bigg\}
%  \leq 1, 
%\ee}% 
%then the numerical scheme \eqref{eq:schemeUP} satisfies all the requirements of  Theorem~\ref{Th:convergenceBS}, and it provides a sequence of solutions $V^\Delta$ that converges to the desired solution when $\Delta$ goes to $0$. 

Two classical choices for the numerical Hamiltonian are in Upwind and Lax--Friedrichs form:

\begin{itemize}

\item 
 If the Hamiltonian $H$ is defined by \eqref{eq:H}, then an upwind numerical Hamiltonian may be constructed in the form
\be \label{eq.HUP}
  \cH^{Up}(t,x,p^-,p^+)=\max_{u\in U} \left[\sum_{i=1}^d \left(\max(-f_i(t,x,u),0)\,p^-_i +\min(-f_i(t,x,u),0)\,p^+_i\right)  -\ell(t,x,u)\right].
\ee
This form fulfils consistency and monotonicity conditions for $\Delta t$ satisfying the CFL condition
\begin{equation}\label{CFL_Upwind}
\Delta t \left(\sum_{1\leq i \leq d} \frac{\max_{t,x,u}|f_i(t,x,u)|}{\Delta x_i}  \right)\leq 1. 
\end{equation}

\item
The Lax-Friedrichs scheme can be defined for a generic Hamiltonian $H$ as
\be \label{eq.HLF}
   \cH^{LF}\left(t,x,p^-,p^+\right):= H\left(t,x, \frac{p^-+p^+}2 \right) - \sum_{i=1}^d
   C_i\ \frac{p_i^+ - p_i^-}{2},
\ee
The numerical Hamiltonian $\cH^{\text{LF}}$ satisfies the monotonicity condition provided the 
constants $C_i$ are chosen such that
$C_i\geq \max_{u\in U}|f_i(t,x,u)| $
and $(\Delta t,\Delta x)$ satisfies \eqref{CFL_Upwind}.

\end{itemize}

% In one dimension, it takes the form
% \begin{equation}\label{fou_hj}
% v_j^{n+1}=v_j^n-\Delta t {\mathcal H}^{Up}\left(D_{j-1}[V^n],D_j[V^n]\right),
% \end{equation}
% where the numerical Hamiltonian ${\mathcal H}^{Up}$ is typically defined by
% \begin{equation}\label{ham_num_up}
% {\mathcal H}^{Up}(\alpha,\beta)=
% \begin{cases}
% H(\alpha)   & \text{if  } \alpha,\beta\ge\alpha_0, \\
% H(\beta)+H(\alpha)-H(\alpha_0)   & \text{if  } \alpha\ge\alpha_0,\beta\le\alpha_0, \\
% H(\alpha_0)   & \text{if  } \alpha\le\alpha_0,\beta\ge\alpha_0, \\
% H(\beta)   & \text{if  } \alpha,\beta\le\alpha_0,
% \end{cases}
% \end{equation}
% in which $H$ is given by \eqref{eq:H}, $D_{j-1}$ is the left incremental ratio,  $D_j$ is the right incremental ratio, and the generalization to higher dimension is straightforward. At a comparison with the linear case, the situation in which the speed of propagation changes sign is handled in a different form, depending on whether characteristics converge or diverge.

% \paragraph{Lax--Friedrichs scheme}

% In one dimension, the Lax--Friedrichs scheme may be recast for the HJ equation in the form
% \begin{equation}\label{lf_hj}
% v_j^{n+1}=\frac{v_{j-1}^n+v_{j+1}^n}{2}-\Delta t H\left(D_j^c[V^n]\right),
% \end{equation}
% where $D_j^c$ is the centered difference at $x_j$ defined by
% \begin{equation}
% D_j^c[V^n]=\frac{v_{j+1}^n-v_{j-1}^n}{2\Delta x}=\frac{D_{j-1}[V^n]+D_j[V^n]}{2},
% \end{equation}
% and which also has a direct extension to higher dimension.

\paragraph{Semi-Lagrangian schemes.}

The Semi-Lagrangian (SL) scheme is written here directly in the form suitable for the backward dynamic programming equation \eqref{sec5:HJB}--\eqref{eq:H}. In fact, the SL scheme can be derived by discretizing the Dynamic Programming Principle on a single time step:
\begin{equation}\label{sl_1}
\begin{cases}
\displaystyle v_j^{n-1}=\min_{\alpha\in U}\left\{\Delta t\ell(t,x,\alpha)+I_1[V^n](x_j+\Delta tf(t,x,\alpha)\right\}\\
v_j^N = \varphi(x_j),
\end{cases}
\end{equation}
in which $I_1[V](x)$ denotes the $\mathbb{P}_1$ (piecewise linear or multilinear) interpolate of the vector $V$ of node values, computed at the point $x$. The SL scheme is consistent, and the choice of a linear interpolation as $I_1$ implies also monotonicity of the scheme.

%%======================
\subsubsection{High-order schemes}
%%==================
While the framework of monotone schemes remains the most classical, in the last decades high-order numerical schemes for HJB equations have been developed and analyzed. Their convergence analysis relies typically on two theoretical tools:

\begin{itemize}

\item {\bf $\bm\varepsilon$-monotonicity}

The Barles--Souganidis theorem allows for an $o(\Delta t)$ monotonicity defect, making it possible to prove convergence for quasi-monotone schemes. This theory has been applied to high-order SL schemes and to filtered schemes, as in \cite{augoula00, FalFer13, Bokaetal15}.

\item {\bf Lin--Tadmor theory}

Lin--Tadmor convergence theory is inspired by the $Lip$'-stability theory for conservation laws. Here, a different concept of stability is singled out, i.e., {\em uniform semi-concavity} of numerical solutions, along with a suitable definition of consistency.
The convergence result, together with a practical application of this theory is presented in \cite{lin01}.

\end{itemize}

\paragraph{Higher-order FD schemes.}
The basic strategy for constructing high-order finite difference methods has been first proposed in \cite{osh-shu-91} and uses  a TVD Runge--Kutta method for the time discretization combined with high order approximations of the right/left derivatives $D_j^\pm[V]$ at the node $x_j$ (for example the  ENO approximation, see~\cite{osh-shu-91}).

\paragraph{Higher-order SL schemes.}

The SL scheme \eqref{sl_1} is easily extended to a higher consistency rate by replacing the $\mathbb{P}_1$ space interpolation $I_1$ with an interpolation of higher accuracy \cite{falcone02, carlini05}. In general, since characteristics are not straight lines, a more accurate method of characteristics tracking is also desirable \cite{Fal6}.
In some model cases convergence of high-order SL schemes, for both the evolutive and the stationary case, can be proved by showing their quasi-monotonicity (see \cite{ferretti02, FalFer13, Bokaetal15}).

\paragraph{Filtered schemes.}
 
The general idea of filtered schemes (which had previously appeared in the context of conservation laws as {\em flux-limiter schemes}) is to provide a clever coupling between a monotone and a high-order scheme.
Starting from a monotone scheme $S^M$, a high-order scheme $S^{HO}$ and  a bounded {\em filter function} $F:\R\rightarrow \R$, the filtered scheme $S^F$ is defined as
\begin{equation}\label{eq:FS}
  v^{n+1}_j = S^F_j(V^n) := S^{M}_j(V^n)+\epsilon\Delta t F\left(\frac{S^{HO}_j(V^n)-S^{M}_j(V^n)}{\epsilon\Delta t}\right),
\end{equation}
where $\eps=\eps(\Delta)>0$ is a parameter vanishing for $\Delta t,\Delta x\to 0$, which controls the monotonicity defect of the filtered scheme (more hints on the choice of $\eps$ and of the filter function can be found in \cite{froese13, boka16}).
In constructing the filter function, the basic idea is that ``large'' values of the ratio $\rho=(S^{HO}-S^M)/(\eps\Delta t)$ indicate a singularity (where the scheme needs to be monotone), while ``small'' values indicate a smooth region (in which the scheme can be high-order). It can be shown that, for a suitable choice of $\epsilon$, the filtered scheme converges to the viscosity solution by quasi-monotonicity.

% {\color{red}{\paragraph{Discontinuous Galerkin.}

% The application of Discontinuous Galerkin (DG) methods to HJ equations uses the relationship with conservation laws. In fact, what is discretized in this case is the conservation law (or system of conservation laws, see \cite{jin98}) associated to the HJ equation, and the outcome of the scheme are the partial derivatives of the solution, to be suitably post-processed. We refer to \cite{hu99} for the basic ideas of this technique.
% }}

\paragraph{Further comments.}
Several advances have been made  to improve the numerical schemes of approximations of HJB equations, in particular in high dimension. 
Let us mention the resolution techniques on sparse grids \cite{Bok-Gar-Gri-Klo,Gar-Kro}, on tree structures as in \cite{Alla-Fal}, or approximation by sophisticated model reduction techniques as in \cite{Alla-Fal-POD}.
We also mention another chapter \cite{Rozza} of the Volume 1 of this Handbook, dedicated to model reduction methods. 
All these methods aim at providing accurate numerical  approaches for solving efficiently  HJB equations with reasonable numerical efforts (measured by complexity of algorithms, CPU time and accuracy).

%%======================

\subsubsection{Optimal trajectory reconstruction  from the value function}

From a control viewpoint, the approximation of the value function $\cV$ has a relatively lesser interest with respect to the construction of the (approximate) optimal control.  
 We propose  in this section some algorithms that lead, given an approximations of the value function, to construct a quasi-optimal controls in feedback form. The procedure does not depend on the specific scheme used to compute $\cV$. 
For simplicity, we consider the trajectory reconstruction on the time interval $[0,T]$, although all the results remain valid for a reconstruction on any sub-interval $[t,T]$.
For $n_h\in \N$ and $h=T/n_h$, consider a partition $s_0=0<s_1<\cdots<s_{n_h}=T$ of $[0,T]$, with $s_k=kh$.
Consider a numerical approximation $f_h$  of the dynamics $f$ such that, for every $R>0$, we have
\be\label{eq:assump_fh}
  |f_h(t,x,u) - f(t,x,u)|\leq C_R h, \qquad \forall t\in [0,T],\ |x|\leq R, u\in U,
\ee
where the constant $C_R$ is independent of $h\in(0,1]$.
An approximation scheme for the differential equation $\dot \bx(t) = f(t,\bx(t),\bu(t))$ (for a constant control $u$, discrete times $s_k$ and time step $h$) can be written as
\be\label{eq:scheme_equadiff} 
  y_{k+1}=y_k + h f_{h}(s_k,y_k,u), \quad k \geq 0.
\ee
Here and in the sequel we use the notation $y_k$ to denote a state at discrete times.
The case of the Euler forward scheme corresponds to the choice $f_h:=f.$
Higher order Runge-Kutta schemes can also be written as \eqref{eq:scheme_equadiff}
and with a function $f_h$ satisfying \eqref{eq:assump_fh}. 
For instance, the Heun scheme (with constant control) corresponds to the choice
$$
  f_h(t,y,u):= \frac{1}{2} (f(t,y,u) + f(t+h,y+ h f(t,y,u),u)).
$$

\paragraph{Bolza problems.}
Consider first, the case of Bolza unconstrained problems. Let $\cV$ the exact value function defined in \eqref{eq.Pb_CO}.  Let $\cV^h$ be an approximation of $\cV$, 
and define $E_h$ as a uniform bound on the error:
$$
  |\cV^h(t,x) - \cV(t,x) |\leq E_h, \quad \forall t\in[0,T], |x|\leq R,
$$ 
with $R>0$ large enough.
The approximate feedback is defined on the basis of the approximate value function with a discrete dynamic programming procedure.
\begin{algorithm*}[!hbtp]
\caption{{\bf (TR) - Trajectory reconstruction algorithm for Bolza problems}
}
\begin{algorithmic}[1]
\REQUIRE 
First we set $y^h_0:=x$. 
\STATE For $k=0,\dots,{n_h}-1$, knowing the state $y^h_k$ 
we define
\begin{itemize}
\item[$(i)$]
 an optimal control value $u^h_k\in U$ such that 
 \begin{eqnarray}\label{cont_feed}
   u^h_k \in \argmin_{u\in U}
   \Big[\cV^h\big(s_k,y^h_k +h_k\,f_{h}(s_k,y^h_k,u)\big)+h \ell(s_k,y^h_k,u) \Big]
 \end{eqnarray}
\item[$(ii)$]
  a new state position $y^h_{k+1}$ 
  \begin{eqnarray}\label{eq:def-ypu}
    y^h_{k+1} :=y^h_k + h_k f_{h}(s_k,y^h_k,u^h_k).
  \end{eqnarray} 
\end{itemize}
\RETURN 
a piecewise constant control $\bu^h(s):= u^h_k$  on $[s_k,s_{k+1}[$,
and a piecewise linear trajectory $\by^h$ such that 
$\by^h(s_k)=y^h_k$. 
\end{algorithmic}
\end{algorithm*}

Following some arguments introduced in \cite{row-vin-91},  it can be shown that  any cluster point of $({\by}^h)_{h>0}$ 
is an optimal trajectory that realizes a minimum in the definition of the original control problem. 

\begin{thm}\label{th:traj-algo1} Assume \Hyp{0}, \Hyp{1}, \Hyp{2} and \Hyp{4}.
Assume also that the approximation \eqref{eq:assump_fh} is valid
and the error estimate 
$E_h=o(h)$. 
Let $(y,z)$ be in $\R^d\times\R$ and 
let $(y^h_k)$ be the sequence generated by Algorithm (TR). 

$(i)$ The approximate trajectories $(y_k^h)_{k=0,\dots,n_h}$ constitute a minimizing sequence in the following sense:
\be\label{eq:res0-1}
  \cV(0,x)
    & = & \lim_{h\converge 0} \bigg(  \varphi(y_{n_h}^h)+h\sum_{k=0}^{n_h}\ell(s_k,y_k^h,u^h_k) \bigg).
\ee

$(ii)$ Moreover, the family $({\by^h})_{h>0}$ admits cluster points, for the $L^\infty$ norm, when $h\converge 0$.
Any such cluster point $\bf\bar y$ is an admissible trajectory
and ${\bf\by}$ is an optimal trajectory for $\cV(0,x)$.
\end{thm}

Let us emphasize that the condition $E_h=o(h)$ indicates that the approximation $\cV^h$ should be provided with a given precision. Typically, a
numerical  scheme would provide an approximation $V^\Delta$ for $\Delta=(\Delta t,\Delta x)$. Under a CFL condition, the error estimate is of order 
$O(\sqrt{\Delta t})$. To ensure that $\|V^\Delta-\cV\|=o(h)$, it suffices to take $\sqrt{\Delta t}=o(h)$.

\paragraph{Minmax problems.}
The same idea as in Algorithm
(TR) can be adapted for minmax control problems (see section~\ref{Other_OCP}). 
Let  $\cV^{\#,h}$ be an approximation of the 
value function $\cV^\#$. Here, the  function $\cV^{\#,h}$
could be again a numerical approximation obtained by solving a discretized form of the HJB equation \eqref{eq.HJB.minmax}.

\medskip

\begin{algorithm*}[!hbtp]
\caption{{\bf (TRM) - Trajectory reconstruction for Minmax problems}
}
\begin{algorithmic}[1]
\REQUIRE Set $y^h_0=x$.
\STATE 
Define the positions  $(y^h_k)_{k=0,\dots,{n_h}}$, and control values $(u^h_k)_{k=0,\dots,{n_h}-1}$, by recursion as follows.
or $k=0,\dots,{n_h}-1$, knowing the state $y^h_k$ 
we define
\begin{itemize}
\item[$(i)$]
 an optimal control value $u^h_k\in U$ such that 
 \begin{eqnarray}\label{eq:feedcontr1}
   u^h_k \in \argmin_{u\in U}
   \cV^{\#,h}\big(s_k,y^h_k +h_k\,f_{h}(s_k,y^h_k,u), z \big) \bigvee  \Psi(s_k,y^h_k)
 \end{eqnarray}
\item[$(ii)$]
  a new state position $y^h_{k+1}$ 
  \begin{eqnarray*}\label{eq:def-ypu2}
    y^h_{k+1} :=y^h_k + h_k f_{h}(s_k,y^h_k,u^h_k).
  \end{eqnarray*} 
\end{itemize}
\RETURN 
a piecewise constant control $\bu^h(s):= u^h_k$  on $[s_k,s_{k+1}[$,
and a piecewise linear trajectory $\by^h$ such that 
$\by^h(s_k)=y^h_k$. 
\end{algorithmic}
\end{algorithm*}
Note that in~\eqref{eq:feedcontr1} the value of $u^h_k$ can also be defined as a minimizer 
of $ u \converge \cV^{\#,h} \big(s_k,y^h_k +h_k\,f_h(s_k,y^h_k,u), z \big)$,
since this will imply in turn to be a minimizer of~\eqref{eq:feedcontr1}

\begin{thm}\label{th:traj-1} Assume \Hyp{0}, \Hyp{1}, \Hyp{2} and \Hyp{4} hold true.
Assume also that the approximation \eqref{eq:assump_fh} is valid
and that $\cV^{\#,h}$ is an approximation of $\cV^\#$ with  error estimate
$\|\cV^\#-\cV^{\#,h}\|=o(h)$. 
let $(y^h_k)$ be the sequence generated by Algorithm~2. 

$(i)$ The approximate trajectories $(y_k^h)_{k=0,\dots,n_h}$ constitute a minimizing sequence in the following sense:
\beno%\label{eq:res0-1}
  \cV^\#(0,y)
    & = & \lim_{h\converge 0} \bigg( \max_{0\leq k\leq {n_h}} \Psi(s_k,y_k^h)\bigg) \bigvee \varphi(y_{n_h}^h).
\eeno

$(ii)$ Moreover, the family $({\by^h})_{h>0}$ admits cluster points, for the $L^\infty$ norm, when $h\converge 0$.
For any such cluster point $\bf\bar y$,  we have ${\bf\bar y}\in \cS_{[0,T]}(y)$ 
and ${\bf\bar y}$ is an optimal trajectory for $\cV^{\#}(0,x)$.
\end{thm}

%%%

% \subsection{Recent developments}
% \begin{itemize}
%     \item Optimistic methods, reinforcement learning
%     \item State space reduction
%     \item ...
% \end{itemize}

\subsection{Numerical examples}
\paragraph{Test 1 - Unconstrained Zermelo problem} In this example, we consider the same setting as in Section~\ref{sec.3.4}.  To compute the minimal time function, we use the level-set approach  described in Section 4.2. We consider 
a domain of computation large enough to contain the initial position $(0,0)$ and the final target $y_f=(20,1)$. More precisely, the computation will be performed on $\cD:=[-1,21]\times [-0.5,1.5]$. 
We use the finite difference scheme \eqref{eq:schemeUP} combined with ENO approximation  and the  Lax-Friedrichs numerical Hamiltonian  \eqref{eq.HLF}.
Figure~\ref{fig.HJB.Zermelo1} shows the optimal solution computed on a uniform grid with $500\times 100$ nodes. The optimal time to steer the system from the initial point $(0,0)$ to the final state $y_f=(20,1)$ is $4.94916$.

While the approximation of the optimal trajectory on Figure~\ref{fig.HJB.Zermelo1}(left) seems quite accurate, the optimal control law in Figure~\ref{fig.HJB.Zermelo1}(right)  presents some oscillations. 
This happens because the control law is constructed in a ``blind'' way: in fact, in the reconstruction algorithms, at each time step, the control value is computed to follow a minimal path. It may happen that several values of the control lead to the same position of the trajectory. The reconstruction process picks one of the optimal control values and this arbitrary choice may generate oscillations. Notice also that Theorems~\ref{th:traj-algo1} and \ref{th:traj-1} state the convergence of the reconstructed optimal trajectories but not the optimal control laws (convergence of control laws would require additional assumptions).

In this example, the minimum time function $\cT$ is continuous on its domain, as it can be seen on Figure~\ref{fig.HJB.Zermelo1_Tmin}, where some level-sets of the $\cT$ are presented.

\begin{figure}[!ht]
    \centering
    \includegraphics[width=0.9\textwidth,height=0.3\textheight]{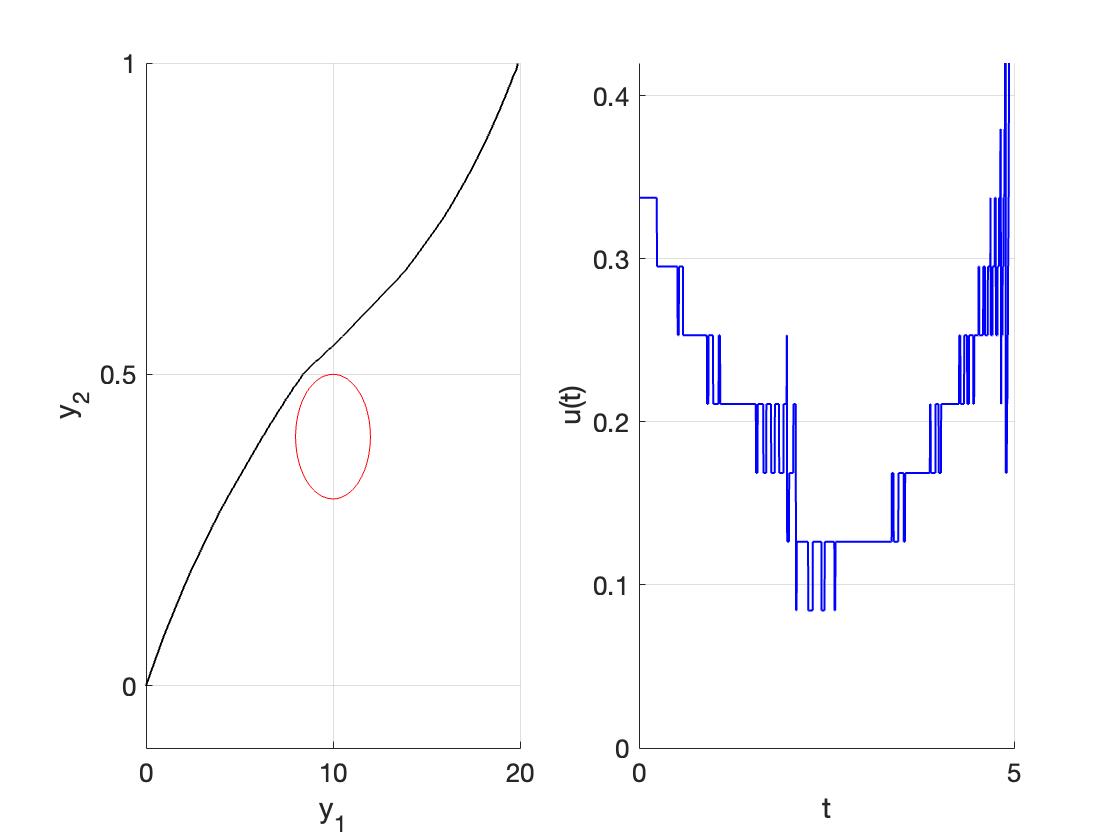}
    \caption{\small(Test 1) Approximation of the optimal trajectory for Zermelo problem with HJB approach.}
    \label{fig.HJB.Zermelo1}
\end{figure}

\begin{figure}[!ht]
    \centering
    \includegraphics[width=0.9\textwidth,height=0.3\textheight]{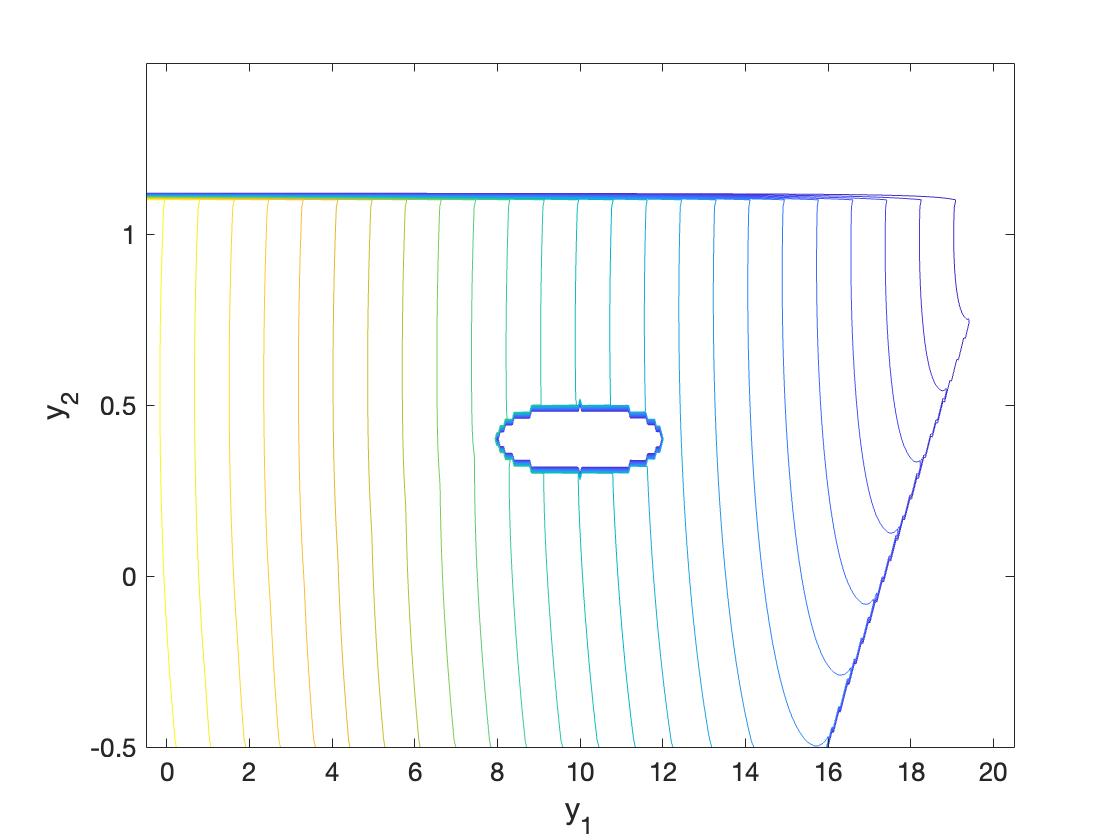}
    \caption{\small (Test 1) Level sets of the minimum time function on the domain $[-0.5, 20]\times [-0.5,1.1]$.}
    \label{fig.HJB.Zermelo1_Tmin}
\end{figure}
Recall that  for this version of Zermelo problem, the shooting method provides a globally optimal trajectory or a locally optimal trajectory depending on the initialization of the adjoint state (see Section 4). By using the HJB approach, the reconstructed trajectory is an approximation of the global optimum, even when performing the calculations on very coarse grids. Note also that an approximation of the derivative of the minimal time function at the initial position (0,0) can be computed. In our simulation, we obtain
$$D_x\cT^\Delta(0)= (-0.24585 ,  -0.09163)^{\sf T},
\qquad \mbox{and } \cT(0)=4.94916.$$
 An initialization of the shooting method with this vector allows the method to converge towards the global solution in very few iterations.
 
\paragraph{Test 2 - Unconstrained Zermelo problem -  Case with a strong current.}
Here, we consider a variant of the Zermelo problem where the dynamics is given by
\begin{subequations}\label{eq:dyn_zermelo_Unconstrained}
\begin{eqnarray}
  & & \dot \by_1(s) =  \bv(s) \cos(\bu(s)) + 2-\frac12\by_2^2(s),\\
  & & \dot \by_2 (s) =  \bv(s) \sin(\bu(s)),
\end{eqnarray}
\end{subequations}    
where the control inputs are
the speed $\bv(s)$  and the angle of orientation $\bu(s)$ of the boat. The set of control values is $U:= [0,1]\times [0,2\pi]$.
In this example, the  drift is strong in the middle of the channel $\R\times [-2,2]$ and is zero along the channel banks.  
The target is a ball centred at the origin and with radius $r=0.1$. The domain of computation is 
$\cD:=[-5,2]\times[-2,2]$. 
\medskip

- {\bf Test 2-1}:   The numerical simulations are
performed on a uniform grid with $N_{x_1}=N_{x_2}=100$ nodes on each axis. The numerical results of this test are displayed in Figure~\ref{fig:zermelo_unconstrained_N=100}. The left-hand side of  Figure~\ref{fig:zermelo_unconstrained_N=100} shows level sets of the minimum time function. Some  trajectories starting from different initial positions are given in the right-hand side of 
Figure~\ref{fig:zermelo_unconstrained_N=100}. 
\medskip

- {\bf Test 2-2 }:  Here, the numerical simulations are
performed on a uniform grid with $N_{x_1}=N_{x_2}=500$ nodes on each axis. The numerical results of this test are given in Figure~\ref{fig:zermelo_unconstrained_N=500}.

\begin{figure}[!ht]
    \centering
    \includegraphics[width=0.48\textwidth]{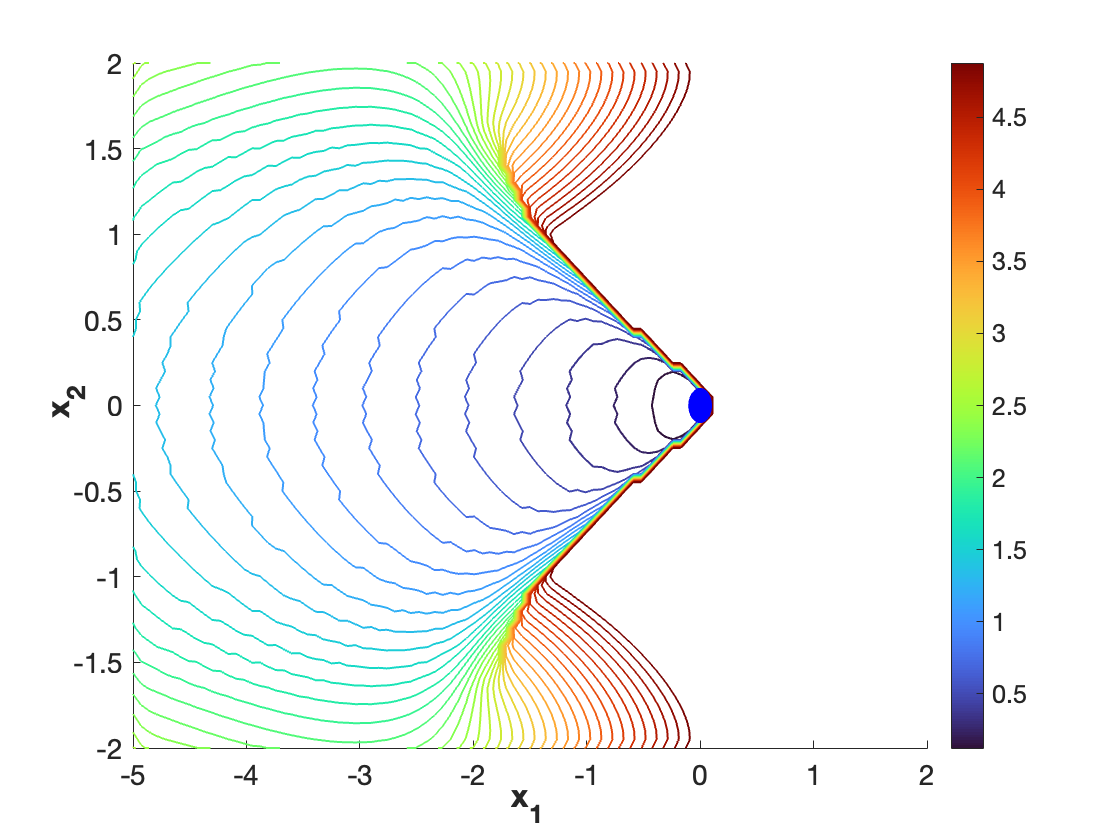}
    \includegraphics[width=0.48\textwidth]{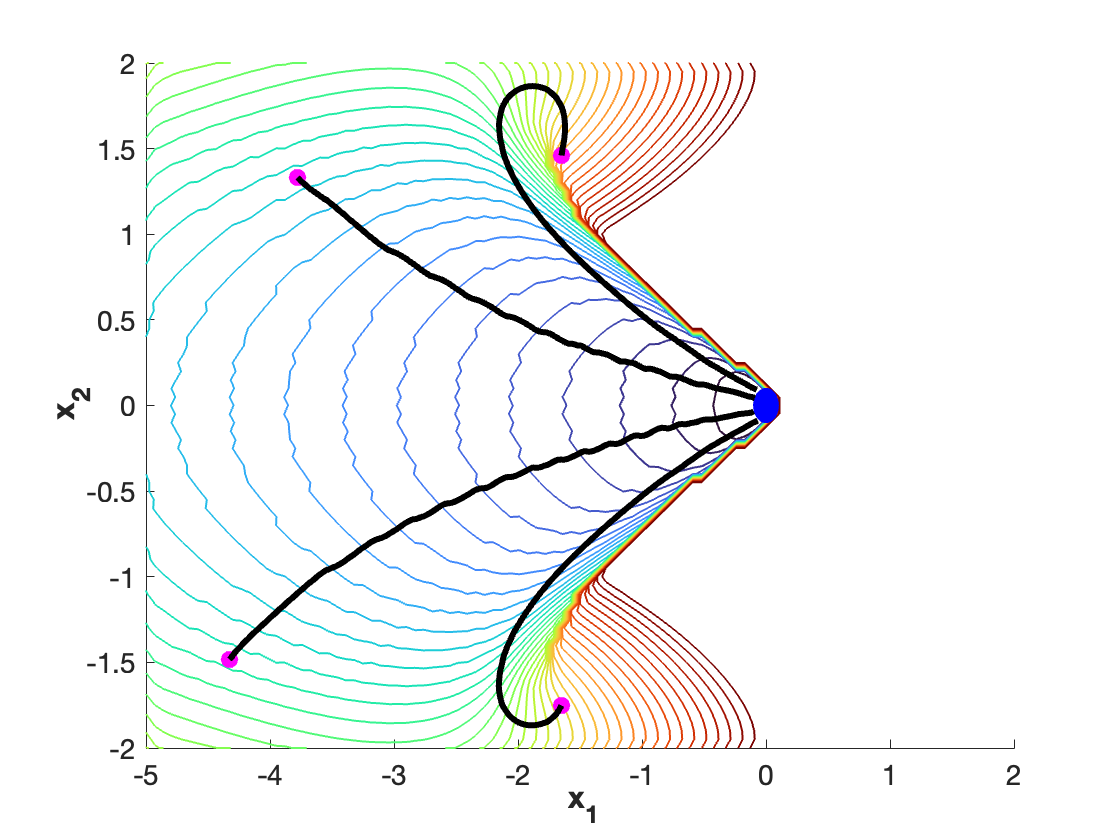}
    \caption{(Test 2-1) Zermelo problem. The figure on the left-hand side  shows some level-sets of the minimum time $\{x\mid \cT(x)=c\}$ for $c$  between $0$ and $5$.  The figure on the right displays some trajectories (the black curves) starting from different initial  positions. The simulations presented in this figure are performed on a grid with  $N_{x_1}=N_{x_2}=100$}\label{fig:zermelo_unconstrained_N=100}
\end{figure}

\begin{figure}[!ht]
    \centering
    \includegraphics[width=0.48\textwidth]{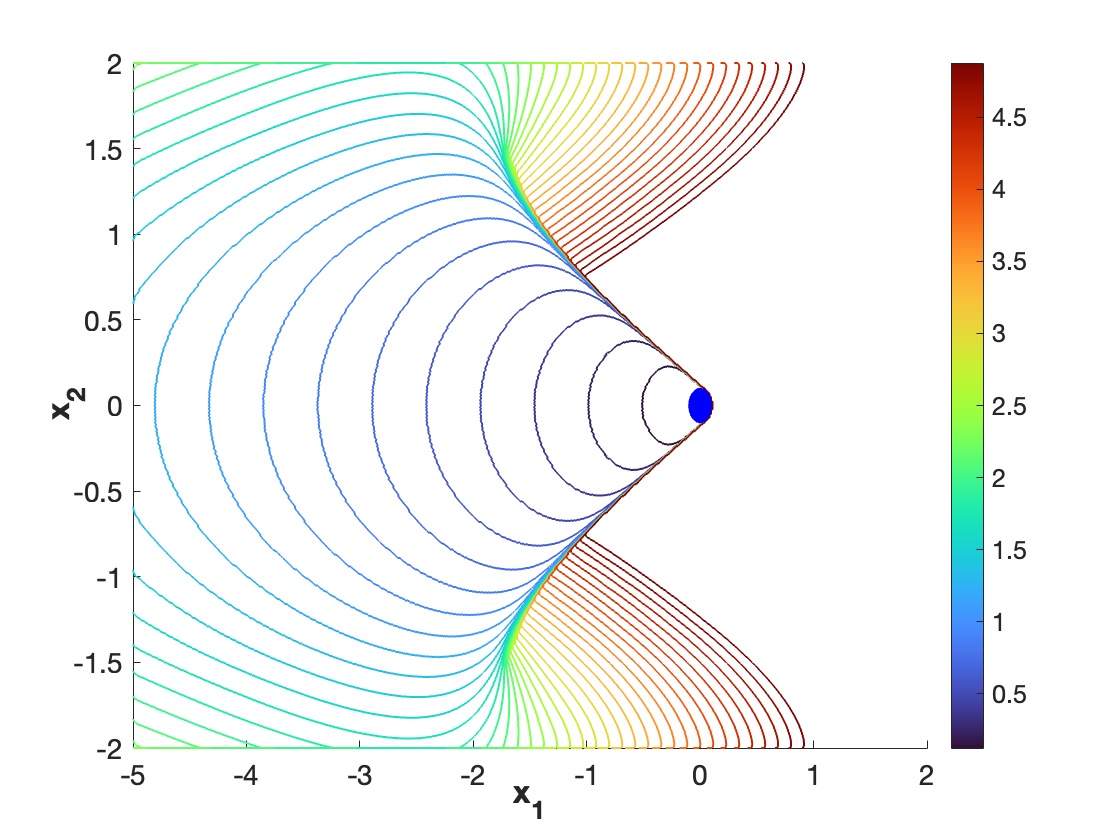}
    \includegraphics[width=0.48\textwidth]{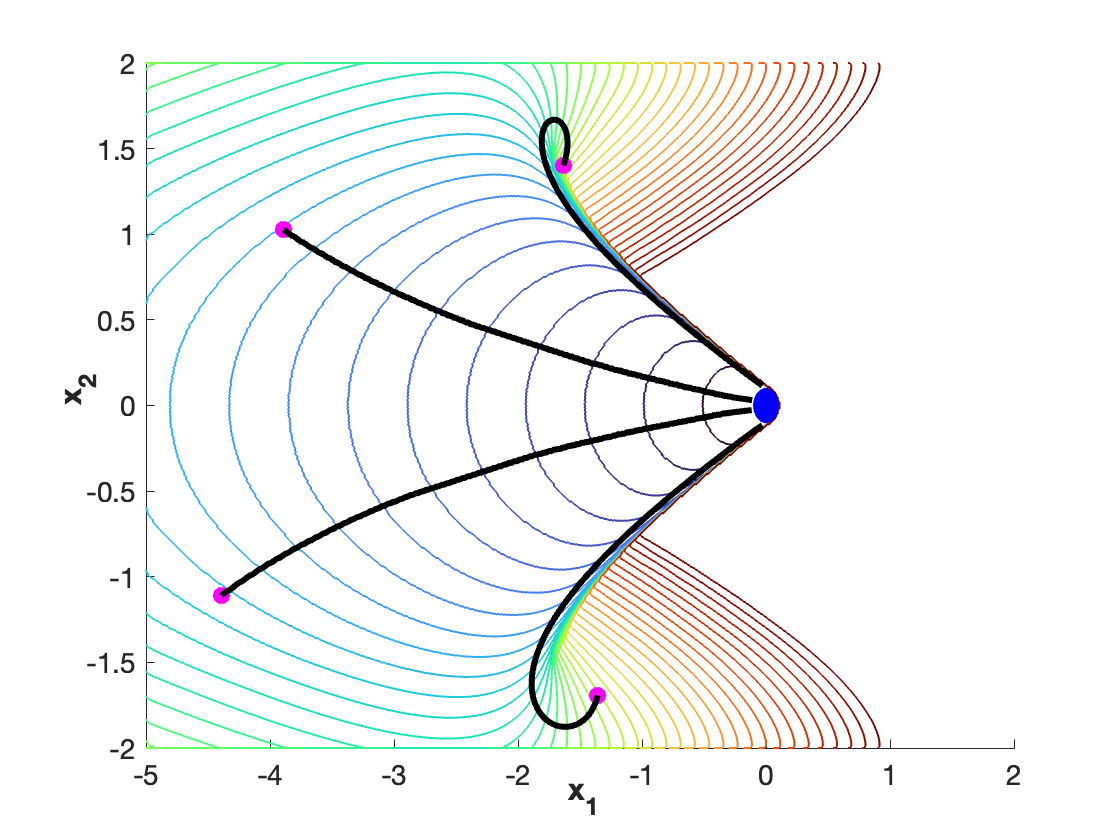}
    \caption{(Test 2-1) Zermelo problem. The figure on the left-hand side  shows some level-sets of the minimum time $\{x\mid \cT(x)=c\}$ for $c$  between $0$ and $5$.  The figure on the right displays some trajectories starting from different initial  positions. The simulations presented in this figure are performed on a grid with  $N_{x_1}=N_{x_2}=500$}\label{fig:zermelo_unconstrained_N=500}
\end{figure}
%    
%\medskip

Numerical convergence can be observed when refining the size of the grid (i.e., at the increase of $N_{x_1}$ and $N_{x_2}$).

\paragraph{Test 3 - Constrained Zermelo problem.}
Consider again the same dynamics for Zermelo problem as in the previous paragraph. Now, 
 the state is required to avoid the rectangular obstacle  $[-2.5,-1.5]\times [-0.5,0.5]$. Figure~\ref{fig:zermelo_cons} shows the level sets of the minimum time function and some samples of optimal trajectories.  Figure~\ref{fig:zermelo_cons} (left) displays the results obtained with $250\times250$ nodes on the domain of computation, while  Figure~\ref{fig:zermelo_cons} (right) corresponds to computation on grid of 
 $500\times500$ nodes.  Notice that the optimal trajectories avoid the obstacle (in red) but they tend to get closer to the central section of the channel, in which the current  is stronger. Moreover, gradient singularities upstream of the obstacles indicate points at which the optimal trajectory is not unique (e.g., it can go either left or right of the obstacle). 
    
\begin{figure}[!ht]
    \centering
    \includegraphics[width=0.45\textwidth]{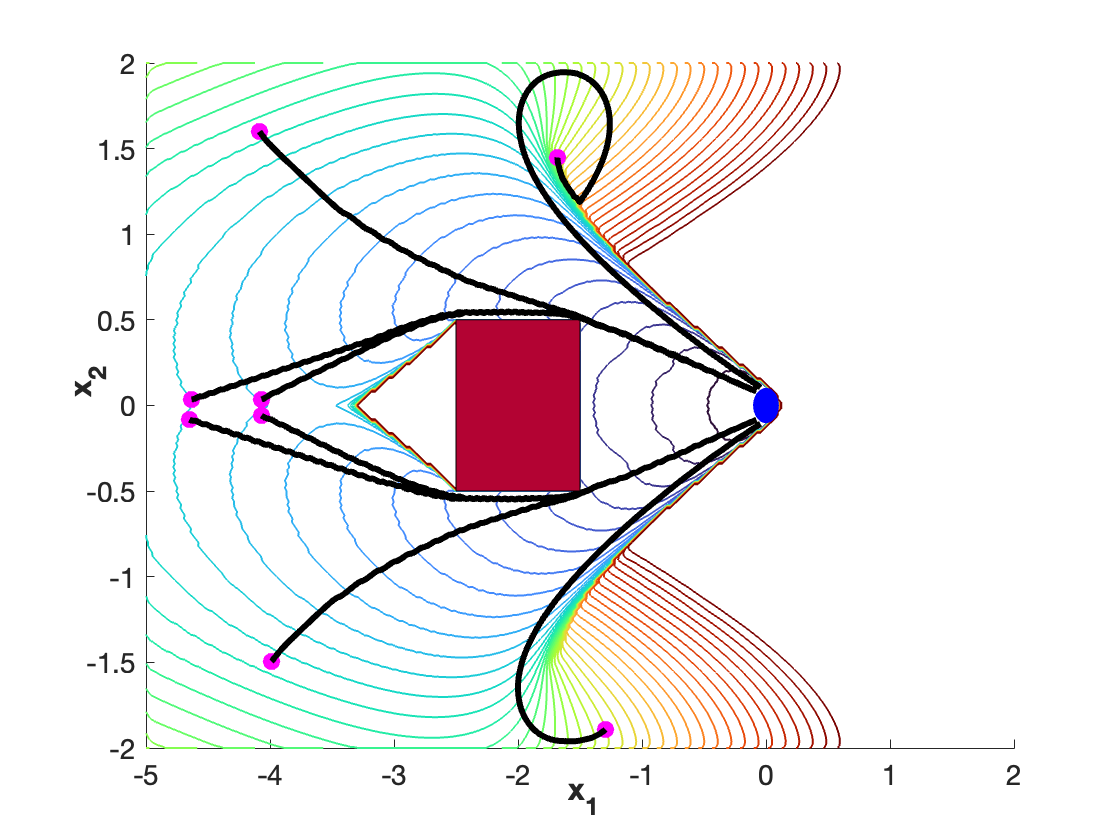}
    \includegraphics[width=0.45\textwidth]{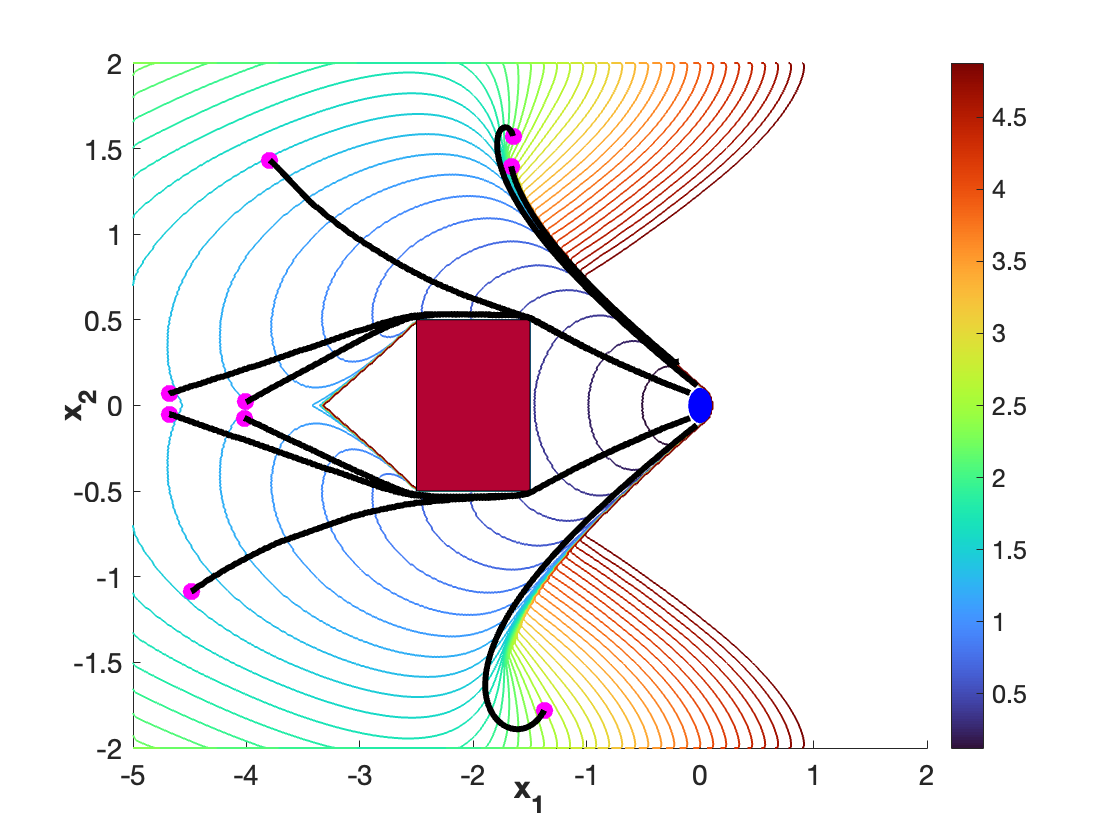}
    \caption{(Test 3) Zermelo problem with an obstacle. The two figures display level sets of the  minimum time function $\{x\mid {\cal T}(x)=c\}$ for some values of $c$ between $0$ and $5$, and some optimal trajectories starting from different initial positions. The figure on the left-hand side displays the numerical results obtained on a grid with $N_{x_1}=N_{x_2}=250$. The results on a grid with $N_{x_1}=N_{x_2}=500$ are given in the right-hand side of  the figure}\label{fig:zermelo_cons}
\end{figure}

\paragraph{Test 4 - Goddard Problem.}
Now, consider the Goddard problem as described in Sections~\ref{s32} and \ref{Section.Goddard_Tir}. 
The problem is with free final time and state constraints. Since the inward pointing condition is not satisfied, we reformulate the problem as described in Section~\ref{Aux_CO}.
We compute the auxiliary value function
(in dimension 4)
on  $\cD=[1.0,1.2]\times[0,0.12]\times[0,1]\times [0,1]$.  We use a uniform grid with 
 $N_x^4$ node points (i.e., $N_x$ points on each axis).
In the sequel, we will choose $N_x=20$
and $N_x=40$ so that the computation of the value function and the  optimal trajectory is performed in less than 1 minute.

\begin{figure}[!ht]
    \centering
    \includegraphics[width=1\textwidth]{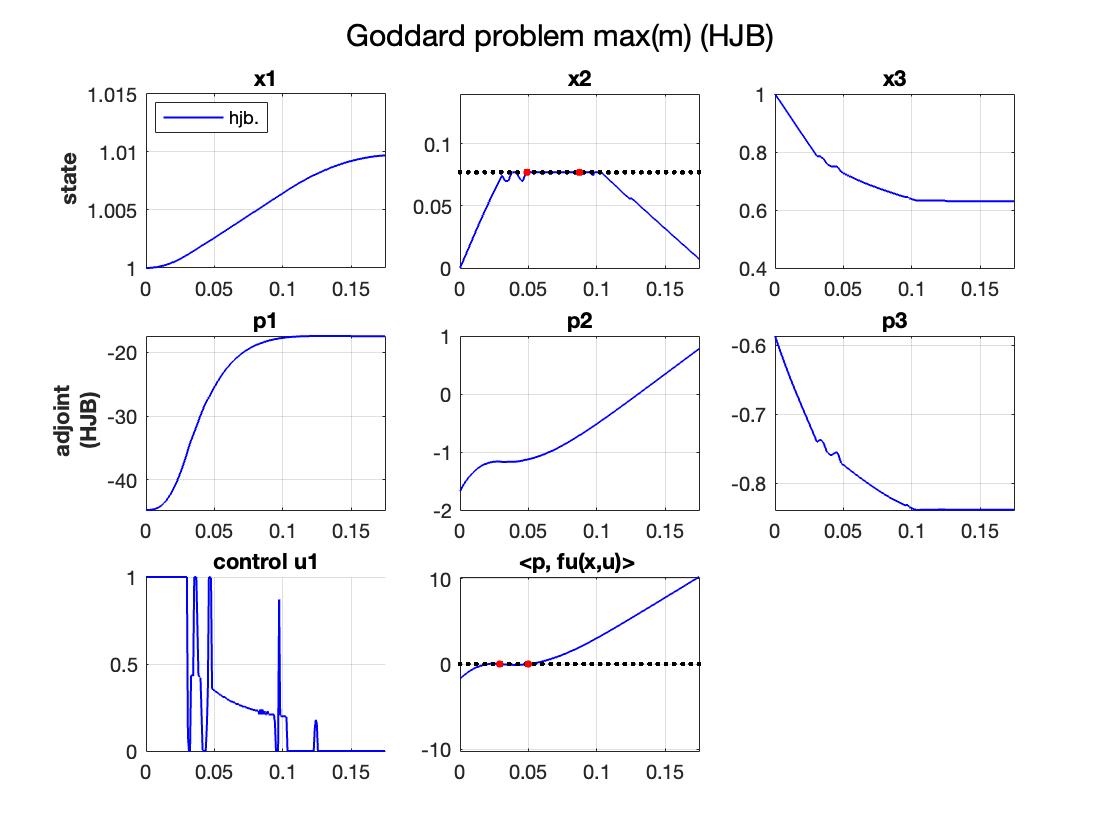}
    \caption{(Test 4) Goddard problem solved by HJB approach, with $N_x=20$}    \label{fig:Goddard20}
\end{figure}

Figure~\ref{fig:Goddard20} corresponds to a simulation with $Nx=20$, while Figure~\ref{fig:Goddard40} corresponds to a simulation with $Nx=40$. 
In both figures, we show the three state variables in the top line, and the adjoint states in the middle line. In the third line, the control variable is displayed on the left. We present also, in the middle of the third  line, an approximation  of the 
derivative of the Hamiltonian with respect to variable $u$. 

In Figures~\ref{fig:Goddard20}--\ref{fig:Goddard40}, 
we notice that  the derivative of the Hamiltonian w.r.t. the control variable $u$ vanishes identically on a time interval whose entry and exit times are indicated by red dots (middle of the third line). On this time interval,  the control law is singular. Moreover, we notice that the constraint on the velocity $\bx_2=\bv$ is saturated on another time interval (the entry and exit times are indicated by red dots in the figure situated in the middle of the first line).

\begin{figure}[!ht]
    \centering
    \includegraphics[width=1\textwidth]{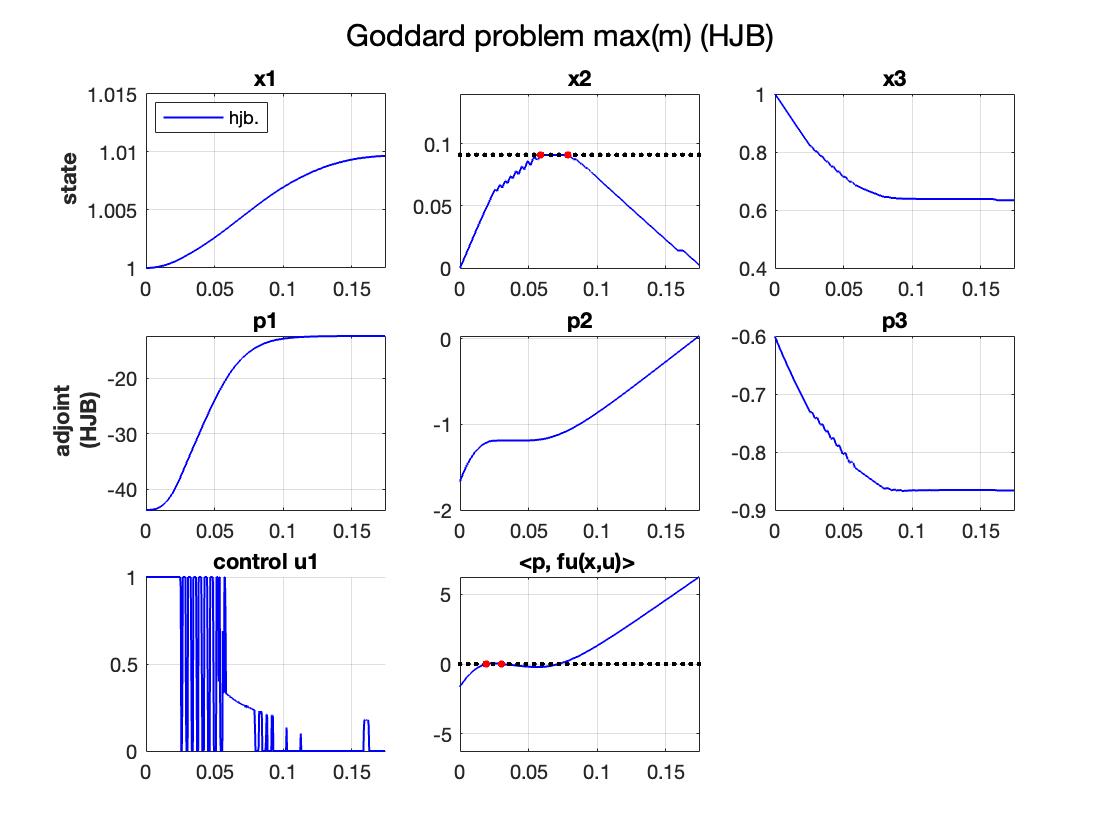}
    \caption{(Test 4) Goddard problem solved by HJB approach, with $N_x=40$}
    \label{fig:Goddard40}
\end{figure}

In this example, we observe again that the approximation of the control law presents high oscillations (for the same reasons we mentioned in Test 2). We notice that the oscillations persist even when $N_x$ increases. These oscillations occur mainly during the time interval when the trajectory is singular-constrained.
Besides, we observe  a numerical convergence of the trajectories  when refining the size of the grid (i.e., increasing $N_{x}$). 
At the increase of the number of nodes, CPU time and memory used also increase, but we notice that even a coarse grid calculation, by HJB approach, provides interesting results that can at least serve as an initialization for a more precise method like the shooting method. The computation of the value function and the reconstruction of the optimal trajectory give an approximation of 

- the co-state $\bp(0)=(\bp_1(0),\bp_2(0),\bp_3(0))$ (by using the sensitivity relations in Section~\ref{sec-PMP-HJB});

- the entry point of the singular arc $t_1$, the entry and exit points of the boundary arc $t_2$ and $t_3$, and the free final time $t_f$.

In  Table~\ref{Table:PMP-HJB}, the  second row  presents the numerical results obtained by HJB
simulations on a grid of $N_x=20$ points on each axis. The fourth row presents the numerical  results  obtained by the shooting method (initialized with the values obtained by  the HJB simulation). 
The convergence of the shooting method requires 43 iterations. 
The adjoint vector being defined up to a multiplicative constant, we give in the third row of Table~\ref{Table:PMP-HJB} the results of HJB with a re-normalization of the adjoint vector. This makes it easier to compare the results of the  HJB simulation and those of the shooting method. 
\begin{table}[!ht]
\centering
\begin{tabular}{|m{2.4cm}||m{1.4cm}|m{1.4cm}|m{1.4cm}||m{1.4cm}|m{1.4cm}|m{1.4cm}|m{1.4cm}|}
\hline
  & $\bp_1(0)$  & $\bp_2(0)$  & $\bp_3(0)$ & $t_1$ &   $t_2$ & $t_3$ & $t_f$ \\ \hline \hline 
 HJB simulation &    
    5.205e1 & 1.947e0 & 
    6.826e-1 & 2.912e-2 &
    4.980e-2 & 8.735e-2 & 
    1.747e-1   \\ \hline \hline
HJB simulation
with a re-normalization of $p(0)$&    
3.945e0 &  1.476e-1  &  5.174e-2
     & 2.912e-2 &
    4.980e-2 & 8.735e-2 & 
    1.747e-1   \\ \hline \hline    
Shooting method &    
    3.945e0& 
    1.504e-1 & 5.371e-2 & 2.351e-2 &
    5.974e-2 & 
    1.016e-1& 
    2.020e-1 \\ \hline   
\end{tabular} 
\caption{Godard problem. The second row  presents the results obtained by HJB
simulation on a grid with $N_x=20$ points on each axis. 
The third row presents the same results of row 2 with a re-normalization of the costate vector. 
The fourth row presents the  results  obtained by the shooting method (initialized by the HJB results given in row 2). \label{Table:PMP-HJB}}
\end{table}

%  3.945764658650458 & 0.15039559623172183 & 0.05371271293984309 & 0.0235096840432302 &0.05973738090016676 & 0.10157134842379754& 0.20204744056499138

%% file: 050-OP.tex
\section{Optimistic planning algorithms}
\label{s42}
%%%%%%----------%%%%%%%%%

As mentioned previously, direct and indirect methods are quite simple to implement, and provide locally optimal solutions with high accuracy.  These methods depend  on the initialization -- especially the shooting method  which is particularly sensitive to the initialization and also requires an \emph{a priori} knowledge of the optimal trajectory structure (existence of bang and/or singular  and/or saturated arcs). On the other hand, the HJB approach always provides a global optimum, but, if ever feasible, it requires a greater computational effort because of the high dimension of the space in which the value function must be computed.

A further approach that we present here is a global approach based on a  discretization in the {\em space of controls}, combined with \emph{optimistic planning} (OP) algorithms \cite{ref52,ref51}
(without  requiring any discretization of the state space). This approach is interesting especially for applications where the control dimension $r$ is  lower compared to the state dimension $d$.   
On a given discretization of the time interval, our approach will seek to identify the best control strategy to apply on each time sub-interval. The OP methods perform the optimal control search by branch and bound on the control set, always refining the region with the best lower bound of the optimal value (this is what justifies  the term ``optimistic'').  An interesting feature of these algorithms is the close relationship between computational resources and quasi-optimality,  which exploits some  ideas of reinforcement learning \cite{mun-2014}. 
Indeed, for  given computational resources,  the OP approaches provide a sub-optimal strategy
whose performance is {\em close } to the optimal value (with the available resources). 

First, for $N\geq 2$, consider a uniform partition of $[0,T]$ with $N+1$ time steps: $t_k=k\Delta t$, $k=0,\dots,N$,
where $\Delta t=\frac{T}{N}$ is the step size. 
For a sequence of actions $\bu=(u_k)_{0\leq k \leq N-1}\in U^N$, we consider $(y^{x,\bu}_k)_{0\leq k \leq N}$ the trajectory solution of the discrete-time dynamical system
\begin{equation}
  \label{eq:discrete}
\begin{cases}
    y_0=x,\\
    y_{k+1}=F_k(y_k,u_k)\quad k=0,...,N-1,
    \end{cases}
\end{equation}
where $F_k(x,u)$ is an approximation of the solution to the system 
$\dot\bx(s)=f(s,\bx(s),u)$ on  $(t_k,t_{k+1})$,
 with the initial condition $\bx(t_k)=x$. More precisely, we assume that 
 $$\left\|F_k(x,u)-x-\int_{t_k}^{t_k+1}f(s,\bx(s),u)\,ds\right\| = \mathrm{O}(\Delta t) \quad \mbox{ for all } x\in \R^d, \ u\in U.$$
Consider also an instantaneous cost function $L_k$ 
that approximates the integral of $\ell$ over an interval $[t_k,t_{k+1}]$, for $k=0,...,N-1$:

$$\left\|L_k(x,u)-\int_{t_k}^{t_k+1}\ell(s,\bx(s),u)\,ds\right\| = \mathrm{O}(\Delta t) \quad \mbox{ for all } x\in \R^d, \ u\in U.$$

In this section, we assume that \Hyp{0}-\Hyp{4} are satisfied and that  $f$ and $\ell$ are Lipschitz continuous with respect to the control variable. The approximations $F_k$ and $L_k$ are also assumed to be  Lipschitz continuous:
%3mm}
\begin{eqnarray*}
& &  \|F_k(x,u)-F_k(x',u')\| \leq L_{F,x} \|x-x'\| + L_{F,u} \|u - u'\|,\\ 
& &  |L_k(x,u)-L_k(x',u')| \leq L_{L,x} \|x-x'\| + L_{L,u} \|u - u'\|,
\end{eqnarray*}
%\MODIFA{
for every $x\in \R^d$ and $u\in U$. Moreover, we assume that the Lipschitz constants of $F_k$ and $L_k$ are related to the Lipschitz constants of $f$ and $\ell$ by the following relations
%3mm}
 \begin{subequations}\label{eq.LFxLFa}
 \begin{eqnarray}
 & &  L_{F,x}  := 1 + \Delta t L_{f,x}C \quad   L_{F,u}  := \Delta t L_{f,u} C \ \ \mbox{and} \\
 & &
  L_{L,x}  := \frac{\Delta t }{2}L_{\ell,x} C,  \quad  L_{L,a}  := \Delta t (L_{\ell,a}+C),
\end{eqnarray}
\end{subequations}
where the constant $C>0$ may depend on $\Delta t$ and the Lipschitz constants of $f$ and $\ell$.

Now, we define the  state-constrained optimal control problem
\begin{equation}
\label{422}
    V(x):=\underset{\bu=(u_k)_k\in U^N}{\inf}
     \Bigg\{ \sum_{k=0}^{N-1} L_k(\yxuk,u_k) + \varphi(\yxuN)   \hspace{0.1cm}
     \mid \hspace{0.1cm} g(\yxuk)\leq 0
     \hspace{0.2cm} \forall k=0,...,N, \hspace{0.2cm} g_f(y^{x,\bu}_N)\leq 0 \Bigg\}.
\end{equation}
Then, for the discrete auxiliary control problem,
we define the cost functional $J$ by 
\begin{equation}
\label{424}
    J(x,z,\bu) := \Big( \sum_{k=0}^{N-1} L_k(\yxuk,u_k) + \varphi(\yxuN) -z \Big)  \bigvee
    \Big(\underset{0\leq k\leq N}{\max}\hspace{0.1cm} g(\yxuk)\Big)\bigvee g_f(y^{x,\bu}_N)
\end{equation}
(for $(x,z,\bu=(u_k))\in \R^d\times\R\times U^N$)
and the corresponding auxiliary value is defined, for $(x,z)\in \R^d\times\R$, by
\begin{equation}\label{eq:pbaux}
    W(x,z) := \underset{\bu=(u_k)_k\in U^N}{\inf} J(x,z,\bu).
\end{equation}
Notice  that $W(x,z)$ converges to $\WW(0,x,z)$ as $N\to \infty$ (i.e., $\Delta t=\frac{T}{N} \to 0$), where the continuous value function $\WW$ is defined in \eqref{eq:wg}.
Under Assumptions \Hyp{0}-\Hyp{4},  the error estimate of $|W(x,z)-\WW(0,x,z)|$ is 
   bounded by $O(\frac{1}{N})$,  see \cite[Appendix B]{Bok-Gam-Zid-2022}. 
  Furthermore,    the sequence of discrete-time optimal trajectories  (for $N\in \N$) provide convergent 
  approximations of optimal trajectories of the continuous problem \eqref{eq:pbaux}, see \cite{ref5}.

\begin{rem}
The discrete dynamics $F$ and the discrete cost $L$  can be defined as  approximations  of the time-continuous function $f$ and $\ell$. 
It is worth mentioning that the algorithms that will be presented in this section can also handle situations where the dynamics $F$ and $L$ are obtained by some statistical models which can be enriched during the computational process. 
\end{rem}

With similar arguments as in the proof of  \eqref{eq:vsharp-wg}, we have
\begin{equation}\label{eq.V-W}
V(x) =\inf\{z \mid W(x,z)\leq 0\}.
\end{equation}
For the sake of simplicity and without loss of generality, we suppose that the control is of dimension $r=1$ and we denote by $D$ its maximal diameter ($\forall a,a'\in A, \|a-a'\|\leq D$), although the approach can be generalized to control variables in multiple dimensions. 

Planning algorithms are based on the principles of optimistic optimization. In order to minimize the objective function $J$ over the space $U^N$, we refine, in an iterative way the search space into smaller subsets. 
A search space, called node and denoted by $\U_i$ with $i\in \mathbb{N}$, is a Cartesian product of sub-intervals of $U$, i.e., $\U_i:=U_{i,0}\times U_{i,1}\times\cdots\times U_{i,N-1} \subseteq U^N$, where $U_{i,k}$  represents the control interval at time step $k$, for $k=0,\ldots,N-1$. 
The collection of nodes will be organized into a tree $\Upsilon$ that will be constructed progressively by expanding the tree nodes. Expanding a node $\U_i$, with $i\in \mathbb{N}$, consists in choosing an interval $U_{i,k}$, for $k=0,\ldots,N-1$, and splitting it uniformly to $M$ sub-intervals where $M>1$ is a parameter of the algorithm. 
The order of expanded nodes and the intervals that have to be split will be chosen in such a way to minimize the cost $J$. 
For now, we introduce some useful notations related to the tree $\Upsilon$:
%%%
\begin{itemize}[align=left, leftmargin=*, itemindent=0mm, nolistsep]
    \item We associate, for any node $\U_i \in \Upsilon$, a sample sequence of controls $\bu_i:=(u_{i,k})_{k=0}^{N-1}\in \U_i$ such that $u_{i,k}$ corresponds to the midpoint of the interval $U_{i,k}$ for any $k=0,\ldots,N-1$.
    \item Denote $d_{i,k}$, for $k=0,...,N-1$, the diameter of the interval $U_{i,k}$ of some node $\U_i\in \Upsilon$. 
    In particular, 
   $$ d_{i,k} = \frac{D}{M^{s_i(k)}}, $$
   where $s_i(k)$ indicates the number of splits needed to obtain the interval $U_{i,k}$ for $k=0,\ldots,N-1$.
    %12mm}
    \item The depth $p_i$ of a node $\U_i$ is  the total number of splits done to obtain this node:
    \begin{equation}
    \label{444}
    p_i:=\sum_{k=0}^{N-1}s_i(k).   
    \end{equation}
    We denote by $\texttt{Depth}(\Upsilon)$ the maximal depth in the tree $\Upsilon$.
    \item A node $\U_i$ is a tree leaf if it has not been expanded. The set of tree leaves is denoted by $\Lambda$. 
    \item Finally, we denote  by $\Lambda_p:=\Big\{\U_i \in \Upsilon \quad s.t. \quad p_i=p\Big\}$ the set of leaves of $\Upsilon$ of depth $p\in \mathbb{N}$.
\end{itemize}

By selecting controls at the intervals centers and by taking $M$ odd, we guarantee that after expanding a node $\U_i$ we generate at least one node $\U_j$ with  $J(x,z,\bu_j) \leq J(x,z,\bu_i).$ Indeed, the middle child $\U_j$
contains the control sequence of $\U_i$.
\begin{prop}
\label{prop441}
By the tree construction, there exists at least a leaf node $\U_i\in \Lambda$  containing an optimal control sequence and satisfying
\begin{equation}
    \label{441}
    J(x,z,\bu_i) -\sigma_{i}  \leq W(x,z) \leq J(x,z,\bu_i),
 \end{equation}
 where $\bu_i$ is the sample control sequence in  $\U_i$ and where
  \begin{equation}
 \label{442}
     \sigma_{i}:= \Big(\sum_{k=0}^{N-1}\beta_k d_{i,k} \Big) \bigvee  \Big(\sum_{k=0}^{N-1} \gamma_k d_{i,k} \Big),
 \end{equation}
with $\beta_k$ and $\gamma_k$ positive constants only depending on the Lipschitz constants of $F, L, \Phi$ and of $\Psi$.
\end{prop}
In the optimistic planning algorithms,  at each iteration, one or several  optimistic nodes are chosen and split to get  from each node $M$ children ($M > 1$ is a fixed parameter of the algorithm). 
To expand a node $\U_i$, we choose an interval from $U_{i,0}\times U_{i,1}\times\cdots\times U_{i,N-1}$ and we partition it uniformly to $M$ sub-intervals. If we choose to split the interval $U_{i,k}$, for some $k=0,\ldots,N-1$, then $M$ nodes  will be generated and then the new error term $\sigma^+_i(k)$ is 
\begin{equation*}
    \sigma^+_i(k):= \Big(\sum_{j=0, j\neq k}^{N-1} \beta_j d_{i,j} + \beta_k \frac{d_{i,k}}{M}\Big)\bigvee \Big(\sum_{j=0, j\neq k }^{N-1} \gamma_j d_{i,j} + \gamma_k \frac{d_{i,k}}{M}\Big).
\end{equation*}
\\[-4mm]
Henceforth, in order to minimize the error $\sigma^+_i(k)$, the best choice of the interval to split, $k^*_i$, is given by:
%3mm}
\begin{equation}
    \label{443}
    k^*_i \in \underset{0\leq k \leq N-1 }{\text{argmin}}  \hspace{0.1cm} \sigma^+_i(k).
\end{equation}

The following result gives 
an upper bound on the error term $\sigma_i$, of any node $\U_i\in \Upsilon$.

\begin{prop}
\label{th441}
Assume that the number of split $M>L_{F,x}>1$. 
%%%Let  $\tau:=\Bigl\lceil \frac{\log M}{\log(L_{F,x})} \Bigr\rceil$.
Consider a node $\U_i$ at some depth $p_i=p$. For $p$ large enough, the error $\sigma_i$ (defined in \eqref{442}) is bounded as follows:  
%3mm}
\begin{equation}
    \label{eq:445}
    \sigma_{i}\leq  \delta_p:=\ c_1(N)\,\Delta t\, M^{-\frac{p}{N}},
\end{equation}
where $c_1(N)>0$ is bounded independently of $N$.
\end{prop}
We refer to \cite{Bok-Gam-Zid-2022} for the proof of this result and for the exact expression of the constant
$c_1(N)$. 

% \if{For a given depth $p\in \mathbb{N}$, we define the set of nodes that will be expanded by optimistic planning algorithms:
% $$ 
%  \Upsilon^*_p:=\{\U_i \in \Upsilon \hspace{0.2cm} \mbox{at depth}\hspace{0.2cm} p \hspace{0.2cm} | \hspace{0.2cm} J(x,z,u_i)-\delta_p\leq W(x,z)\},
% $$
% where $\delta_p$ is defined as in~\eqref{eq:445}. The set containing all expanded nodes :
% $  \Upsilon^*:= \bigcup_{q\geq 0} \Upsilon^*_q
% $
% is in general smaller than the whole tree $\Upsilon$. The next definition introduces a notion of {\em branching factor} that is a sort of complexity of the algorithm. The precise definition of this notion is given hereafter.
% \begin{definition}
% \label{def441}
% For a given depth $p\in \mathbb{N}$, we define the asymptotic branching factor  $m$, as the smallest real $m\in [1,M]$ such that 
% %6mm}
% \be
%   \exists R\geq 1,\ \forall p\geq 0,\ |\Upsilon^*_p|\leq R m^p.
%   \label{eq:m}
% \ee
% \end{definition}
% }\fi 
%
Now, we will present the rules for refining the search of an optimal control strategy. In the first algorithm, at each iteration,  the node $\U_{i^*}$ minimizing the lower bound ($J(x,z,u_i)-\sigma_i$) will be selected and split to $M$ children. More precisely, we identify an interval  $U_{i^*,k^*_{i^*}}$ whose partition in $M$ sub-intervals will produce the lowest error  $\sigma_{i^*}(k^*_{i^*})$.

\begin{algorithm*}[!hbtp]
\caption{{\bf Optimistic Planning (OP)}
  \label{alg:OP}
}
\begin{algorithmic}[1]
\REQUIRE 
 The number of intervals $N$, the split factor $M$, the maximal number of expanded nodes $I_{\max}$
\STATE 
 Initialize $\Upsilon$ with a root $\U_0:=U^N$ and $n=0$ ($n:=$ number of expanded nodes).
\WHILE{$n < I_{\max}$}
\STATE 
 Select an optimistic node to expand: $\U_{i^*}\in \underset{\U_i\in \Lambda}{\text{argmin}} \hspace{0.1cm} (J(x,z,\bu_i)-\sigma_i)$. 
\STATE 
Select $k^*_{i^*}$, defined in \eqref{443}, the interval to split for the node $\U_{i^*}$.
\STATE 
Update $\Upsilon$ by expanding $\U_{i^*}$ along $k^*_{i^*}$ and adding its $M$ children.
\STATE 
Update $n=n+1$.
\ENDWHILE
\RETURN 
Control sequence $\bu_{i^*}=(u_{i^*,k})_k\in  U^N$ of the node $\U_{i^*} \in \underset{\U_i\in \Lambda}{\text{argmin}}  \hspace{0.1cm}  J(x,z,\bu_i)$.
\end{algorithmic}
\end{algorithm*}

\begin{thm}
\label{th442}
Assume that $M>L_{F,x}>1$.
Let $\bu_{i^*}$ and $J(x,z,\bu_{i^*})$ be the output of the OP\ algorithm, 
and let $n\geq 1$ be the corresponding number of expanded nodes.
We have
  \begin{equation} \label{eq:err-bound-OP}
  0 \leq J(x,z,\bu_{i^*})-W(x,z) 
\longrightarrow 0, \mbox{ when } n\to+\infty.  
\end{equation}
%%%
\end{thm}
In Algorithm (OP), the number $I_{\max}$ represents a maximal available computational resource. The number of expanded nodes corresponds to the number of iterations,
since at each iteration only one node is expanded. 
Other  optimistic planning methods can be considered.
For instance, the {\em simultaneous optimistic planning} (SOP)  algorithm or {\em simultaneous optimistic planning with multiple steps} (SOPMS) algorithm that 
expand at each iteration several  nodes at every iteration, see \cite{Bok-Gam-Zid-2022}.
%%%

\paragraph{Test 5.} To show the relevance of this approach, we consider a variant of Zermelo problem where a boat targets the set $\mathcal{C}:=\mathbb{B}(0;0.1)$, 
at time $T=1$,  
with minimal fuel consumption. The dynamics is similar to the one considered in \eqref{eq:dyn_zermelo_Unconstrained}.
We consider also two rectangular obstacles with  horizontal and vertical half lengths 
$(r_x,r_y)$. The  first obstacle is centered at $(-2.0,0.5)$   with 
$(r_x,r_y)=(0.4,0.4)$,  and the second obstacle is centered at  $(-2.5,-1)$  with 
        $(r_x,r_y)=(0.2,1)$.
To take into account   the pointwise and final state constraints, we define the functions $g$ and $g_f$ by
\begin{equation*}
  g(x):= \Big( 0.4 - \|x-(-2,0.5)\|_{\infty} \Big) \bigvee \min\big( 0.2 - |x_1 +2.5|, 1 - |x_2+1|\big) 
   \quad \mbox{and} \quad g_f(x) := \|x\|_\infty - 0.1. 
\end{equation*}
For a given $N$,  the discrete control problem  becomes:
\begin{equation*}
  \begin{array}{l}
  V(x)={\inf} \Big\{ 
      \frac1N\sum_{k=0}^{N-1} u_{1,k}     \quad \mbox{with} \quad \bu=((u_{1,k},u_{2,k}))_k \in U^N, \\
      \hspace*{3cm}
       g(y^{x,\bu}_k)\leq 0\ \ \ \mbox{for } k=0,\cdots,N, \quad \mbox{and } \ \  g_f(y^{x,\bu}_N)\leq 0\Big\},   
 \end{array}
\end{equation*}
where $(y^{x,\bu}_k)_k$ is the discrete state variable, corresponding to the control policy 
$\bu \in U^N$,  and starting at the initial position $x$, while the discrete auxiliary value function is defined as 
\begin{equation*}
  W(x,z) = \underset{\bu \in U^N}{\inf} \Big\{ \Big(\frac1N\sum_{k=0}^{N-1}u_{1,k}  -z \Big)  \bigvee
    \Big(\underset{0\leq k\leq N}{\max}\hspace{0.1cm} g(y^{x,\bu}_k)\Big)\bigvee g_f(y^{x,\bu}_N) \Big\}.
\end{equation*}

Figure~\ref{fig:zermelo_unconstrained_OP_traj} displays optimal  trajectories obtained  from three different initial positions.
A simultaneous optimistic planning algorithm is used for this simulation with $I_{\max}=3200$ and $N=40$. The optimal controls are displayed on Figure~\ref{fig:zermelo_unconstrained_OP_Co}.
\begin{figure}[!htbp]
    \centering
    \includegraphics[width=0.7\textwidth]{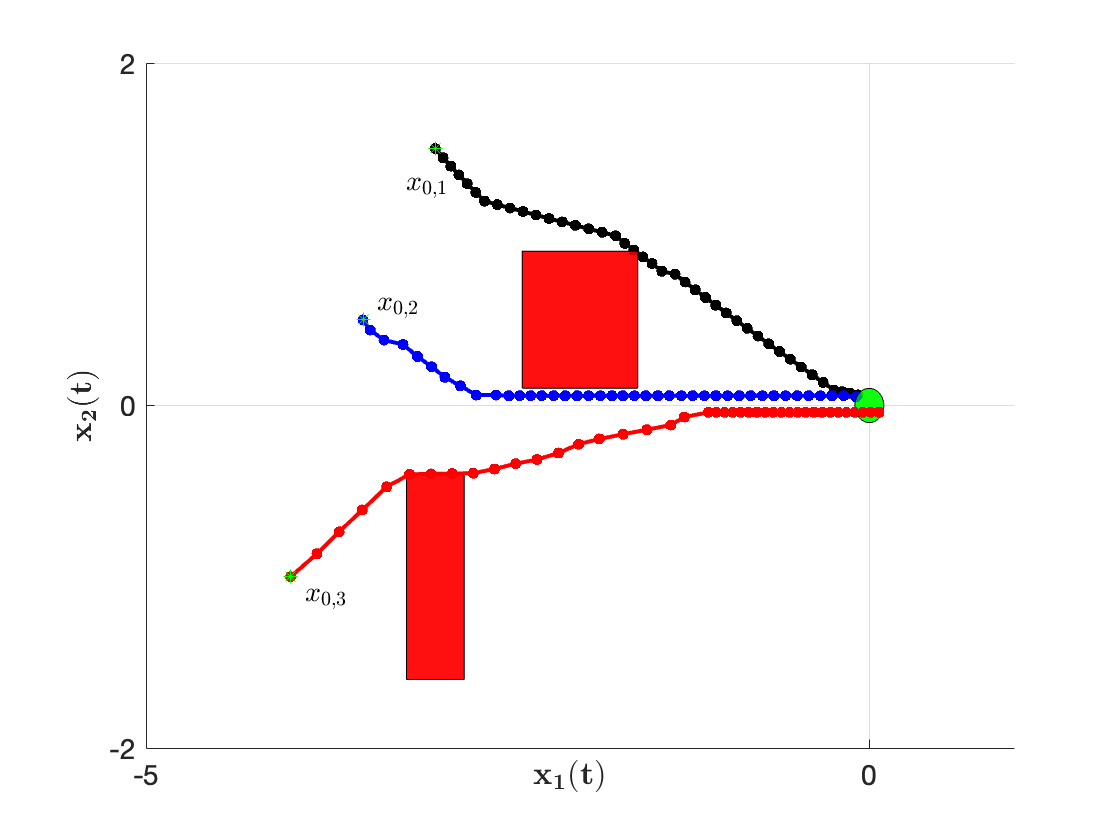} 
     \caption{(Test 5) An optimistic planning approach with  $N=40$. Optimal trajectories corresponding to three different initial data. The trajectory in black corresponds to the initial position $x_{0,1}=(-3,1.5)$, the trajectory in blue corresponds to $x_{0,2}=(-3.5,0.5)$, and the trajectory in red corresponds to $x_{0,3}=(-4,-1)$ }\label{fig:zermelo_unconstrained_OP_traj}
\end{figure}
\begin{figure}[!hbpt]
    \centering
    \includegraphics[width=0.7\textwidth]{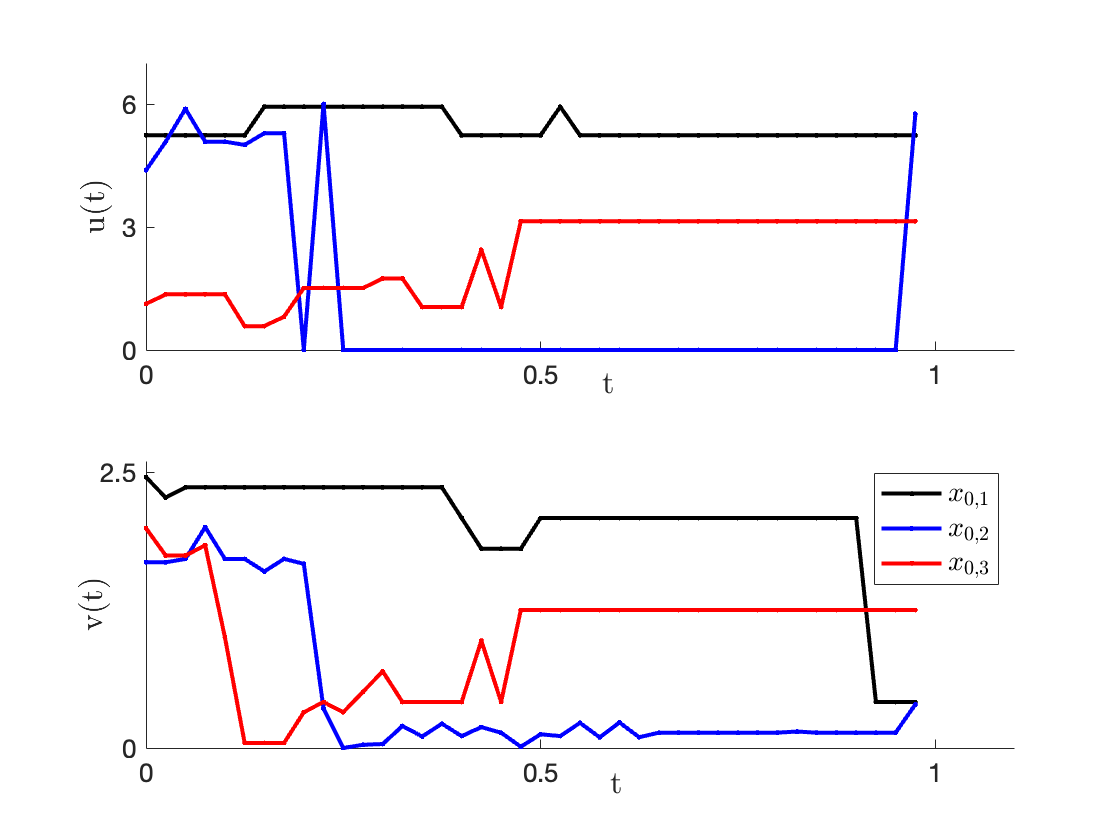}
     \caption{(Test 5)  An optimistic planning approach with  $N=40$. Control laws corresponding to three different initial positions. The control law in black corresponds to the initial position $x_{0,1}=(-3,1.5)$, the control laws in blue corresponds to $x_{0,2}=(-3.5,0.5)$, and the control laws in red corresponds to $x_{0,3}=(-4,-1)$ }\label{fig:zermelo_unconstrained_OP_Co}
\end{figure}
In this example, we can see that the  optimal trajectories, computed by an OP approach, reach the target and avoid the obstacles. The main feature of the OP approaches is the fact that they give an approximation of the global solution. For a fixed value of $N$, the complexity of these (global) approaches depends on the dimension of the control and not on the dimension of the space variable. The complexity increases also when the discretization is refined (i.e., when $N$ increases).  As pointed out in the literature \cite{ref51,ref52,Bok-Gam-Zid-2022}, the accuracy of the numerical solutions, obtained by OP methods, depends on the available numerical resources $I_{\max}$. The convergence results derived in the literature give some hints on how to choose $I_{\max}$ to obtain a given precision, but this question deserves further analysis.